\documentclass[12pt ]{article}

\usepackage{fullpage}

\usepackage{tipa}
\usepackage{dsfont}
\usepackage{tikz-cd}

\usepackage{graphicx}
\usepackage{wrapfig}
 
\usepackage{xr}

\usepackage[all,2cell]{xy}
\UseTwocells

\usepackage{amsxtra}
\usepackage{amsmath}
\usepackage{amssymb}
\usepackage{amsfonts}
\usepackage[all,2cell]{xy}
\usepackage{mathrsfs}
\usepackage{amsthm}
\usepackage{enumitem}
\usepackage{hyperref}
\usepackage{subfigure}
\usepackage{mathabx}
\usepackage{euscript}
\usepackage[bbgreekl]{mathbbol}

\usepackage{xargs}                      

\usepackage[colorinlistoftodos,prependcaption,textsize=tiny]{todonotes}
\newcommandx{\note}[2][1=]{\todo[linecolor=Plum,backgroundcolor=Plum!25,bordercolor=Plum,#1]{#2}}

\usepackage{tikz, tikz-3dplot, pgfplots}
\usetikzlibrary{decorations.markings, patterns, decorations.pathmorphing}

\tikzset{->-/.style={decoration={
markings,
mark=at position #1 with {\arrow{>}}},postaction={decorate}}}

\newtheoremstyle{example}{\topsep}{\topsep}%
     {}
     {}
     {\bfseries}
     {.}
     {2pt}
     {\thmname{#1}\thmnumber{ #2}\thmnote{ #3}}

\theoremstyle{example}
\newtheorem{exa}[equation]{Example}
\newtheorem{exas}[equation]{Examples}
\newtheorem{ex}[equation]{Example}
\newtheorem{rem}[equation]{Remark}
\newtheorem{rems}[equation]{Remarks}
\newtheorem{defi}[equation]{Definition}

\newtheorem{thm}[equation]{Theorem}
 \newtheorem{cor}[equation]{Corollary}
\newtheorem{lem}[equation]{Lemma}
\newtheorem{prop}[equation]{Proposition}

\setlist[enumerate,1]{label=(\arabic{*})}
\setlist[enumerate,2]{label=(\roman{*})}
\setlist[enumerate,3]{label=(\alph{*})}

\def\Ac{\mathcal{A}}

\def\Ec{\mathcal{E}}
\def\Fc{\mathcal{F}}

\def\Mc{\mathcal{M}}

\def\Oc{\mathcal{O}}

\newcommand{\arr}[1]{ {\{ #1 \} }}

\def\<<{\langle {}\hskip -.1cm {}\langle}
\def\>>{\rangle \hskip -.1cm \rangle}
\def\={{\, \simeq\, }}
\def\-{{\,\setminus\,}}

 \def\2{{\mathbf{2}}}

\def\A{ {\EuScript A}}

\def\Ac{\mathcal {A}}

\def\Att{{\tt{A}}}

\def\B{ {\EuScript B}}
\def\be{\begin{equation}}

\def\C{{\EuScript C}}

\def\CC{\mathbb{C}}

\def\Cat{{\mathcal Cat}}
\def\Cat{{\EuScript Cat}}
\def\CCat{{\mathbb{Cat}}}

 \def\Cof{{\on{Cof}}}
 
 \def\Coh{{\on{Coh}}}

\def\Cone{\operatorname{Cone}}

\def\Cot{{\on{Cot}}}

\def\Di{{\EuScript D}}
\def\D{\Di}

\def\DD{{\mathbb{D}}}

\def\del{{\partial}}
\def\dep{{\on{dep}}}
\def\Dtt{{\tt{D}}}

 \def\DK{{\on{DK}}}

\def\Ec{\mathcal{E}}
\def\ee{\end{equation}}
\def\EE{{\mathbb{E}}}

\def\eps{{\varepsilon}}
\def\ev{\on{ev}}
\def\Ex{\on{Ex}}
\def\ExFun{{\on{ExFun}}}

\def\Fc{\mathcal{F}}

\def\Fib{ \operatorname{Fib}}
\def\fib{\operatorname{fib}}

\def\Fun{\operatorname{Fun}}

\def\h{\operatorname{h}\!}

\def\Hom{{\operatorname{Hom}}}

\def\hra{\hookrightarrow}

\def\id{\on{id}}
\def\Id{\on{Id}}

\def\k{{\mathbf k}}
\def\KK{{\mathbb K}}

\def\lan{{\,\langle }}
\def\laxlim{{\on{lax-}\varprojlim}}
\def\laxcolim{{\on{lax-}\varinjlim}}

\def\lra{\longrightarrow}

\def\LMat{{\on{LMat}}}

\def\lperp{{^\perp\!}}

\def\lst{{ ^* \!}}

\def\M{{\EuScript M}}
\def\Map{\operatorname{Map}}
\def\Mat{\operatorname{Mat}}
\def\Mc{\mathcal{M}}

\def\N{\operatorname{N}}

\def\Ob{ \operatorname{Ob}}

\def\Oc{{\mathcal O}}
\def\ol{\overline}
\def\on{\operatorname}

\def\oo{{\infty}}
\def\op{{\operatorname{op}}}
\def\oplaxlim{{\on{oplax-}\varprojlim}}
\def\oplaxcolim{{\on{oplax-}\varinjlim}}

\def\phi{{\varphi}}

\def\PP{{\mathbb{P}}}

\def\ran{{\rangle}}

\def\RR{\mathbb{R}}

\def\Set{ {\operatorname{Set}}}

\def\Sp{{\on{Sp}}}

\def\Spec{{\on{Spec}}}

\def\SSph{{\mathbb {Sph}}}

\def\St{{\EuScript {S}t}}
\def\SSt{{\mathbb{St}}}

\def\Tot{\on{Tot}}

\def\ul{\underline}

\def\Vect{ {\on{Vect}}}

\def\wt{\widetilde}

\def\X{{\EuScript X}}

\def\Y{{\EuScript Y}}

\def\ZZ{\mathbb{Z}}

\DeclareMathSymbol\DDelta\mathord{bbold}{"01}
\DeclareMathSymbol\GGamma\mathord{bbold}{"00}
\DeclareMathSymbol\SSigma\mathord{bbold}{'117}
\DeclareMathSymbol\LLambda\mathord{bbold}{'003}

\newcommand\laxcone[7][]{
  \def\temptwocell{#1}
  \def\tempzero{#2}
  \def\tempone{#3}
  \def\temptwo{#4}
  \def\tempzeroone{#5}
  \def\temponetwo{#6}
  \def\tempzerotwo{#7}
  \begin{tikzcd}[column sep=small,ampersand replacement=\&]
    \& \ar[dr,"\temponetwo"{name=U}]\ar[dl,"\tempzeroone"'] \tempone   \& \\
    \tempzero \ar[rr, swap, "\tempzerotwo"]\ar[to=U,"\temptwocell", shorten <=20pt, shorten >=10pt, Rightarrow] \& \& \temptwo 
  \end{tikzcd}
}
\newcommand\laxcocone[7][]{
  \def\temptwocell{#1}
  \def\tempzero{#2}
  \def\tempone{#3}
  \def\temptwo{#4}
  \def\tempzeroone{#5}
  \def\temponetwo{#6}
  \def\tempzerotwo{#7}
	\begin{tikzcd}[column sep=small,ampersand replacement =\&]
    \& \tempone   \&    \\
    \tempzero \ar[ur, "\tempzeroone"{name=U}]\ar[rr, swap, "\tempzerotwo"]
    \& \& \ar[ul,"\temponetwo"'] \temptwo
    \ar[to=U,"\temptwocell",shorten <=20pt, shorten >=10pt, Rightarrow]
	\end{tikzcd}
}

\newcommand\oplaxcone[7][]{
  \def\temptwocell{#1}
  \def\tempzero{#2}
  \def\tempone{#3}
  \def\temptwo{#4}
  \def\tempzeroone{#5}
  \def\temponetwo{#6}
  \def\tempzerotwo{#7}
  \begin{tikzcd}[column sep=small, ampersand replacement=\&]
		\& \ar[dr,"\temponetwo"{name=U}]\ar[dl,"\tempzeroone"'] \tempone   \& \\
    \tempzero \ar[rr, swap, "\tempzerotwo"]
    \ar[to=U,"\temptwocell",shorten <=20pt, shorten >=10pt, Leftarrow] \& \& \temptwo
	\end{tikzcd}
}
\newcommand\oplaxcocone[7][]{
  \def\temptwocell{#1}
  \def\tempzero{#2}
  \def\tempone{#3}
  \def\temptwo{#4}
  \def\tempzeroone{#5}
  \def\temponetwo{#6}
  \def\tempzerotwo{#7}
  \begin{tikzcd}[column sep=small, ampersand replacement=\&]
		\& \tempone   \& \\
    \tempzero \ar[rr, swap, "\tempzerotwo"]\ar[ur,"\temponetwo"{name=U}]
     \& \& \temptwo \ar[ul,"\tempzeroone"'] \ar[to=U,"\temptwocell",shorten <=20pt, shorten >=10pt, Leftarrow]
	\end{tikzcd}
}

\title{N-spherical functors and categorification of Euler's continuants}
\author{Tobias Dyckerhoff, Mikhail Kapranov, Vadim Schechtman }
 
\begin{document}

\maketitle

\begin{abstract}
Euler's continuants are universal polynomials expressing  the numerator and denominator
of a finite continued fraction whose entries are independent variables. We introduce their
categorical lifts which are natural complexes (more precisely, coherently commutative cubes)
of functors involving compositions of a given functor and its adjoints of various orders, with
the differentials built out of units and counits of the adjunctions. In the stable $\oo$-categorical
context these complexes/cubes can be assigned totalizations which are new functors
serving as higher analogs of the spherical twist and cotwist. We define $N$-spherical functors
by vanishing of the twist and cotwist of order $N-1$ in which case those of order $N-2$
are equivalences. The usual concept of a spherical functor corresponds to $N=4$. We characterize
$N$-periodic semi-orthogonal decompositions of triangulated (stable $\oo$-)
categories in terms of $N$-sphericity of their gluing functors. The procedure of forming iterated
orthogonals turns out to be analogous to the
procedure of forming a continued fraction. 
\end{abstract}

\tableofcontents

\addcontentsline{toc}{section}{Introduction}


\numberwithin{equation}{section}

\section*{Introduction}

The goal of this paper is to connect the classical formalism of continued fractions with
the modern theories of spherical functors and semi-orthogonal decompositions
for triangulated and stable $\oo$-categories. 

\paragraph{Classical continuants.} Euler's continuants are universal polynomials 
$E_N(x_1,.\cdots, x_N)$ expressing the numerator and denominator
of a finite continued fraction whose entries are independent variables \cite[\S 3-4]{perron}
\cite[\S 1.2]{khrushchev}. For example (we use a signed version)
\be\label{eq:E2&E3}
E_2(x_1, x_2) = x_1 x_2-1, \quad E_3(x_1, x_2, x_3) = x_1x_2x_3-x_1-x_3
\ee
 so that
\be\label{eq:CF-3-term}
  x_1 - \cfrac{1}{x_2  -\cfrac{1} {x_3}} \,=\, \frac{ x_1x_2x_3 - x_1-x_3} {x_2x_3-1} \, = \,
  \frac{E_3(x_1, x_2, x_3)}{E_2(x_2, x_3)}. 
  \ee

They were first written down  (in a version without signs) in a 1764 paper \cite{euler} by Euler and have been studied
classically as a chapter of the theory of determinants,   entire books having been devoted to them
   \cite{borini}. 
More recently, 
continuants attracted considerable  attention in the theory of moduli spaces
of irregular connections   \cite{boalch, fairon},
  where they appear as certain group-valued moment maps.  

\paragraph{Categorical lift: continuant complexes of  iterated adjoints.  } 

In this paper we present another unexpected appearance of continuants: in modern category theory. 
More precisely, we work in the context of stable $\oo$-categories of Lurie  \cite{lurie:ha},
which provide a robust ``enhancement" of the classical concept of triangulated categories. 
But the underlying idea  can be explained even at the level of ordinary (additive) categories. 
Recall that a pair of functors   $
  \xymatrix{
  \A_0 \ar@<.4ex>[r]^F&\A_1\ar@<.4ex>[l]^G
  }
$
is called an adjoint pair (and we write $G=F^*$), if we are given  unit and counit transformations
\[
 e: \Id_{\A_0} \lra GF,\quad c: FG \lra \Id_{\A_1},
\]
satisfying the standard properties \cite{maclane}. If we now have a {\em string of $N$ adjoints}, i.e.,
a sequence of functors $F_1: \A_0\to \A_1, F_2: \A_1\to \A_0, \cdots, F_N$ in alternating directions
such that $F_{i+1}= F_i^*$, we can  form a diagram of functors  $\Ec_N(F_1,  \cdots, F_N)$ 
starting with the composition  $F_1\cdots F_N$ and successively applying the various counits. Dually, we have
a diagram $\Ec^N(F_N,\cdots, F_1)$ formed by applying units, terminating in $F_N\cdots F_1$.
 For $\A_0, \A_1$ additive we can totalize these diagrams into complexes of functors by passing to 
appropriate alternating sums. They can be seen as categorical liftings of the continuant $E_N$ as
their terms match the monomials in $E_N$. 
For example, $\Ec_2(F, F^*)$ is just the counit arrow,  $\Ec^2(F^*, F)$ is the unit arrow while
\[
\Ec_3(F, F^*, F^{**}) = \bigl\{ F F^* F^{**}  \lra  F\oplus F^{**}\bigr\}, \quad 
\Ec^3(F^{**}, F^*, F) = \bigl\{ F^{**} \oplus F \lra F^{**} F^* F\bigr\},
\]
in complete match with \eqref{eq:E2&E3}. See  further Examples \ref{exas:fibcubes}.

\paragraph{ Higher (co)twists and $N$-spherical functors.}  The continuant diagrams
  lead to a new point of view on and a ``higher'' generalization of the concept of spherical functors
  so important in categorical mirror symmetry and categorified representation theory.

In the main body of the paper
$F = F_1: \A_0\to\A_1$ is an exact functor of stable $\oo$-categories. Assuming the iterated adjonts
$F_2,\cdots, F_N$ exist, we denote by $\Ec_N(F)$, $\Ec^N(F)$ the corresponding
continuant diagrams. By inserting some $0$'s, they can be seen as homotopy commutative $(N-1)$-cubes of functors,
see \S \ref{subsec:fibcubes},  \ref{subsec:cont-cubes}, and so we can form their {\em totalizations},
see \S  \ref{subsec:gen-oo-cat},
which are single functors $\EE_N(F) = \Tot(\Ec_N(F))$, $\EE^N(F)=\Tot(\Ec^N(F))$. 
For example,
\[
\EE_2(F) = \Cone \bigl\{ c: F F^* \to \Id\bigr\}[-1], \quad \EE^2(F) = \Cone \bigl\{e: \Id\to F^* F\bigr\}
\]
are the  spherical twist and cotwist functors associated with $F$, and $F$ is called
spherical, if they are equivalences \cite{AL17, dkss-spher}. 

For any $N$ we call $\EE_N(F)$ and $\EE^N(F)$ the {\em $N$th twist and cotwist functors associated to $F$}
and say that $F$ is {\em $N$-spherical}, if $\EE_{N-1}(F) \= \EE^{N-1}(F)=0$. In this case
$\EE_{N-2}(F)$ and $E^{N-2}(F)$ are equivalences. 

Usual spherical functors are, in this
terminology, $4$-spherical. The characterization of usual spherical functors by vanishing of $\EE_3$ and
$\EE^3$ has been given by Kuznetsov \cite{kuznetsov} in 2015 without, however, the connection
with continuants having been recognized. 

In \S \ref{sec:ex} we provide several examples of $N$-spherical functors for arbitrary $N$.

\paragraph{$N$-spherical functors and $N$-periodic semi-orthogonal decompositions.} 
 An exact $\oo$-functor $F: \A_0\to\A_1$  as above can be used as a {\em gluing functor}, see 
 \cite{kuznetsov-lunts, dkss-spher, CDW} to construct a new stable $\oo$-category $\A=\A_0\oplus_F \A_1$
 in which $\A$ and $\B$ are embedded as full $\oo$-subcategories forming a  ($2$-term)
 {\em semi-orthogonal decomposition (SOD)} $(\A_0, \A_1)$ of  $\A$. This means that 
 $\Map_\A(a_1, a_0)$ is contractible for any objects $a_1\in\A_1$, $a_0\in A_0$ and each object $a$ of $\A$ 
  fits into a triangle
 \[
 a_1\lra a \lra a_0, \quad a_1\in\A_1, \, a_0\in \A_0.
 \]
The concept of SOD, going back to \cite{BK:SOD} for classical triangulated categories, has been 
re-cast in the stable $\oo$-categorical framework in \cite{dkss-spher}; the construction
of $\A_0\oplus_F \A_1$ was further interpreted
in \cite{CDW} as giving  a kind of lax additivity property of   the $(\oo,2)$-category of all stable $\oo$-categories. 

For an SOD $(\A_0, \A_1)$ of $\A$ we have that  $\A_1 = {^\perp \A}_0$ is the {\em left orthogonal} of $\A_0$, i.e.,
the full subcategory formed by objects $a_1$ such that for any $a_0\in\A_0$ the space $\Map_\A (a_1,a_0)$
is contractible. One can form further iterated orthogonals ${^{\perp\perp}} \A_0 = ^\perp ({^\perp} \A_0)$, 
${^{\perp\perp\perp}\A}_0$ etc.
and ask when the chain of  orthogonals in $N$-periodic, i.e., when the $N$-fold iterated
left orthogonal $^{\perp N} \A_0$ coincides with $\A_0$. One of our main results,
Theorem \ref{thm:N-spher=per} says that this happens if and
only if $F$ is $N$-spherical. For $N=4$ it is a result of Halpern-Leistner and Shipman
\cite {halpern-shipman}: a functor is spherical in the usual sense if and only if the chain of orthogonals
in $\A_0\oplus_F \A_1$ is $4$-periodic. This explains our terminology. 

\paragraph{Categorified continuants via lax $2\times 2$ matrices.} 
Passing to orthogonals corresponds to {\em mutations } of semi-orthogonal decompositions, see 
\cite{BK:SOD, dkss-spher}.
Our results suggest the following slogan:
\[
\begin{gathered}
\text {\em The procedure of forming iterated mutations of 2-term semi-orthogonal decompositions}
\\
 \text{\em  is
a categorical counterpart of the procedure of forming  continued fractions.
}
\end{gathered} 
\]
This can be explained as follows. 
It is classical, see, e.g.,  \cite[\S 5]{perron},  that continuants and continued fractions
 (with sign conventions as in \eqref{eq:CF-3-term})
are related to products of $2\times 2$ matrices of the form
\be\label{eq:CF-matrix}
\begin{pmatrix}
0&-1
\\
1& x_i
\end{pmatrix}. 
\ee
Now, an SOD $(\A_0, \A_1)$ of a stable $\oo$-category $\A= \A_0\oplus_F \A_1$ 
 can be seen as a kind of direct
sum decomposition of $\A$. In \cite{CDW} a theory of {\em lax matrices} was developed which allows one to describe
 an exact $\oo$-functor 
 $\phi: \A_0\oplus_F \A_1 \to \B_0\oplus_G \B_1$ in terms of a $2\times 2$ matrix $\|\phi_{ij}\|$  consisting  of functors
 $\phi_{ij}: \A_i\to \B_j$ together with additional data of certain natural transformations, see \S \ref  {subsec:laxmat}
 below. These data further  allow us to
 multiply such matrices by lifting the classical formula for each matrix element of the product  as the sum of two summands
 to the cone (cofiber) of a natural morphism between such summands, see Proposition \ref{prop:comp-2-laxmat}.
 If the gluing functor $F: \A_0\to \A_1$ has a right adjoint $F^*$, then we have
the  mutated SOD  $(\A_1, \A_2 = {^{\perp\perp} \A}_0)$ in which $\A_2$ is identified with $\A_0$ by the
mutation isomorphism; after this identification the gluing functor for the new SOD is $F^*$, so we have
an equivalence $\A_0\oplus_F \A_1 \to \A_1\oplus_{F^*} \A_0$ via the identification of both with $\A$. 
The   lax matrix of this equivalence is:
\be\label{eq:CF-lax-matrix}
\begin{pmatrix}
0& \Id_{\A_0} 
\\
\Id_{\A_1} & F^*
\end{pmatrix},
\ee
the minus sign in \eqref{eq:CF-matrix} being accounted for by the shift by $1$ involved  in the definition of lax matrices. 
Here we   omit the natural transformation data, see Proposition \ref{prop:laxmat-mu} for more details. 
Multiplying several  such matrices (for iterated mutations) gives continuant complexes, thus 
 justifying the slogan above. 
 
 Remarkably, several other features of continued fractions extend to this categorical context. Thus, the classical
 formula  \eqref{eq:best-approx} for the difference between successive convergents (often expressed by saying that
 ``continued fractions give best approximations'') upgrades,  after being put in a denominator-free form,  to a distinguished
 triangle of functors, see Proposition \ref{prop:det-triang}.

 \paragraph{Relation to other work.} Our categorification of continuants suggests new
 possibilities in several areas of recent interest. 
 
 First, our constructions can be seen as providing a categorical
 lift of some aspects of the study in  \cite{boalch}  of irregular local systems on the complex line
 $\CC$ near $\oo$. They can therefore be seen as describing some 
 ``irregular perverse schobers'' on 
 $\CC$  generalizing the point of view of \cite {kapranov-schechtman:schobers},
 see \cite{kuwagaki}. This suggests that 
  Poisson aspects of continuants
   \cite{boalch, fairon} can also have  a categorical upgrade.
   
   Second, continuants appear in the recent work \cite{etingof} on the analytic Langlands 
   correspondence
   for $PGL_2$ over the real and complex fields.  Our categorical lift makes it plausible
   that some aspects of this analytic approach can be categorified as well.
   
   Finally, various features of continued fractions are a familiar presence in the cluster
  algebra  formalism, continuants being cluster variables in some important examples,
   e.g., \cite[Ex. 5.3.10]{FWZ}.
  More conceptually,  moduli spaces of Stokes data possess cluster structures.
  By now there have been already two levels of categorifying the cluster formalism:
  additive  \cite{keller, reiten} and monoidal  \cite{hernandez}. Our
  approach may provide yet another level.
  
  These and other connections are discussed in more detail in \S  \ref{sec:fur-dir}.

 \paragraph{Organization of the paper. }

 In Section \ref{sec:decat} we present a ``decategorified'' version of our constructions. 
 
 More precisely, 
 in \S \ref{subsec:bil-spaces} we consider, instead of categories, 
 finite-dimensional vector spaces over some field $\k$ with a non-degenerate bilinear form $\beta$,
 not necessarily symmetric, 
 which imitates the properties of Hom-spaces, cf. \cite{bondal-symplectic}. In this case we have the formalism of left and right
 orthogonals and one can illustrate the categorical constructions that follow later by easier linear algebra constructions.
 Thus, in \S \ref{subsec:semi-orth-decat} we describe the concepts of semi-orthogonal decompositions
 and {\em gluing operators} (instead of gluing functors) in this linear-algebraic context. 
 In \S 3 we discuss our alternating sign version $E_N(x_1, \cdots, x_N)$ of Euler continuants 
 (in the main body of the paper we call them {\em Euler polynomials}
 to avoid confusion with the  traditional terminology where the plus signs are used). 
 The variables $x_i$ do not have to be commutative
 (this was already noted in \cite{fairon}), and we need this generality to apply the $E_N$ to operators.  
 We discuss the (classical) relations of the $E_N$ with continued fractions and, in \S \ref{subsec:Fib-Cheb-poly}, with
 Fibonacci   and
 Chebyshev polynomials. Then in \S \ref{subsec:peri-orth-decat} we describe the condition of $N$-periodicity
 of linear algebra orthogonals in terms of the vanising of $E_N$ applying to the gluing operator and its adjoints. 
 
 \vskip .2cm
 
 Section \ref{sec:stable-oo} collects background material on stable $\oo$-categories, some of it very standard and some more recent. 
 
 Thus, \S \ref{subsec:gen-oo-cat} recalls   mostly standard conventions on (stable) $\oo$-categories, following \cite{lurie:htt, lurie:ha}
 and fixes the notation to be used. In particular, we fix the notation for totalizations of coherent cubes in stable $\oo$-categories,
 to be used later. 
 In \S \ref{subsec:SOD} we recall the central for us concept of semi-orthogonal decompositions (SODs)
 of stable $\oo$-categories, following  \cite{dkss-spher}. The next \S \ref{subsec:oplax} summarizes, in the form convenient for our use,
 some of the recent material from \cite{CDW}:  the gluing procedure for constructing SODs  $\A\oplus_F\B$   is interpreted as a form of (op)lax (co)limit,
and  all the four  possibilities give the same answer  in this case.  In \S \ref{subsec:laxmat} we further recall, restricting to the
context of glued SODs only,  the concept of
lax matrices from \cite{CDW} and explain how such matrices can be composed. 

 \vskip .2cm
 
 In Section \ref{sec:cont-compl} we introduce and study our main construction: the continuant complexes of adjoints. 

 These complexes are in fact (homotopy coherent)  diagrams  labelled by posets  known as {\em Fibonacci cubes},
 the concept originally introduced in computer science \cite{hsu}. In \S
 \ref {subsec:fibcubes} we explain how such diagrams can be completed by zeroes to ordinary cubes and so allow totalization. 
 In \S \ref{subsec:cont-cubes} we introduce Fibonacci cubes of functors and their iterated adjoints. By totalizing these
 cubes, we form new functors called {\em higher twists and cotwists} . They are the categorical analogs of continuants
 and  are studied in \S \ref{subsec:hi-twist}. 
 
  \vskip .2cm
  
  Section \ref{sec:N-spher} is dedicated to the concept of $N$-spherical functors, which are higher generalizations of the usual spherical functors, the 
  latter correponding to $N=4$. 
  
  In \S \ref {subsec:N-spher} we give the definitions and list the basic properties of $N$-spherical functors, with some proofs postponed 
  until the next \S \ref {subsec:N-spher-peri}, where we prove our main result, Theorem \ref{thm:N-spher=per}. This theorem characterizes
  $N$-spherical functors in terms of $N$-periodicity of the associated glued SODs. This characterization allows us to extablish many
  properties of $N$-spherical functors which otherwise are not so easy to prove. 
  
    \vskip .2cm
    
    In Section \ref{sec:ex} we give several examples of $N$-periodic SODs and $N$-spherical functors.

 In \S\S $\, $ \ref {subsec:quivera} and  \ref {subsec:quiverd} we discuss the categories of representations of quivers of Dynkin types $A$ and $D$
 in an arbitrary stable $\oo$-category $\DD$. 
  In \S \ref  {subsec:fra-CY} we recall the concepts of a Serre functor and of  a  fractional Calabi-Yau category \cite{kuznetsov}.
  In such a category any SOD is periodic. 
  Then in \S \ref{subsec:N-spher-obj} we study the higher $N$ analog of the concept of  spherical objects, meaning objects $E$
  such that the functor $-\otimes E$  from the derived category of vector spaces is $N$-spherical. We give examples
for $N=6$   related to Enriques surfaces and their generalizations.
  
      \vskip .2cm
      
      In the final Section \ref{sec:fur-dir}  we discuss the relation of our results with other work and the outlook for future developments. 
 
In  \S \ref{subsec:inter-cat-stokes} we present the analogy between $N$-periodic SODs and a particular type of Stokes data for
irregular connections on the complex line. In particular, we see how our categorical constructions match the appearance
of continuants in the study of such data \cite{boalch, fairon}. In \S  \ref {subsec:more-gen-lax} we discuss the natural
$(\oo, 2)$-categorical framework for our constructions. Finally, in \S \ref{subsec:cluster} we discuss possible relations of
our constructions with the (categorified) cluster formalism. 

\paragraph{Acknowledgements.} 
We are grateful to T. Kuwagaki and  A. Kuznetsov for valuable discussions and comments on this work during its various stages.  We would like to thank K. Coulembier and P. Etingof for informing
us about their work in progress \cite{coulembier} on a related subject
and  sharing their preliminary notes. 
We would also  like to thank  S. Fomin for helpful pointers to the cluster algebra literature.

The work of M.K.  was supported by the World Premier International Research Center Initiative (WPI
Initiative), MEXT, Japan and  by the JSPS KAKENHI grant 20H01794. The work of V. S.   was partly
supported by the World Premier International Research Center Initiative (WPI Initiative), MEXT,
Japan. T.D. acknowledges support by the Deutsche Forschungsgemeinschaft under Germany’s Excellence
Strategy – EXC 2121 “Quantum Universe” – 390833306.


\numberwithin{equation}{subsection}

\section{Decategorified picture: Euler polynomials}\label{sec:decat}

\subsection {Bilinear spaces.} \label{subsec:bil-spaces}

Let $\k$ be a field. By a {\em bilinear space}
we mean a finite-dimensional $\k$-vector space $C$ together with a 
non-degenerate, not necessarily symmetric bilinear form 
$\beta=\beta_C: C\times C\to\k$. An {\em isometry} of bilinear spaces is
a linear isomorphism preserving the bilinear forms. Given a bilinear space
$(C,\beta)$, we can view $\beta$ as an isomorphism $C\to C^*$ sending
$x\mapsto \beta(x,-)$.  We then have
the isomorphism $S=S_C=(\beta^t)^{-1}\circ\beta: C\to C$ so that
$\beta(x,y)=\beta(y, S(x))$ for any $x,y\in C$. The isomorphism $C$ is easily
seen to be an isometry and will be called the {\em Serre isometry} for $C$. 

\vskip .2cm

Let $(A,\beta_A)$ and $(B, \beta_B)$ be two bilinear spaces and $f: A\to B$
a linear operator. The right  and left adjoints $f^*, \lst f: B\to A$ are defined
by the conditions
\[
\beta_B(f(a),b) = \beta_A(a, f^*(b)), \quad \beta_A(^*f(b),a) = \beta_B(b,f(a)),
\quad \forall \, a\in A. b\in B. 
\]
They are connected by $f^* = S_A\circ \lst f \circ S_B^{-1}$. 
Note that $\lst(f^*) = (\lst f)^*=f$. 
By taking adjoints of adjoints, we get the chain of iterated adjoints
\[
\cdots, \lst\lst f = f^{(-2)}, \lst f=f^{(-1)}, f=f^{(0)}, f^*=f^{(1)}, f^{**}=f^{(2)}, \cdots
\]
 If $f$ is an isometry, then $f^*= \lst f = f^{-1}$. Conversely, if $f$ is isomorphism
 and either $f^*$ or $\lst f$ coincides with $f^{-1}$, then $f$ is an isometry.  
 Moreover, if $\dim A=\dim B$ and either of the four compositions $f(f^*)$,
 $(f^*)f$, $f(\lst f)$, $(\lst f)f$ is equal to the identity, then $f$ is an isometry. 
 
 \vskip .2cm

Let $(C,\beta)$ be a bilinear space. For a linear subspace $A\subset C$ its right and
left orthogonals are defined by
\[
A^\perp \,=\,\{c\in C| \,\beta(a,c)=0,\,\,\forall, a\in A\}, \quad
\lperp A\,=\,
\{c\in C| \,\beta(c,a)=0,\,\,\forall, a\in A\}.
\]
Note that $^\perp(A^\perp) = (^\perp A)^\perp = A$. 
By taking orthogonals of orthogonals, we get the chain of iterated orthogonals
\be
\cdots,  \lperp A = A^{\perp(-1)}, A=A^{\perp(0)}, A^\perp = A^{\perp (1)}, 
\A^{\perp\perp}=A^{\perp(2)}, \cdots
\ee

 A linear subspace $A\subset C$ is called
{\em non-degenerate}, if $\beta|_A$ is non-degenerate. 

\subsection{Semi-orthogonal decompositions and gluing.} \label{subsec:semi-orth-decat}

Let $C$ be a
bilinear space. We say that a pair $(A,B)$ of linear subspaces of $C$ forms
a {\em semi-orthogonal decomposition (SOD)} of $C$, if $C=A\oplus B$ and
$\beta(b,a)=0$ for any $b\in B, a\in A$. 

\begin{prop}\label{prop:SOD=nondeg}
 (1) If $(A,B)$ is an SOD for $C$, then $B=\lperp A$, $A=B^\perp$
and both $A,B$ are non-degenerate. 

\vskip .2cm

(2) For a subspace $A\subset C$ the following are equivalent:
\vskip .2cm

\noindent (2i)  $(A, \lperp A)$ is an SOD for $C$.

\noindent (2ii)  $A$ is non-degenerate.

\noindent (2iii)  $\lperp A$ is non-degenerate. \qed
\end{prop}

If $(A,B)$ is an SOD of $C$, then $(A, \beta_A=\beta|_A)$ and
$(B,\beta_B=\beta|_B)$ are bilinear spaces. The remaining part of $\beta$
is encoded in the {\em gluing operator} $f: A\to B$ defined by the conditions
\[
\beta(a,b) \,=\,\beta_B(f(a),b), \quad a\in A, b\in B. 
\]
Conversely, given two bilinear spaces $(A, \beta_A)$,  $(B, \beta_B)$
and a linear operator $f: A\to B$, we can construct a bilinear space
$(C=C_f, \beta_C)$ as follows, As a vector space $C=A\oplus B$, and
$\beta=\beta_C$ is defined by
 \[
 \beta(a,a')=\beta_A(a,a'), \quad \beta(b,b')=\beta_B(b,b'), \quad \beta(a,b) = \beta_B(f(a),b), \quad \beta(b,a)=0
 \]
for any $a,a'\in A$, $b,b'\in B$. We refer to $C_f$ as the {\em (bilinear space with the)
SOD glued from $A,B$ via $f$}. 

\vskip .2cm

Let $(A, B=\lperp A)$ be an SOD of $C$. By Proposition 
\ref{prop:SOD=nondeg}, $(A^\perp, A)$ is another SOD. In particular, both $\lperp A$ and $A^\perp$ are direct complements to $A$ in $C$. Therefore we have the
isomorphism $M_A: A^\perp\to\lperp A$ given by the projection to $\lperp A$
with respect to the direct sum decomposition $C= A\oplus \lperp A$. 
The inverse isomorphism $M_A^{-1}: \lperp A\to A^\perp$ is given by the projection
to $A^\perp$ with respect to the direct sum decompositon $C=A^\perp \oplus A$.
Further, it is immediate that $M_A$ is in fact an isometry with respect to the
bilinear forms induced by $\beta_C$ on $A^\perp$ and $\lperp A$. We will
call $M_A$ the {\em mutation isometry around} $A$. 

\begin{prop}\label{prop:glue-for-perp}
Let $f: A\to \lperp A$ be the gluing operator for the SOD $(A, \lperp A)$.
Then the gluing operator for $(A^\perp, A)$ is $(-f^*)\circ M_A: A^\perp\to A$. 
\end{prop}

\noindent {\sl Proof:} Denote the gluing operator for $(A^\perp, A)$ by
$g: A^\perp\to A$ so that $\beta(b', a) = \beta_A(g(b'), a)$ for any
$b'\in A^\perp, a\in A$. Given any $b'\in A^\perp$, we can write it uniquely as $b'=a_1+b_1$
with $a_1\in A$ and $b_1\in \lperp A$. Then $b_1=M_A(b')$. Since
$b_1\in A^\perp$, we have $\beta(b',a)=\beta_A(a_1,a)$ for any $a\in A$, so
$a_1=g(b')$ and therefore $g(b')$ is given by the projection of $b'$ to
$A$ in $A\oplus \lperp A$. It remains to identify $g$ as given by this projection. 
Taking into account the equality $b_1=M(b')$, we need to show that $a_1=-f^*(b_1)$,
i.e., that $-\beta_A(a,a_1) \buildrel ?\over = \beta_B(f(a), b_1)$\
for any $a\in A$. But the RHS of this putative equality is, by definition of $f$,
equal to $\beta(a,b_1)$. Since $b_1=b'=a_1$ and $\beta(a,b')=0$,
we get $\beta(a,b_1) = -\beta_A(a,a_1)$, so the equality marked ``?" is true. \qed

\begin{cor}\label{cor:coord-change}
The coordinate change transformation obtained as the composition of the
two isomorphisms
\[
A\oplus A^\perp \lra C \lra \lperp A \oplus A,
\]
has the matrix
\[
\begin{array}{c|c|c|}
& A& A^\perp
\\
\hline
\lperp A & 0 & M_A
\\
\hline
A& 1 & -f^*\circ M_A
\\
\hline
 \end{array}
\]
 \qed
\end{cor} 

\subsection{Cotwinned subsets and Euler polynomials.} \label{subsec:euler-poly}

For integers $p\leq q$ we denote by $[p,q]=\{p, p+1,\cdots, q\}$ the  integer interval bounded by $p$ and $q$. A subset
$I\subset [p,q]$ will be called {\em cotwinned}, if its complement
$\ol I = [p,q]\- I$ is a (possibly empty) disjoint union of ``twins'', i.e., of 
$2$-element subsets of the form $\{i, i+1\}$. In particular, $|\ol I|$ is even.
We define the {\em depth} of $I$ as $\dep(I)=|\ol I|/2$. 
We denote $\Cot[p,q]$ the set of all cotwinned subsets of $[p,q]$.
The following is well known. 

\begin{prop}\label{prop:fibonacci}
(a) $|\Cot[1,N]|=\phi_N$ is the $N$th Fibonacci number defined by $\phi_1=1, \phi_2=2$,
$\phi_N=\phi_{N-1}+\phi_{N-2}$. 

\vskip .2cm

(b) The number of $I\in\Cot[1,N]$ with $\dep(I)=k$ is equal to $N-k\choose k$. 
\end{prop}

\noindent{\sl Proof:} (a) Denote temporarily $|\Cot[1,N]|=c_N$. Then $c_1=1$
(the only cotwinned subset in $\{1\}$ is  $\{1\}$) and  $c_2=2$ (the cotwinned
subsets in $[1,2]$ are $\emptyset$ and $[1,2]$). Let us prove that
$c_N=c_{N-1}+c_{N-2}$ for $N\geq 3$. A cotwinned subset $I\subset [1,N]$
either contains $N$, or it does not. If $I\ni N$, then $I=I'\cup\{N\}$ where
$I'\in\Cot[1,N-1]$ can be arbitrary. If $I\not\ni N$, then $I\not\ni N-1$ as well
(otherwise $\ol I$ contains $N$ but not $n-1$ so it cannot be a disjoint union of
twins). Therefore $I\in\Cot[1,N-2]$. This establishes a bijection
$\Cot[1,N]\to\Cot[1,N-1]\sqcup\Cot[1,N-2]$, thus proving (a).
Part (b) is obtained by an elementary computation.
\qed

\vskip .2cm

Let now $x_1,\cdots, x_N$ be noncommutative variables
 and $\k\lan x_1,\cdots, x_N\ran$ be the $\k$-algebra of noncommutative polynomials in them.
 Let also $q$ be a central variable so that $\k[q]\lan x_1,\cdots, x_N\ran$ is the algebra
 of noncommutative polynomials in the $x_i$ over the commutative ring $\k[q]$. 
For any subset $I=\{i_1<\cdots < i_p\} \subset [1,N]$ we denote by
\[
x_I \,=\, x_{i_1} \cdots x_{i_p} \,\in \, \k\lan x_1,\cdots, x_N\ran
\]
the noncommutative monomial associated to $I$ and its natural ordering. 
We  set $x_\emptyset=1$. 

\begin{defi}
For $N\geq 1$ the $N$th (noncommutative) {\em parametric Euler polynomial}  is 
\[
E_N(x_1,\cdots, x_N;  q) \,=\,\sum_{I\in\Cot[1,N]} q^{\dep(I)} x_I \,\in\, 
\k[q]\lan x_1,\cdots, x_N\ran \, .
\]
We also put $E_0=1\in\k$. We define the {\em  (alternating) Euler polynomial} as
\[
E_N(x_1,\cdots, x_N) \,=\, E_N(x_1,\cdots, x_N; -1)\, \in\,  \k\lan x_1,\cdots, x_N\ran. 
\]
\end{defi}
For example,
\[
\begin{gathered}
E_1(x)=x,
\\
E_2(x_1, x_2) = x_1 x_2 -1,
\\
E_3(x_1, x_2, x_3) = x_1x_2x_3 -x_1-x_3,
\\
E_4(x_1, x_2, x_3, x_4) = x_1x_2 x_3 x_4 - x_1x_2 - x_1x_4 - x_3x_4 + 1,
\end{gathered}
\]
and so on. Thus $E_N$ has $\phi_N$ monomials. The more standard
 {\em Euler's continuants} \cite{boalch, knuth}
are, in our notation, 
\[
(x_1,\cdots, x_N) \,=\, E_N(x_1,
\cdots, x_N; 1)
\]
The alternating version will be more important for us and will later categorify into complexes\footnote{This form of $E_N$ is  also more natural because for commuting $x_i$ it appears as the determinant
of a  {\em symmetric} $3$-diagonal matrix \cite[\S 4]{perron} while for the standard $(x_1,\cdots x_N)$
the matrix is not symmetric.
 }. 
We have the following analogs of the standard properties of continuants
  \cite{boalch, fairon, knuth, perron}.
    \begin{prop}\label{prop:euler-recurs}
  (a) We have
  \[
  \begin{gathered}
  E_N(x_1,\cdots, x_N;q) \,=\, x_1 E_{N-1}(x_2,\cdots, x_N;q) +q E_{N-2}(x_3, \cdots, x_N; q)
  \,=
  \\
  =\,  E_{N-1}(x_1,\cdots, x_{N-1}; q)x_N  +q E_{N-2}(x_1,\cdots, x_{N-2}; q). 
  \end{gathered}
  \]
  
   (b) Let $\KK$ be a
  skew field containing $\k[q]\lan x_1,\cdots, x_N\ran$. Consider the continued fraction
  \[
  R_N \,=\,  x_1 + \cfrac{q}{x_2 + \cfrac{q}{\ddots \, +\cfrac{q}{x_N}}} \,\,\in \,\, \KK. 
  \]
  Then
  \[
  \begin{gathered}
  R_N\,=\, P_N Q_N^{-1} \,=\, (Q'_N)^{-1} P'_N, \quad \text{where}
  \\
  P_N=  E_N(x_1,\cdots, x_N;q), \quad Q_N= E_{N-1}(x_2,\cdots, x_N; q), 
  \\
  P'_N=  E_N(x_N, \cdots, x_1; q), \quad Q'_N=E_{N-1}(x_N, \cdots, x_2; q). 
  \end{gathered}
  \]

  (c) Let $U_i=\begin{bmatrix} 0&1\\1& (-1)^i x_i\end{bmatrix}\in\Mat_2 \k\lan x_1,\cdots, x_N\ran$. Then
  \[
  U_1\cdots U_n \,=\, \begin{bmatrix}
  (-1)^{[{N-2\over 2}] }E_{N-2}(x_2, \cdots, x_{N-1}) &  (-1)^{[{N-1\over 2}]} 
  E_{N-1}(x_2,\cdots, x_N)  &
  \\
  (-1)^{[{N\over 2}]}E_{N-1}(x_1,\cdots, x_{N-1}) & 
   (-1)^{[{N+1\over 2}]} E_N(x_1,\cdots, x_N)
  \end{bmatrix}. 
  \]
  
  (d) We have the identity
  \[
  E_N(x_1, \cdots, x_N; q) E_N(x_{N+1}, x_N, \cdots, x_2; q) -
  E_{N+1}(x_1, \cdots, x_{N+1}; q) E_{N-1}(x_N, x_{N-1}, \cdots, x_2; q) = (-q)^N. 
  \]

 \end{prop}
 
 \noindent{\sl Proof:} (a)  follows from the bijection established in the proof of
  Proposition 
 \ref{prop:fibonacci}(a). Parts (b) and (c)  follow from (a) by induction. 
 Part (d) is obtained by direct comparison of monomials in the two products.
 \qed
 
 \begin{rem}\label{rem:Euler-det-id}
 After renumbering the variables $x_1, \cdots, x_{N+1}$ in the reverse order,
 part (d) of Proposition \ref{prop:euler-recurs} becomes
 \[
 E_N(x_{N+1}, \cdots, x_2; q) E_N(x_1,\cdots, x_N; q) - E_{N+1}(x_{N+1}, \cdots, x_1; q)
 E_{N-1}(x_2, \cdots, x_N; q) =(-q)^N,
 \]
 or, in the notation of part (b),
 \[
 Q'_{N+1}P_N - P'_{N+1} Q_N = (-q)^N. 
 \]
 Multiplying this by $(Q'_{N+1})^{-1}$ on the left and by $Q_N^{-1}$ on the right and
 applying  the identities of part (b), we can rewrite this as
 \be\label{eq:best-approx}
 R_{N+1}-R_N  \,=\, { (-1)^{N+1} q^N \over Q_N Q'_{N+1}},
 \ee
 which is a noncommutative lift of the classical formula for the difference
 between two successive convergents of a continued fraction, see, e.g.,
 \cite[\S 6]{perron}. In the noncommutative case, part (d) (for $N$ even and $q=1$)
 was established in \cite[Lem. A.1]{fairon}. 
 \end{rem}
 
 \subsection {Fibonacci and Chebyshev polynomials.}\label{subsec:Fib-Cheb-poly}
 
 By specializing all noncommutative variables in $E_N(x_1,\cdots, x_N;q)$
 to one commutative variable $x$, we obtain the
 $N$th  {\em  $2$-variable Fibonacci polynomial} 
 \[
 \Phi_N(x,q) \,=\, E_N(x,\cdots, x; q) \,=\,\sum_{0\leq k\leq N/2} 
 {N-k\choose k} x^{N-2k} q^k \,\in \, \k[x,q], 
 \]
 so that $\Phi_N(1,1)=\phi_N$. The following is also elementary,
 cf. \cite{benjamin, hoggatt}.
 
 \begin{prop}\label{prop:chebyshev}
 The specialization
 $\Phi_N(x, -1) = E_N(x,\cdots, x)$ is equal to $U_{N}(x/2)$, where 
 $U_{N}(x)$ is the $N$th Chebyshev polynomial
 of the second kind defined by
 \[
 U_N(\cos\theta)\sin\theta = \sin((N+1)\theta).
 \]
  In particular, the roots of $\Phi_N(x,-1)$ have the form
 $2\cos (k\pi /(N+1))$, $k=1,\cdots, N$. \qed
 
 \end{prop}
 
 \subsection {Periodic orthogonals.} \label{subsec:peri-orth-decat}
 
 Let $C$ be a bilinear space and $A\subset C$
 a nondegenerate subspace. We say that $A$ is $N$-{\em periodic} ($N\geq 2$), if
 $A^{\perp (N)}=A$. In this case each $A^{\perp(i)}$, $i\in\ZZ$, is $N$-periodic as
 well. 
 
 Assume that $C=C_f$ is glued from bilinear spaces $A,B$ via an operator $f: A\to B$. 
 We want to find the conditions  on $f$ guaranteeing that $A$ (or, equivalently, $B$)
 is $N$-periodic in $C$. Here the case if even and odd $N$ are somewhat different.
 Since $\dim A^{\perp\perp}=\dim A$, for $N$ even the orthogonals split into two
 groups, the $A^{\perp(2i)}$ and the $B^{\perp(2i)}$, while for $N$ odd and $A$ periodic
  we necessarily have $\dim A=\dim B = {1\over 2} \dim C$. 
 
 \begin{prop}\label{prop:period=E_{N-1}}
Let $N\geq 2$, let $A,B$ be bilinear spaces and, in the case $N$ is odd, assume
$\dim A=\dim B$.  Let $f: A\to B$ be a linear operator. In order that $A\subset C_f$
be $n$-periodic, it is nevessary and sufficient that
$E_{N-1}(f^*, f^{**}, \cdots f^{(N-1)})=0$. In this case we also have
$E_{N-1}(f^{**}, f^{***}, \cdots, f^{(N)})=0$, while
$E_{N-2}(f^{**}, \cdots, f^{(N-1)})$ and $E_N(f^*, f^{**}, \cdots, f^{(N)})$
are isometries. 
 \end{prop}
 
 \noindent{\sl Proof:}\ul{ Necessity}: For $i\geq 1$ let 
 $V_i: A^{\perp (i-1)}\oplus A^{\perp(i)} \to A^{\perp (i-2)}\oplus
A^{\perp(i-1)}$ 
  be the   coordinate change
transformation 
induced by the identifications of the source and target
with $C$.
If $A$ is $N$-periodic, then $A^{\perp(N)}=A$, $A^{\perp(N-1)}
=\lperp A$ and $V_1\cdots V_N = \Id$. On the other hand, by Proposition 
 \ref{cor:coord-change}, the $2\times 2$ block matrix of $V_i$  is
 \[
 W_i \,=\,\begin{bmatrix}
 0 & M_{A^{\perp (i-1)}} 
 \\
 1 & -f_i^* M_{A^{\perp(i-1)}}
 \end{bmatrix},
 \]
where $f_i: A^{\perp(i-1)} \to A^{\perp(i-2)}$ is the gluing operator for the SOD 
$(A^{\perp(i-1)}, A^{\perp(i-2)})$. The condition $V_1\cdots V_N=\Id$ as operators
is equavalent to 
 $W_1\cdots W_N=1$ as matrices. 
By iterating Proposition 
\ref{prop:glue-for-perp}, we obtain that
$-f_i^* = (-1)^i M' f^{(i)} M''$, where, for $i$ even we have put
\[
M'= M^{-1}_{A^{\perp (i-2)}} M^{-1}_{A^{\perp(i-3)}} \cdots M^{-1}_{A^\perp}, \quad
M'' = M_A M_{A^{\perp\perp}}\cdots M_{A^{\perp(i-2)}}, 
\]
while for $i$ odd we have put
\[
M'= M^{-1}_{ A^{\perp(i-2)}} M^{-1}_{A^{\perp(i-4)}}\cdots M^{-1}_A, \quad
M'' = M_{A^\perp} M_{A^{\perp (3)}}\cdots M_{A^{\perp(i-2)}}. 
\]
So applying Proposition \ref{prop:euler-recurs}(c), we get 
\[
W_1\cdots W_N\,=\,\begin{bmatrix}
\pm M'_{11} E_{N-2}(f^{**}, \cdots, f^{(N-1)}) M''_{11} &
\pm M'_{12} E_{N-1}(f^{**}, \cdots, f^{(N)}) M''_{12}
\\
\pm M'_{21} E_{N-1}(f^*, \cdots, f^{(N-1)}) M''_{21}  &
\pm M'_{22} E_N(f^*, \cdots, f^{(N)}) M''_{22}
\end{bmatrix}, 
\]
 where $M'_{ij}, M''_{ij}$ are isometries (appropriate compositions of mutations
 or their inverses). Therefore $W_1\cdots W_N=1$ being the unit matrix
 means that 
 \[
 E_{N-1}(f^*, \cdots, f^{(N-1)}) \,=\, E_{N-1}(f^{**}, \cdots, f^{(N)})\,=\,0, 
 \]
 whille $E_{N-2}(f^{**}, \cdots, f^{(N-1)})$ and $E_N(f^*, \cdots, f^{(N)})$
 are isometries. 
 
 \vskip .2cm
 
\ul{ Sufficiency}: Suppose $ E_{N-1}(f^*, \cdots, f^{(N-1)})=0$. Then, the matrix
 $W_1\cdots W_N$ is block triangular  and therefore
 $A^{\perp(N-1)} \subset \lperp A$. Since
 $\dim A^{\perp(N-1)}  = \dim \lperp A$
  (always for $N$ even, in virtue of our assumptions
 for $N$ odd), we conclude that $ A^{\perp(N-1)}=  \lperp A$, i.e., the chain
 of orthogonals is $N$-periodic. \qed
 
 \begin{rems}
 (a) Since $N$-periodicity of $A$ is equivalent to $N$-periodicity of $A^{\perp(i)}$ for
 any $i$, the  conditions of Proposition \ref{prop:period=E_{N-1}} for $f$ 
 imply the validity of the same conditions for any $f^{(i)}$.
 
 \vskip .2cm
 
 (b) The conclusion that the operators  $E_{n-N}(f^{**}, \cdots, f^{(N-1)})$ and $E_N(f^*, f^{**}, \cdots, f^{(N)})$ (denote them for short $E_{n-2}$ and $E_n$)
are isometries, can be seen directly from  Proposition \ref{prop:euler-recurs}(d).
For example, isometricity of  $E_N$ means that $E_N E_N^*$ and $(\lst E_N) E_N$  
are equal to the identity operators. Applying Proposition \ref{prop:euler-recurs}(d) 
with $q=-1$, we write, say,  $E_N E_N^*$ as
\[
 E_N(f^*, \cdots, f^{(N)}) E_N(f^{(N+1)}, \cdots, f^{(2)}) =
E_{N+1}(f^*, \cdots, f^{(N+1)}) E_{N-1}(f^{(N)}, \cdots, f^{(2)})+1,
\]
and $E_{N-1}(f^{(N)}, \cdots, f^{(2)})=0$ by passing to the right adjoint of 
$E_{N-1}(f^*, \cdots, f^{N-1})=0$.

 \vskip .2cm
 
 (c) Using Proposition \ref{prop:euler-recurs}(b), we can express the condition of
 $N$-periodicity for $f$ symbolically as the vanishing of the continued fraction:
 \[
  f^* - \cfrac{1}{f^{**} - \cfrac{1}{\ddots \,  - \cfrac{1}{  f^{(N-1)}   }}} \,=\,0. 
 \]
 More precisely, in this case the numerator of the continued fraction turns out to be $0$
 while the denominator is invertible, so the equality makes sense. 
 \end{rems}
 
 \begin{exas}
 (a) $2$-periodicity. For $A=C_f$ to be $2$-periodic, i.e., to have $A^{\perp\perp}=A$,
 is the same as $f=0$. Now, Proposition  \ref{prop:period=E_{N-1}} gives the condition
 as $E_1(f^*)=f^*=0$ which amounts to the same. In this case $E_0=1$
 and $E_1(f^*, f^{**})=f^* f^{**}-1 =-1$ are obviously invertible.
 
 \vskip .2cm
 
 (b) $3$-periodicity. The condition of Proposition  
 \ref{prop:period=E_{N-1}}
 (assuming
 $\dim A=\dim B$) is  that $E_2(f^*, f^{**})=f^*f^{**}=1=0$. This means that $f^*$
 (and therefore $f$) is an isometry. Thus $C_f$ is $3$-periodic if and only if $f$
 is an isometry. In this case $E_1(f)=f^*$ is of course an isometry, while
 \[
 E_3(f^*, f^{**}, f^{***}) \,=\, f^* f^{**} f^{***}-f^* -f^{***} \,=\,
 f^{-1} f f^{-1} -f^{-1} -f^{-1} \,=\, -f^{-1}
 \]
 is an isometry as well.
 
 \vskip .2cm
 
 (c) $4$-periodicity.  Proposition  \ref{prop:period=E_{N-1}} gives the condition as
 \[
  E_3(f^*, f^{**}, f^{***}) \,=\, f^* f^{**} f^{***}-f^* -f^{***} \,=\,0
 \]
 and says that in such case both $E_2(f^*, f^{**})=f^*f^{**}-1$ and
 \[
 E_4(f^*, \cdots f^{(4)}) \,=\, f^* f^{**} f^{***} f^{(4)} -f^* f^{**} - f^* f^{(4)} - f^{***} f^{(4)} +1
 \]
 are isometries.  
 \end{exas}
 
 \begin{ex}[($N$-periodicity in $2$ dimensions)] \label{ex:2d}
 Let $A=B=\k$ be $1$-dimensional with the standard bilinear form in which $1\in\k$
 is of length $1$. Then $f: A\to B$ is given by a number $\rho\in \k$, and each
 iterated adjoint $f^{(i)}$ is given by $\rho$ as well. Explicitly, gluing an SOD via
 $\rho\in\k$ amounts to considering $C=\k^2$ with the bilinear form $\beta$ given
 by the matrix $\begin{bmatrix} 1&\rho\\0&1\end{bmatrix}$. 
 The condition of $N$-periodicity is then given by vanishing
 of $E_{N-1}(\rho,\cdots, \rho) = \Phi_{N-1}(\rho,-1)$. By Proposition 
 \ref{prop:chebyshev} this means that $\rho$ must have the form
 $2\cos(k\pi/N)$ for some $k=1,\cdots, N-1$. 
 We see, in particular, that the only possible integer values of $\rho$ are $\rho=\pm 1$
 corresponding to $N=3$. 
 
 \end{ex}
 
 \begin{exas}[(Even periodicity for $A=\k$ and $B$ symmetric)]
 Let $A=\k$ as before and assume for simplicity that $\beta=\beta_B$ is symmetric.
 Assume further that $n$ is even. 
 A linear operator $\k\to B$ is given by a vector $v\in B$ via $f_v(t)=tv$. 
 We can ask for which $v$ the SOD glued via $f_v$ is $n$-periodic. 
 The reasoning of Example \ref{ex:2d} is modified straightforwardly to give thar
 $\rho=\sqrt{B(v,v)}$ must be among the roots of $\Phi_{n-1}(x,-1)$, i.e., that
 $B(v,v)$ must be the square of one of these roots.  This gives two possibilities
 of integer $B(v,v)$, which can be realized inside an integer lattice.
 
 \vskip .2cm
 
 (a) $4$-periodicity. It corresponds to $B(v,v) = (2\cos(k \pi/4))^2$, $k=1,2,3$,
 i.e.,  $B(v,v)=2$ (``root vector'', $k=1,3$) or $B(v,v)=0$ (null-vector, $k=2$). The first case
 is an analog of the concept of a spherical object in an even-dimensional  Calabi-Yau 
 category. In the second case we   have $2$-periodicity which of course
 implies $4$-periodicity.
 
 \vskip .2cm
 
 (b) $6$-periodicity. It corresponds to $B(v,v) = (2\cos (k\pi/6))^2$, $k=1,\cdots, 5$, 
 which gives
 $B(v,v)=3$ (for $k=1,5$) or $B(v,v)=1$ (for $k=2,4$) or
 $B(v,v)=0$ (for $k=3$), the last case corresponding to $2$-periodicity as before. 
 
 \end{exas}
 

\section{Stable $\oo$-categories,  semi-orthogonal decompositions and lax matrices}
\label{sec:stable-oo}

\subsection{Generalities on (stable) $\oo$-categories.}\label{subsec:gen-oo-cat}

\paragraph{Conventions.} We follow the standard conventions of \cite{lurie:htt, lurie:ha} as summarized in \cite[\S 1.1]{dkss-spher} to which we refer for more details. Thus, $\oo$-categories are understood as simplicial sets satisfying the weak Kan condition.
The (ordinary) category of simplicial sets is denoted by $\Set_\Delta$. For a simplicial set $S$ we denote by $S^\op$ the opposite
simplicial set \cite[\S 1.2.1]{lurie:htt}. 
The nerve of an ordinary category $C$ will be denoted $\N(C)\in\Set_\Delta$. Thus $\N(C^\op) = \N (C)^\op$
We denote by $\Sp$ the $\oo$-category of spaces (understood as Kan simplicial sets). 
For any two objects $x,y$ of an $\oo$-category $\C$ one has a homotopy canonical (i.e., defined uniquely up to a contractible space of choices)
space $\Map_\C(x,y)\in\Sp$ called the {\em mapping space} from $x$ to $y$. 
The $\oo$-categorical limits and colimits will be denoted $\varprojlim$ and $\varinjlim$. They are defined in the standard way by universal
properties which are understood as
contractibility if appropriate mapping spaces. An $\oo$-functor between $\oo$-categories is called {\em exact}, if it preserves all limits
and colimits. We denote by $\ExFun(\C, \D)$ the $\oo$-category of exact $\oo$-functors between $\oo$-categories
$\C$ and $\D$. It is stable, if $\D$ is stable. 
 
  \vskip .2cm 
  
 For a morphism $f: x\to y$ in a stable $\oo$-category $\C$ its {\em fiber} and {\em cofiber (cone)} are defined by
  \[
   \Fib(f) \,=\, \varprojlim \, \{ x\buildrel f\over\to y \leftarrow 0\},
    \quad \Cof(f)=\Cone(f) \,=\,   \varinjlim \, \{ 0\leftarrow x \buildrel f\over\to y\},  
   \]
   while the {\em shift functors} (suspension and desuspension) are defined by
    \[
   x[1] \,=\,\Cof\{x\to 0\}, \quad x[-1] \,=\,\Fib\{ 0\to x\}. 
   \]
   It is standard that  
   \be\label{eq:fib=cof[-1]}
   \Fib(f) \simeq \Cof (f)[-1].
   \ee
    A diagram
   \[
   x\buildrel f\over\lra y \buildrel g\over\lra z
   \]
 in a stable $\oo$-category $\C$ is called a {\em triangle}, if $g$ induces an equivalence
 $\Cof(f)\to  z$ or, equivalently, if $f$ induces an equivalence $x\to \Fib(g)$. 
 Such triangles give distinguished triangles in the homotopy category $\h\C$, making it into a triangulated category.

\paragraph{Totalization of cubes.}
  We will need to iterate the fiber and cofiber constructions. 

Let $\2=\{0 < 1\}$ be the standard $2$-element poset. Thus the nerve $\N(\2)=\Delta^1$ is the
$1$-simplex. Let $I$ be a finite set. We denote by $\2^I$ the set of all subsets in $I$,
ordered by inclusion. We refer to this poset as the $I$-{\em cube}. For $N\in\ZZ_{>0}$
we write $\2^N=\2^{[1,N]}$ and call it the $N$-{\em cube}. 

Let $\C$ be an $\oo$-category. By an $I$-{\em cube in} $\C$ we mean an $\oo$-functor
$Q: \2^I\to\C$. A {\em morphism of $I$-cubes in} $\C$ is a natural transformation of
$\oo$-functors.  Let $i\in I$. Any $I$-cube $Q$ in $\C$ gives two $I\-\{i\}$-cubes
$\del_i^0 Q, \del_i^1 Q$ in $\C$ connected by a morphism 
of $I\-\{i\}$-cubes $\del_i^0 Q \to \del_i^1 Q$.
 They are obtained by composing $Q$ with the
two emdeddings
\[
\delta_i^0, \delta_i^1: I\- \{i\} \to I, \quad \delta_i^0(J)=J, \quad 
\delta_i^1(J)=J\cup\{i\}. 
\]
 Given an $I$-cube $Q$ in $\C$, we denote by $Q^{\neq\emptyset}$ the restriction of
 $Q$ to $\2^I\-\{\emptyset\}$. 
 
 Assume now that $\C$ is a stable $\oo$-category and $Q: \2^I\to \C$ is an $I$-cube in $\C$.
 Adapting the approach of  \cite[Def. 1.1b]{goodwillieII},
 we define the {\em totalization} of $Q$ to be the object
 \be
 \Tot(Q) \,=\, \Fib\bigl\{ Q(\emptyset) \buildrel c\over\lra \varprojlim Q^{\neq\emptyset}\bigr\},
 \ee
where $c$ is the canonical morphism given by the structure of an $I$-cube on $Q$.

\begin{prop}\label{prop:Tot(Q)=iterfiber}
(a) For each $i\in I$ we have a triangle
\[
\Tot(Q) \lra \Tot(\del_i^0Q) \lra \Tot(\del_i^1Q). 
\]
(b) We have an equivalence
\[
\Tot(Q) \,\=\, \Cof\bigl\{ \varinjlim Q^{\neq I}\buildrel c'\over  \lra Q(I)\bigr\}[-|I|],
\]
the canonical morphism $c'$ being provided by the structure of $I$-cube on $Q$.
\end{prop}

Part (a) of the proposition says that $\Tot(Q)$ can be defined as the iterated fiber of the arrows in $Q$. 

\vskip .2cm

\noindent{\sl Proof:}   (a) is proved by computing the limit in two steps (``Fubini theorem''). 
 Part (b), in the simplest case $|I|=1$, reduces to  \eqref{eq:fib=cof[-1]}. The general
 case follows from this  by induction, using (a). 
 We omit straightforward details. \qed

 \paragraph {Grothendieck constructions.} 
 We will freely use the concepts of Cartesian and coCartesian fibrations of simplicial sets \cite{lurie:htt}.
 Given an ordinary category $C$ and an (ordinary) functor $q: C\to \Set_\Delta$, we denote by 
 $\Gamma(q)$ resp. $\chi(q)$ the {\em covariant} resp. {\em contravariant Grothendieck construction} of $q$,
 see, e.g., \cite[Def. 1.1.9] {dkss-spher}. It is equipped with a coCartesian resp. Cartesian fibration
 \[
 \pi: \Gamma(q) \to \C, \quad \text{resp.} \quad \varpi: \chi(q)\to \C^\op, \quad \C:=\N C,
 \]

 \begin{ex}
 We will be particularly interested in the case $C=\2$ (the poset considered
 as a category) so that $\C=\Delta^1$ 
 and  $q$ reduces to a morphism of simplicial sets $F: \A=q(0)\to \B=q(1)$. We further assume that $\A,\B$ are
 $\oo$-categories, so $F$ is an $\oo$-functor. In this case we use the notations $\Gamma(F), \chi(F)$ for corresponding
 Grothendieck constructions. At the low levels, they are described as follows, cf. \cite[Ex. 1.1.11]{dkss-spher}.
 
 Vertices of  $\Gamma(F)$  are formal symbols $[a], a\in\Ob(\A)$
 and $[b], b\in\Ob(\B)$.  The same for $\chi(F)$. 
 Further, edges among the $[a]$ or the $[b]$ are
  the same as edges (morphisms) in $\A$ and $\B$ in both cases. In adidition, in $\Gamma(F)$ we have new edges 
 $[a]\to [b]$ corresponding to morphisms $F(a)\to b$ in $\B$, while in $\chi(F)$ we have new edges $[b]\to[a]$
 corresponding to morphisms $b\to F(a)$. 
 
 General relations between (co)Cartesian fibrations and pseudo-functors \cite[Th. 3.2.0.1]{lurie:htt} 
 imply the following. Each coCartesian fibration $\pi: \X\to \Delta^1$ with $\A=\pi^{-1}(0)$
 and $\B=\pi^{-1}(1)$ being $\oo$-categories,   can be identified with $\Gamma(F)$
 for a  homotopy canonical functor $F:  \A\to\B$. Similarly, a Cartesian fibration $\varpi: \Y\to\Delta^1$
 with $\A=\varpi^{-1}(0)$
 and $\B=\varpi^{-1}(1)$ being $\oo$-categories,  can be identified with $\chi(F)$
  for a  homotopy unique functor $F:  \A\to\B$. See \cite[Ex. 1.1.12]{dkss-spher} for explicit details. 
 \end{ex}
 
 \paragraph{Adjoint $\oo$-functors.} Let $F: \A\to\B$ be an $\oo$-functor of $\oo$-categories. 
 We say that $F$ {\em admits a right adjoint}, if the covariant Grothendieck construction
 $\Gamma(F)\to\Delta^1$ (which is always a coCartesian fibration), is also a Cartesian fibration. In this
 case $\Gamma(F)^\op \to(\Delta^1)^\op=\Delta^1$ is a Cartesian fibration and so comes from
 a homotopy unique $\oo$-functor $G: \B\to\A$ which is called the {\em right adjoint} to $F$ and denoted $G=F^*$.
 In this case we have homotopy coherent identifications of mapping spaces
 \[
 \Map_\B(F(a), b) \,\= \, \Map_\A(a, G(b)), \quad a\in \A, b\in \B.
 \]
 In particular, we have the unit and counit natural transformations 
 \be
 e: \Id_\A \lra GF, \quad c: FG\lra \Id_\B. 
 \ee
 Dually, an $\oo$-functor $G: \B\to \A$ is said to {\em admit a left adjoint}, if $\chi(G)\to \Delta^1$ is also a coCartesian fibration.
 In this case $\chi(G)^\op$ is Cartesian and so comes from a  homotopy unique functor $F: \A\to B$ which is called 
 {\em left adjoint} to $G$ and denoted ${^*G}$. If $G=F^*$, then $G={^*}F$ and vice versa. We also express this by saying 
 that  $(F,G)$ is an
 {\em adjoint pair}.   
 

 \subsection{Semi-orthogonal decompositions}\label{subsec:SOD}
 
 \paragraph{General definitions. Mutations. } 
 In   \cite[\S 2.2]{dkss-spher},  the theory of semi-orthogonal
 decompositions of triangulated categories \cite{BK:SOD} was extended to the stable
 $\oo$-categorical context. We recall the main points, referring to \cite{dkss-spher}
 for more details.

    Let $\C$ be a stable $\oo$-category. 
A full subcategory $\A\subset \C$ is called {\em stable}, if it contains $0$ and is closed
under taking fibers and cofibers of morphisms. Such $\A$ is a stable $\oo$-category
in the intrinsic sense.

\vskip .2cm

Given stable $\A\subset \C$, its   {\em left} and {\em right orthogonals} are full $\oo$-subcategories
in $\C$ defined by
\[
\begin{gathered}
^\perp\A\,=\, \bigl\{ c\in \C\bigl| \, \Map_\C(c,a) \text{ is contractible, }\,\,\forall \, a\in\Ac\bigr\}, 
\\
\A^\perp \,=\, \bigl\{ c\in \C\bigl| \, \Map_\C(a,c) \text{ is contractible, }\,\,\forall \, a\in\Ac\bigr\}. 
\end{gathered}
\]
They are automatically stable. The iterated orthogonals are defined inductively as 
\[
\A^{\perp n} = (\A^{\perp (n-1)})^\perp, \quad {^{\perp n}}\A = {^\perp}(^{\perp (n-1)}A), \quad n\geq 1. 
\]

\vskip .2cm

Let $\A, \B$ be stable subcategories of $\C$. Denote 
 $\arr{\A,\B} \subset \Fun(\Delta^1, \C)$
  the full subcategory spanned by edges $a \to b$ in $\C$ with $a$ in $\A$ and $b$ in
	$\B$. We have the $\oo$-functor $\Fib: \arr{\A,\B}\to\C$ which associated to a morphism
	its fiber.  
	We say that $(\A, \B)$ form a {\em semi-orthogonal decomposition} (SOD for short) of $\C$,
	if the functor $\Fib$ is an equivalence. This implies, in particular, that 
	 $\B = {^\perp}\A$ and $\A = \B^{\perp}$, see \cite[Prop.2.2.4]{dkss-spher}.  
	 
	 \vskip .2cm
	 
	  Let $\A\subset \C$ be  stable subcategory. We denote $\eps_\A: \A\to \C$ the embedding.
	  The subcategory $\A$   is called {\em left admissible}, if $(\A, {^\perp}\A)$
	 is an SOD for $\C$. This is equivalent to $\eps_\A$ having a left adjoint ${^*}\eps_\A$.
	  Similarly, $\A$ is called 
	 {\em right admissible}, if $(\A^\perp, \A)$ is an SOD. This is equivalent to $\eps_\A$ having a right adjoint $\eps_A^*$. 
	 See 
	 \cite[Prop. 2.3.2]{dkss-spher}.
	 
	 \vskip .2cm
	 
	 If $(\A, \B)$ is an SOD, then the diagram in $\ExFun(\C, \C)$
	 \be\label{eq:triangle-eps*}
	 \eps_\B \eps_\B^* \buildrel e\over\lra \Id_\C \buildrel c\over\lra \eps_\A ({^*}\eps_\A)
	 \ee	 
	 (with $e$ and $c$ being the unit and counit of the adjunctions), is a triangle. This can be expressed by saying
	  that each object $c\in C$
	 is included in a functorial triangle
	 \[
	  b= \eps_\B^*(c) \lra c \lra a= \, {^*}\eps_A(c), \quad a\in\A, \,\, b\in \B,
	 \]
	 where we identify objects of $\A$ and $\B$ with their images in $\C$. 
	 We will use the shift of the triangle \eqref{eq:triangle-eps*}:
	 \be\label{eq:triangle-delta}
	 \eps_\A (^*\eps_\A)[-1] \buildrel \delta_{(\A, \B)} \over\lra \eps_\B \eps^*_\B \lra \Id_\C,
	 \ee
	 representing $\Id_\C$ as the cofiber of the canonical transformation $\delta_{(\A, \B)}$ coming from
	  \eqref{eq:triangle-eps*} . 
	 	 \vskip .2cm

	 We say that $\A$ is {\em admissible}, if it is both left and right admissible. In this case the {\em left} and {\em right
	 mutation functors}
	 \be
	 \lambda_\A = ({^*}\eps_{\A^\perp})\circ \eps_{{^\perp} \A}: {^\perp}\A \lra \A^\perp, 
	 \quad \rho_\A = (\eps^*_{{^\perp}\A}) \circ \eps_{A^\perp}: \A^\perp \lra {^\perp} \A
	 \ee
	 are mutually quasi-inverse equivalences \cite[Prop. 2.4.2]{dkss-spher}. 
	 
	 	 \vskip .2cm

	 Further, $\A$ is called $\oo$-{\em admissible}, if each $\A^{\perp n}$ and ${^{\perp n}}\A$,
	 $n\geq 0$,  is
	 admissible. In this case we have an infinite chain of orthogonals conveniently 
	 indexed by $\ZZ$:
	 \be
	 \cdots,  \A^{\perp (-2)} = {^{\perp 2}}\A, \,\, \A^{\perp (-1)} = {^\perp \A}, \,\, 
	 \A^{\perp 0} = \A, \,\,  \A^{\perp 1} = \A^\perp, \,\,  \A^{\perp 2}, \cdots
	 \ee
	in which each $(\A^{\perp i}, \A^{\perp (i-1)})$ is an SOD and each $\A^{\perp i}$ identified with $\A^{\perp(i+2)}$
	by the mutation $\lambda_{\A^{\perp (i+1)}}$.    
	
	\paragraph{Gluing functors.} Let $F: \A\to \B$ be an exact functor of stable $\oo$-categories. We define the stable
	$\oo$-category $\A\oplus_F \B$  as the pullback (of simplicial sets)
\[
\begin{tikzcd}
	\A \oplus_F \B \ar[r]\ar[d] & \Fun(\Delta^1,\B) \ar[d, "\on{ev}_0"] \\
	\A \ar[r, "F"] & \B.  
\end{tikzcd}
\]
 Alternatively, we can write
 \be\label{eq:A+_FB}
 \A\oplus_F \B\, =\, S_1(F) \,=\, \Map_{\Delta^1}(\Delta^1, \Gamma(F)), 
 \ee
where $S_1(F)$  is the first level of the relative Waldhausen
S-construction of $F$, see \cite[\S 3.1C]{dkss-spher} and
$\Map_{\Delta^1}(\Delta^1, \Gamma(F))$ is the space of sections of the coCartesian Grothendieck construction.

\vskip .2cm

By definition,  objects of $\A\oplus_F \B$  are given by triples $(a,b,\eta)$ where $a \in \A$, $b \in \B$ are 
objects, and $\eta: F(a) \to b$ is a morphism in $\B$.
 We have embeddings of full $\oo$-subcategories
\be\label{eq:eps-a-eps-b}
\eps_\A: \A \lra \A\oplus_F \B, \,\,\, a\mapsto (a,0,0); \quad \eps_\B: \B \lra \A\oplus_F\B, \,\,\, b\mapsto (0,b,0)
\ee

\begin{prop}\label{prop:F=*epssb.epsa}
(a) The images of $\eps_\A$ and $\eps_\B$ (still denoted by $\A, \B \subset \A\oplus_F \B$), form an  SOD
$(\A, \B)$ of $\A\oplus_F \B$. Moreover, $\B$ is admissible in $\A\oplus_F \B$.

\vskip .2cm

(b) Conversely, let $\C$ be a stable $\oo$-category with an SOD $(\A, \B)$. If   $\B$ is  admissible , then
the functor $F= {^*}\eps_\B \circ \eps_\A[-1]: \A\to \B$ is exact and we have an equivalence $\C \= \A \oplus_F \B$,
compatible with $\eps\A$ and $\eps_\B$. 
 
\end{prop}

\noindent{\sl Proof:} This is Prop. 2.3.3(b) of \cite{dkss-spher}. 
To explain (b), note that each $c\in \C$ fits into a triangle obtained from \eqref{eq:triangle-delta}:
\[
a[-1] \buildrel \delta_c \over \lra b \lra c, \quad a = \eps_\A(^*\eps_\A)(c), \,\, b = \eps_\B\eps_\B^*(c)
\]
and 
$
\delta_c\in \Hom_\C( \eps_\A(^*\eps_\A)[-1](c), \eps_\B\eps_\B^*(c))
$
can identified, by adjunction, with an element  $\eta$ of
\[
\Hom_\B(^*\eps_\B\eps_\A [-1] (^*\eps_A(c)), \eps_\B(c)) \,\=\, \Hom_\B (F(^*\eps_\A(c)), \eps^*_\B(c))
\]
figuring in the definition of objects of $\A\oplus_F\B$. 

\qed

\vskip .1cm

 The category $\A \oplus_F \B$ together with its semi-orthogonal decomposition $(\A, \B)$
 is  known as the {\em SOD glued via $F$}. The functor $F$ is known as the {\em (coCartesian)
gluing functor}.


\subsection{$(\oo,2)$-categorical point of view: (op)lax (co)limits} \label{subsec:oplax}

\paragraph{(Op)lax (co)limits of $\oo$-categories.}\label{par:oplax-colim-cat-oo}
All $\oo$-categories, considered as simplicial sets, form, via  simplicial mapping spaces, an $\oo$-category
$\Cat_\oo$, see \cite{lurie:htt}. On the other hand, 
 the term ``$\oo$-category'' means, more precisely, an $(\infty, 1)$-category,
i.e,, a higher category in which all $p$-morphisms, $p\geq 1$, are (weakly) invertible. 
By allowing $2$-morphisms not to be weakly invertible, we get a concept of an $(\oo,2)$-category,
which can be defined, e.g., as a category enriched in $\Cat_\oo$, see \cite[\S 8.2]{CDW} and references therein,
esp.  \cite{GH}. This concept includes ``classical''  2-categories \cite{KS}. 
In particular, the $(\oo,1)$-category $\Cat_\oo$ can be refined to an $(\oo,2)$-category
$\CCat_\oo$ whose objects are $\oo$-categories, 1-morphisms are $\oo$-functors, but 2-morphisms
include non-invertible natural transformations.

\vskip .2cm

 By {\em diagram} in an $(\oo, 2)$-category $\DD$ we will mean a functor $q: C\to \DD$ from an ordinary category
 to $\DD$. Such a functor
 can potentially lead to {\em four limit constructions} (see \cite[\S 8.3]{CDW} for a summary): 
 \begin{itemize}
\item  lax limit $\laxlim (q)$ and  oplax limit $\oplaxlim (q)$; 
 \item lax colimit $\laxcolim (q) $ and
  oplax colimit $\oplaxcolim (q)$.  
\end{itemize}
As in the usual ($1$- or $(\oo,1)$-categorical) case these (co)limits are defined as vertices of universal cones
or cocones over the diagram. But in the $(\oo,2)$ case, the triangles in these (co)cones are required to be
commutative only up to a $2$-morphism for whose direction there are two possibilities. 
Such triangles, associated to a morphism $\phi: c\to c'$ in $C$,  are depicted in Fig. \ref{fig:four}.  

  \begin{figure}[h]
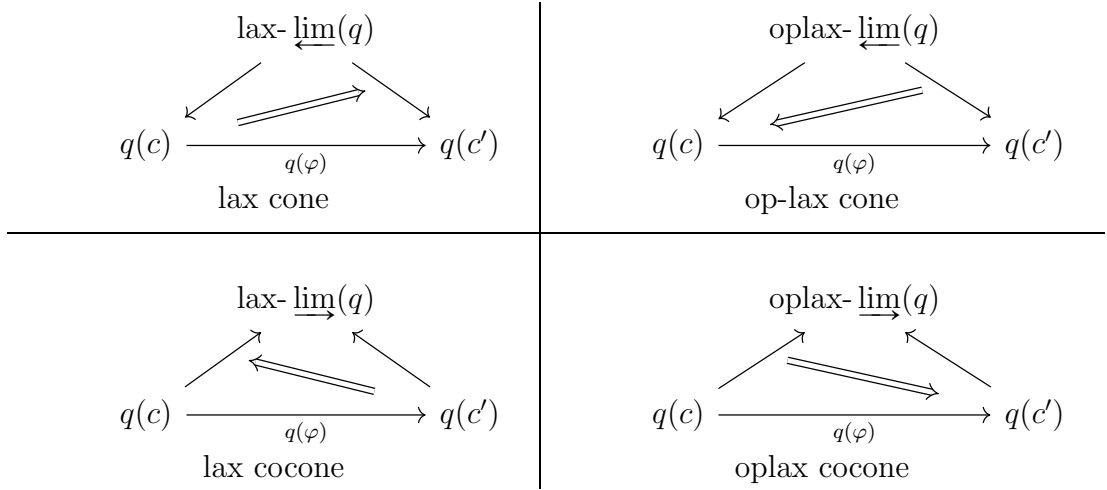

	\begin{center}
\begin{tabular}{c|c}
	\laxcone{q(c) }{\laxlim(q)}{q(c') }{}{}{q(\phi)}
	 &
    \oplaxcone[]{q(c)}{\oplaxlim(q)}{q(c')}{}{}{q(\phi)}
    \\
lax cone &    op-lax cone
	\\[1ex]\hline\\
  \laxcocone{q(c)}{\laxcolim (q)}{q(c')}{}{}{q(\phi)}
	&
    \oplaxcocone{q(c)}{\oplaxcolim(q)}{q(c')}{}{}{q(\phi)}
    \\ lax cocone & oplax cocone
\end{tabular}
\end{center}
\caption{Four  types of (co)cones and (co)limits.}\label{fig:four}
\end{figure}
We will  also use the notations $\laxlim^\DD(q)$ and $\laxlim^\DD_{c\in C} q(c)$ for $\laxlim(q)$, and similarly for other types
of (co)limits. 

\begin{ex}\label{ex:4lim-cat-oo}
  In  the case $\DD=\CCat_\oo$ the four (co)limits can be found explicitly  \cite[\S 8.3]{CDW}.
  For simplicity let us consider 
  a diagram given by an ordinary functor
 $q: C\to Set_\Delta$, 
 such that each $q(c), c\in C$, is an $\oo$-category. In this case: 
 \begin{itemize}

 \item[(1)]  
$\laxlim^{\CCat_\oo} (q) =\Map_{\C}(\C, \Gamma(q))$ is the category of sections of the covariant Grothendieck
construction $\pi: \Gamma(q)\to \C = \N(C)$. 

\item[(2)] 
 $\oplaxlim^{\CCat_\oo}(q) =\bigl(\laxlim^{\CCat_\oo} (q^\op)\bigr)^\op$, where $q^\op: C\to \Set_\Delta$
takes each $c\in C$ to $q(c)^\op$.

\item[(3)]    $\laxcolim^{\CCat_\oo} (q) = \chi(q)$ is the contravariant Grothendieck construction 
itself (not the category of sections!). 

  \item[(4)]
  $\oplaxcolim^{\CCat_\oo}(q) = \bigl(\laxcolim^{\CCat_\oo} (q^\op)\bigr)^\op$. 

\end{itemize}

\end{ex}

 \paragraph{(Op)lax (co)limits of stable $\oo$-categories and glued SODs.} 
Restricting to stable $\oo$-categories and  their {\em exact functors}, we get  an $(\oo,2)$-category $\SSt$
which is a (non-full) sub-$(\oo,2)$-category in $\CCat_\oo$, see \cite[\S 8.2]{CDW}. 
The construction of (op)lax (co)limits in $\SSt$ can be reduced to that in $\CCat_\oo$.
More precisely, the (op)lax {\em limit} of a diagram  in $\SSt$ is the same as in $\CCat_\oo$,
 (it is stable), while the (op)lax colimit in $\CCat_\oo$ of a diagram  in $\SSt$  is not necessarily
 stable, and the (op)lax colimit in $\SSt$ should be obtained from it by a kind of  ``stabiilization'' 
 procedure. 
 
 \vskip .2cm
 
 We will be interested in the case when the indexing category $C = \2$, so $q = q_F$ reduces to an exact $\oo$-functor
 $F: \A = q(0) \to\B = q(1)$ of stable $\oo$-categories. In this case the situation simplifies:

 \begin{thm}\label{thm:oo-2-add}
 All four limit constructions  for $q_F$ are canonically equivalent: 
 \[
 \laxlim (q_F)\, \=\, \laxcolim(q_F)\,  \=\, \oplaxlim (q_F)\,  \=\,\oplaxcolim(q_F) \, \= \, \A\oplus_F\B.
 \]
 
\end{thm}
 
 \noindent{\sl Proof:} This is \cite[Th. 2.4.2]{CDW}. 
 For future reference, we indicate the two of these identifications. 
 First, we exhibit $\A \oplus_F \B$ as the oplax colimit of $q_F$ via the cocone
\begin{equation}\label{eq:oplaxcocone}
		\begin{tikzcd}
			|[alias=S]|\B \ar[dr,"{b\mapsto (0,b, 0=F(0)\to b) = \eps_\B(b)}"] & \\
			&  \A \oplus_F \B\\
			\A \ar[uu, "F"] \ar[ur, ""{name=T}, swap, "{a\mapsto (a[-1], 0, F(a[-1]\to0) = \eps_\A(a)[-1]}"]  &
			\arrow[from=T, to=S, shorten <=10pt, shorten >=10pt, Rightarrow]
		\end{tikzcd}
\end{equation}
Second, we exhibit $\A \oplus_F \B$ as the lax limit of $q_F$ via the cone
\begin{equation}\label{eq:laxcone}
	\begin{tikzcd}
		& \B\\
		\A \oplus_F \B\ar[ur, 
		"{c=(a,b, F(a)\to b)\mapsto b = \eps_\B^*(c)}", ""{name=V}]\ar[dr, swap, 
		"{c=(a,b, F(a)\to b)\mapsto a = ({^*}\eps_\A)(c)}"] & \\
		&  \A \ar[uu, swap,"F"]\ar[to=V,shorten
		<=20pt, shorten >=10pt, Rightarrow].
	\end{tikzcd}
\end{equation}
Here $\eps_\A, \eps_\B$ are the embeddings from \eqref{eq:eps-a-eps-b}. 
\qed

\begin{rem}Theorem \ref{thm:oo-2-add}
 was interpreted in \cite{CDW} as the property of
{\em lax additivity} of $\SSt$. To explain this, recall that additivity at various categorical
levels can be expressed by saying that certain limits can also be seen as colimits and vice versa. 

\vskip .1cm

At the level  of an ordinary category, (semi-)additivity 
 means that finite products coincide with finite coproducts: $x\times y \= x\sqcup y =: x\oplus y$.

\vskip .1cm
 
At the level of an individual  $\oo$-category,  being stable can also be seen as a kind
of additivity property, but now with the ``sum'' (or, rather difference) of two objects $x$ and
$y$ depending on a choice of a morphism $f: x\to y$.  It can be defined in two ways:  either
as the fiber (limit) or as the cofiber (colimit) which are identified via \eqref{eq:fib=cof[-1]}. 

\vskip .1cm

At the next level of the   $(\oo,2)$-category $\SSt$, 
 additivity appears  as Theorem \ref{thm:oo-2-add}
with $\A\oplus_F\B$ playing the role of the direct sum of two stable $\oo$-categories $\A$ and $\B$
along $F$. 
\end{rem}


\subsection {Lax matrices}\label{subsec:laxmat}

\paragraph{Functors between glued SODs and lax matrices.} We specialize the construction of
\cite[\S 8.4]{CDW} to the situation of interest for us.

Let $F: \A_0\to\A_1$,  $G: \B_0\to\B_1$ be exact functors of stable $\oo$-categories and 
\[
	\phi: \A_0 \oplus_F \A_1 \lra \B_0 \oplus_G \B_1 
\]
 be an exact functor. By Theorem  \ref{thm:oo-2-add} we can write
 \[
  \A_0 \oplus_F \A_1  = \oplaxcolim_{i\in\2} \A_i, \quad  \B_0 \oplus_G \B_1 = \laxlim_{j\in\2} \B_j
 \]
 and use the corresponding (co)cones \eqref{eq:laxcone} and \eqref{eq:oplaxcocone} to construct the diagram
\begin{equation}\label{eq:butterfly}
	\begin{tikzcd}
		|[alias=S]| \A_1 \ar[dr]&  &  & \B_1\\
		& \A_0 \oplus_F \A_1 \ar[r, "\phi"] & \B_0 \oplus_G \B_1 \ar[ur, ""{name=V}]\ar[dr] & \\
		\A_0 \ar[uu, "F"] \ar[ur, ""{name=T}] &  &  & \B_0 \ar[uu, swap,"G"]\ar[to=V,shorten
		<=20pt, shorten >=10pt, Rightarrow].
		\arrow[from=T, to=S, shorten <=10pt, shorten >=10pt, Rightarrow]
	\end{tikzcd}
\end{equation}
 This allows us to represent $\phi$ as a kind of matrix consisting of 
  functors $\phi_{ij}: \A_i\to B_j$ together  with extra data. More precisely, we define the {\em stable $\oo$-category
  of lax matrices} of format $F\times G$ to be
  \be
  \LMat(F,G) \,=\, \oplaxlim^{\SSt}_{i\in\2}\,  \laxlim^{\SSt}_{j\in \2}\,  \ExFun(\A_i, \B_j). 
  \ee
  Theorem  \ref{thm:oo-2-add} implies: 
   \begin{prop}
  We have an equivalence of stable $\oo$-categories 
  \[
  \ExFun(\A_0\oplus_F \A_1, \B_0\oplus_G \B_1) 
  \= \LMat(F,G). \quad \qed
  \]
  \end{prop}
  
  Objects of  $\LMat(F,G)$  can be identified explicitly using parts (1) and (2) of  Example \ref{ex:4lim-cat-oo} in terms of the Grothendieck constructions of $q_F, q_G: \2\to\SSt$. This gives a system of data (called a {\em lax matrix})
\begin{equation}\label{eq:laxmatrix}
		\left(
		\begin{tikzcd}
			\phi_{00} \ar[r]\ar[d] & \ar[d]\phi_{01}\\
			\phi_{10} \ar[r] & \phi_{11}
		\end{tikzcd}
		\right)
\end{equation}
where $\phi_{ij}: \A_i\to\B_j$ are exact functors 
  given by post- and precomposing $\phi: \A_0\oplus_F \A_1\to \B_0\oplus_G \B_1$ with the
arrows from the lax (co)cones.  The arrows in the matrix signify natural transformations
\begin{align}\label{eq:naturalmatrix}
	\begin{split}
	\phi_{00} \Rightarrow    \phi_{10} \circ F, \quad 
	  \phi_{01} \Rightarrow   \phi_{11}  \circ F \\
	G\circ  \phi_{00} \Rightarrow \phi_{01}, \quad 
	G \circ \phi_{10} \Rightarrow \phi_{11}\\
	\end{split}
\end{align}
such that the square 
\begin{equation}\label{eq:coherentsquare}
	\begin{tikzcd}
		G \circ \phi_{00}\ar[d, Rightarrow] \ar[r, Rightarrow] & G \circ \phi_{10} \circ F \ar[d,
		Rightarrow]\\
		\phi_{01} \ar[r, Rightarrow] & \phi_{11} \circ F
	\end{tikzcd}
\end{equation}
in $\Fun(\A_0, \B_1)$ commutes (coherently).

\begin{rems}
(a) One can visualize the $4$ categories $\A_0, \A_1, \B_0, \B_1$ as the vertices of a tetrahedron
and the $6$ functors $F, G, \phi_{ij}$ as its edges. Then the $4$ natural transformations \eqref{eq:naturalmatrix}
fill the $2$-faces  and   \eqref{eq:coherentsquare} is interpreted
as $2$-dimensional commutativity of the tetrahedron. 

\vskip .2cm

(b) Sometimes we will denote a lax matrix simply by $\|\phi_{ij}\|$, understanding implicitly that the other data
are given. 
\end{rems}

Consider the instance of the triangle \eqref{eq:triangle-delta} for the SOD $(\B_0, \B_1)$:
\be\label{eq:triangle-delta-B}
\xymatrix{
\eps_{\B_0} (^*\eps_{\B_0}) [-1] \ar[rr]^{\quad \quad \delta_\B:= \delta_{(\B_0, \B_1)}}  &&   \eps_{\B_1}\eps^*_{\B_1} \ar[r] & \Id
}. 
\ee

\begin{prop}\label{prop:lax-matrix-expl}
(a) The matrix elements $\phi_{ij}$ associated to $\phi$ have the form:
\[
\begin{pmatrix}
\phi_{00}= \, {^*}\eps_{\B_0}\circ  \phi\circ  \eps_{\A_0} [-1],  & \phi_{01} = \eps_{\B_1}^* \circ \phi \circ \eps_{\A_0}[-1]
\\
\phi_{10} = \, {^*}\eps_{\B_0}\circ  \phi \circ \eps_{\A_1}, & \phi_{11} = \eps_{\B_1}^* \circ \phi \circ \eps_{\A_1}
\end{pmatrix}. 
 \]
 (b) In this identification, the arrows \eqref{eq:naturalmatrix} have the following form: the first two,
 \[
 \begin{gathered} 
 \phi_{00}  = {^*{\eps}}_{\B_0} \circ \phi\circ \eps_{\A_0}[-1]\buildrel e_*\over \longrightarrow {^*} \eps_{\B_0} \circ \phi \circ
 \eps_{\A_1} \circ (^*\eps_{\A_1}) \circ \eps_{\A_0}[-1] \,= \, \phi_{10}\circ F, \\
 \phi_{01} = \eps^*_{\B_1}\circ \phi \circ \eps_{\A_0}[-1] \buildrel e_* \over \longrightarrow \eps^*_{\B_1} \circ\phi\circ \eps_{\A_1}
 \circ {^*\eps}_{\A_1} \circ\eps_{\A_0}[-1]\, =\,  \phi_{11}\circ F
 \end{gathered} 
\]
 are induced by the unit $e: \Id\to\eps_{\A_1}\circ {^*\eps}_{\A_1}$, while the other two 
 are  compositions
  \[
  \begin{gathered}
  G\circ \phi_{00} \,=\, {^*\eps}_{\B_1} \circ \eps_{\B_0} \circ {^*\eps}_{\B_0} \circ \phi \circ \eps_{\A_0}[-2] 
  \buildrel \delta_*\over\lra {^*\eps}_{\B_1}\circ \eps_{\B_1}\circ  \eps^*_{\B_1} [-1] 
  \buildrel c_* \over \lra \eps_{\B_1}^*\circ\phi\circ \eps_{\A_0}[-1] \,=\,\phi_{01},
  \\
  G\circ\phi_{10} \,=\, {^*\eps}_{\B_1}\circ  \eps_{\B_0} \circ {^*\eps}_{\B_0} \circ\phi\circ \eps_{\A_1}[-1] 
  \buildrel \delta_*\over\lra  {^*\eps}_{\B_1}\circ\eps_{B_1} \circ \eps^*_{\B_1} \circ\phi\circ \eps_{\A_1}
  \buildrel c_* \over\lra \eps_{\B_1}\circ \phi\circ \eps_{\A_1}\,=\,\phi_{11}
  \end{gathered} 
   \]
   where $\delta_*$ is induced by $\delta_\B$ and $c_*$ by the counit ${^*\eps}_{\B_1}\circ  \eps_{\B_1}\to\Id$. 
\end{prop}

\noindent{\sl Proof:} It follows by
recalling from \eqref{eq:oplaxcocone} \eqref{eq:laxcone}
the nature of the arrows in the triangles of \eqref{eq:butterfly}, 
as well as  the identifications $F= \, ^*\eps_\B \circ \eps_A[-1]$, $G= \, ^*\eps_C\circ  \eps_B[-1]$ from
Proposition \ref{prop:F=*epssb.epsa}(b). 

\qed

Note the shift by $-1$ in the first row of the matrix. It is the consequence of the fact that 
we interpret the source of $\phi$
	as an oplax colimit while the target is interpreted differently,  as a lax limit. This convention
	has the advantage of simplifying the formulas for composition of lax matrices (below).  

\begin{exa}
	The identity functor
	$
		\id : \A_0 \oplus_F \A_1 \lra \A_0 \oplus_F \A_1 
	$
	is represented by the lax matrix
	\[
		\left(
		\begin{tikzcd}
			{[-1]}_{\A_0} \ar[r]\ar[d] & \ar[d] 0 \\
			0 \ar[r] & \id_{\A_1}.
		\end{tikzcd}
		\right)
	\]
		The square \eqref{eq:coherentsquare} has the form
	\[
		\begin{tikzcd}
			F [-1]_{\A_0} \ar[r]\ar[d] & \ar[d]0 \\
			0  \ar[r] & F 
		\end{tikzcd}
	\]
	  Its coherence data are those of a 
  biCartesian square  in $\Fun(\A_0,\A_1)$.  
\end{exa}

\paragraph{Composition of two lax matrices.} Let $F: \A_0\to \A_1$, $G: \B_0\to\B_1$, $H: \C_0\to\C_1$
be three exact functors of stable $\oo$-categories, and 
\begin{align*}
		 \A= \A_0 \oplus_F \A_1\buildrel \phi\over  \lra  \B= \B_0 \oplus_G \B_1 \buildrel\psi\over\lra
		 \C= \C_0 \oplus_H \C_1 
	\end{align*}
be two exact functors between their lax sums, represented by lax matrices
$\|\phi_{ij}\|$ and
	$\|\psi_{ij}\|$ as in \eqref{eq:laxmatrix}.

\begin{prop}\label{prop:comp-2-laxmat}
	 The lax matrix corresponding to the
	composite $\psi \circ \phi$ is given by the {\em lax matrix product}
	\begin{equation*}
		\left(
		\begin{tikzcd}
			\psi_{00} \ar[r]\ar[d] & \ar[d]\psi_{01}\\
			\psi_{10} \ar[r] & \psi_{11}
		\end{tikzcd}
		\right)
		\left(
		\begin{tikzcd}
			\phi_{00} \ar[r]\ar[d] & \ar[d]\phi_{01}\\
			\phi_{10} \ar[r] & \phi_{11}
		\end{tikzcd}
		\right)
	=
	\left(
		\begin{tikzcd}
			(\psi\phi)_{00} \ar[r]\ar[d] & \ar[d] (\psi\phi)_{01}\\
			(\psi\phi)_{10} \ar[r] & (\psi\phi)_{11}
		\end{tikzcd}
		\right),
	\end{equation*}
	with 
	\[
	(\psi\phi)_{ij} = \Cof \bigl\{ u_{ij}: \psi_{0j}\circ \phi_{i0} \lra \psi_{1j}\circ\phi_{i1}\bigr\}
	\]
	where the morphism $u_{ij}$ can be defined in one of the two equivalent ways:
 
	 (i) As the composite
	\[
		\psi_{0j} \circ \phi_{i0} \lra
		\psi_{1j} \circ G \circ \phi_{i0}\lra \psi_{1j} \circ \phi_{i1} 
	\]
	obtained via the natural tranformations from \eqref{eq:naturalmatrix}.
	
 (ii) As the $(i,j)$th of the morphisms
	\[
	\begin{gathered}
	\psi_{00}\circ \phi_{00}\, = \, {^*\eps}_{\C_0} \circ \psi \circ \eps_{\B_0}\circ  {^*\eps}_{\B_0}
	\circ  \phi \circ \eps_{\A_0}[-2]
	\lra {^*\eps}_{\C_0}\circ \psi \circ  \eps_{\B_1}\circ  \eps_{\B_1}^*\circ  \phi \circ \eps_{\A_0}[-1] \,=\,
	  \psi_{10}\circ \phi_{01},
	\\
	\psi_{01}\circ \phi_{00} \,=\, \eps_{\C_1}^*\circ \psi\circ \eps_{\B_0}
	\circ {^*\eps}_{\B_0} \circ\phi\circ \eps_{\A_0}[-2]\lra \eps_{\C_1}^*\circ\psi\circ\eps_{\B_1}\circ\eps_{\B_1}^* \circ\phi\circ\eps_{\A_0}[-1]\,=\, \psi_{11}\circ \phi_{01},
	\\
	\psi_{00}\circ \phi_{10}\,=\, {^*\eps}_{\C_0}\circ\psi\circ\eps_{\B_0}\circ {^*\eps}_{\B_0} \circ\phi\circ \eps_{\A_1}[-1]
	\lra {^*\eps}_{C_0}\circ\psi\circ\eps_{|B_1}\circ\eps_{\B_1}^*\circ\phi\circ\eps_{\A_1}
	\,=\, \psi_{10}\circ\phi_{11},
	\\
	\psi_{01}\circ\phi_{10} \,=\, \eps^*_{\C_1}\circ\psi\circ\eps_{\B_0}\circ{^*\eps}_{\B_0} \circ\phi\circ\eps_{\A_1}[-1]
	\lra \eps_{\C_1}^* \circ\psi\circ\eps_{\B_1}\circ\eps_{\B_1}^* \circ\phi\circ\eps_{\A_1} 
	\,=\, \psi_{11}\circ\phi_{11}
	\end{gathered}
	\]

induced by $\delta_\B: 	\eps_{\B_0}\circ\eps_{\B_0}[-1]\to \eps_{\B_1}\circ\eps_{\B_1}^*$ from 
 \eqref{eq:triangle-delta-B} 
	
 \end{prop}	
	
\noindent{\sl Proof:} 	With the statement (a), this is  variant of the computations from \cite {CDW}.
Alternatively, the statement with (b) is obtained directly by applying Proposition \ref{prop:lax-matrix-expl}(a) to $\psi\circ\phi$
and inserting the triangle of functors \eqref{eq:triangle-delta-B}  between $\psi$ and $\phi$ in the resulting formulas.
Finally, the equivalence of (a) and (b) follows from the detailed form of the morphisms in the lax matrices
given by  Proposition \ref{prop:lax-matrix-expl}(b). 
 \qed

\paragraph{Composition of several lax matrices.}

Next, let $F_i: \A_0^{(i)}\to \A_1^{(i)}$, $i=0, \cdots, N$, be $N+1$ exact functors between stable $\oo$-categories and
\[
\A^{(0)} = \A^{(0)}_0\oplus_{F_0} \A_1^{(0)} \buildrel \phi^{(1)}\over\lra  \A^{(1)} = \A_0^{(1)}\oplus_{F_1} \A_1^{(1)} 
\buildrel \phi^{(2)} \over\lra \cdots \buildrel \phi^{(N)}\over\lra\A^{(N)} =  \A_0^{(N)} \oplus_{F_n}\A_1^{(N)}
\]
be $N$ composable exact functors. Suppose that each $\phi^{(\nu)}$ is given by a lax matrix $\|\phi^{(\nu)}_{ij}\|$. 
Let $\phi = \phi^{(N)} \circ\cdots\circ\phi{(1)}$ be the composite functor. 
Just as in the product of $N$ usual $2\times 2$ matrices each matrix element is a sum of $2^{N-1}$ terms,
the matrix elements $\phi_{ij}: \A_i^{(0)} \to \A^{(n)}_j$ of $\phi$ can be found as totalizations of  certain 
$(N-1)$-cubes of functors. To describe them, consider the triangle of functors
\[
\eps_{\A^{(\nu)}_0} ({^*\eps}_{\A^{(\nu)}_0}) \buildrel \delta_{\A^{(\nu)}}\over\lra \eps_{\A^{(\nu)_1}} \eps_{\A^{(\nu)}_1}^*
\lra \Id_{\A^{(\nu)}}, \quad \nu=1,\cdots, N-1,
\]
an instance of \eqref{eq:triangle-delta}. Consider the morphism $ \delta_{\A^{(\nu)}}$ in this triangle as a
$1$-cube $D_\nu$ in $\ExFun(\A^{(\nu)}, \A^{(\nu)})$, so that $\Tot(D_\nu)=\Id_{\A^{(\nu)}}$. 
We have then the $(N-1)$-cube
\[
Q\,=\, \phi^{(N)} \circ D_{N-1}\circ \phi^{(N-1)}\circ D_{N-2}\circ \cdots \circ  \phi^{(2)} \circ D_1\circ \phi^{(1)}
\]
in $\Ex\Fun(\A^{(0)}, \A^{(N)})$, 
 obtained by inserting each $D_\nu$ between $\phi^{(\nu+1)}$ and $\phi^{(\nu)}$. Then $\Tot(Q)=\phi$. This
 implies the following.
 
 \begin{prop}\label{prop:comp-n-laxmat}
 We have $\phi_{ij}= \Tot(Q_{ij})$ where the $(N-1)$-cubes $Q_{ij}$ in $\ExFun(\A^{(0)}_i, \A^{(n)}_j)$ are
 obtained as: 
 \[
\begin{pmatrix}
Q_{00}= \, {^*}\eps_{\A^{(N)}_0}\circ  Q \circ  \eps_{\A^{(0)}_0} [-1], 
 & Q_{01} = \eps_{\A^{(N)}_1}^* \circ Q \circ \eps_{\A^{(0)}_0}[-1]
\\
Q_{10} = \, {^*}\eps_{\A^{(N)}_0}\circ  Q \circ \eps_{\A^{(0)}_1},
 & Q_{11} = \eps_{\A^{(N)}_1}^* \circ Q \circ \eps_{\A^{(0)}_1}
\end{pmatrix}. \quad\quad \qed
 \]
 \end{prop}
	
The edges of $Q_{ij}$ are obtained from the morphisms $u_{kl}$ in Proposition \ref{prop:comp-2-laxmat}
for  the consecutive compositions. The formulation of Proposition \ref{prop:comp-n-laxmat}
  exhibits also the higher coherence data for commutativity of these edges.


\section {Continuant complexes of iterated adjoints}\label{sec:cont-compl}

\subsection{Fibonacci cubes and their totalizations}\label{subsec:fibcubes}

\paragraph{The poset $\Cot[1,N]$ inside the $(N-1)$-cube.}
Consider the  $(N-1)$-cube, i.e., the 
poset $\2^{N-1}$ of all subsets in $[1,N-1]$. We can think of $\2^{N-1}$ as consisting of sequences
$\eps=(\eps_1,\cdots, \eps_{N-1})$, $\eps_i\in \2=\{0,1\}$ so that the subset
$J\subset [1,N-1]$ associated to such an $\eps$ is $J=\{i|\eps_i=1\}$. 

Call $\eps\in\2^{N-1}$ {\em separated}, if there are no two successive terms
$\eps_i, \eps_{i+1}$ which are both equal to $1$, i.e., if
$\sum_{i=1}^{N-2} \eps_i\eps_{i+1}=0$. Let $\Gamma_N\subset \2^{N-1}$
ne the set of all separated sequences. The poset $\Gamma_N$ (with the
order induced from $\2^{N-1}$) is known as the $n$th {\em Fibonacci cube}
\cite{hsu}. 

\begin{prop}
$\Gamma_N\subset \2^{N-1}$ is a left order ideal, i.e., if $J\in \Gamma_N$,
$J'\in \2^{N-1}$ and $J'\leq J$, then $J'\in\Gamma_N$.
\end{prop}

\noindent{\sl Proof:} If $\eps=(\eps_1,\cdots, \eps_{N-1})$ is a separated
$(0,1)$-sequence and $\eps'$ is obtained from $\eps$ by changing
some $\eps_i$ from $1$ to $0$, then $\eps'$ is separated as well. \qed

\vskip .2cm

We have an embedding $\alpha: \Cot[1,N] \hra \2^{N-1}$ which associated to a 
cotwinned $I\subset [1,N]$ the subset $J_I\subset [1,N-1]$ of
``missing twins'' in $I$, i.e.,
\[
J_I \,=\, \bigl\{ i\in [1,N-1] \,\bigl| \, \{i, i+1\} \subset [1,N]\- I\bigr\}. 
\]
Because the missing twins in $I$ do not intersect, the $(0,1)$-sequence
$\eps$ associated to $I$ is separated. This means that
$\alpha$ induces a bijection $\Cot[1,N]\to \Gamma_N$. 
In particular, $|\Gamma_N|=\phi_N$ is the $n$th Fibonacci number. 
We equip $\Cot[1,N]$ with the partial order, denoted $\leq$, which is induced
from $\Gamma_N$ (i.e., from the inclusion order on $\2^{N-1}$) via
the map $\alpha$. 

Thus $[1,N]\in\Cot[1,N]$ is the minimum element (i.e., an element which
is less of equal than any other).  For $N\geq 3$ the set $\Cot[1,N]$
does not have a maximum element (greater or equal than any other). 
In particular, when $N\geq 4$ is even,  $\emptyset\in\Cot[1,n]$
is not a maximum element. This can be seen from the next example
which also clarifies the following issue.

An inequality $I\leq I'$ in $\Cot[1,N]$ implies an inclusion of
subsets $I\supset I'$, but the converse implication does not hold in general. 

\begin{ex}
Let $N=4$. The order $\leq$ on $\Cot[1,4]$ is depicted in Fig. \ref{fig:Cot[1,4]}, with arrows
indicating elementary instances of $\leq$. In particular, the inequality
$\{1,4\}\leq \emptyset$ does not hold. This is because the corresponding
sets of missing twins (one twin $\{2,3\}$ for $\{1,4\}$, two twins
$\{1,2\}, \{3,4\}$ for $\emptyset$) are not included into one another. 

\begin{figure}[h]
\centering
\begin{tikzpicture}
\node(1234) at (0,0){$\{1,2,3,4\}$}; 
\node (14) at (0,1) {$\{1,4\}$}; 
\node (34) at (-2,1) {$\{3,4\}$}; 
\node (12) at (2,1) {$\{1,2\}$}; 
\node (0) at (0,2) {$\emptyset$}; 

\draw [->, line width = 0.8] (1234) -- (34); 
\draw [->, line width = 0.8] (1234) -- (14);
\draw [->, line width = 0.8] (1234) -- (12); 
\draw [->, line width = 0.8](34) -- (0); 
\draw [->, line width = 0.8] (12) -- (0); 

\end{tikzpicture}
\caption{The poset $(\Cot[1,4],\leq)$.} \label{fig:Cot[1,4]}
\end{figure}
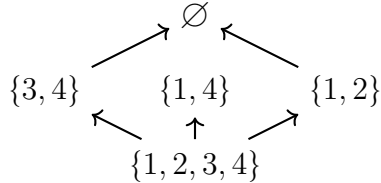
\end{ex}

We now explain, adapting an idea of  \cite{hsu}, how to lift the relation $\phi_N=\phi_{N-1}
+\phi_{N-2}$ for Fibonacci numbers to the level of posets. At the level of sets
this was done in the proof of Proposition \ref{prop:fibonacci}(a), so we need to
describe the partial orders. 

Let $f: A\to B$ be a monotone map of posets. We can consider $A,B$ as categories
and $f$ as a functor and form the contravariant Grothendieck construction
$\chi(f)$, see \S \ref{subsec:gen-oo-cat} above. The category $\chi(f)$ is again a poset, namely
the disjoint union $B\sqcup A$ with the order described as follows.
On the disjoint summands it is given by the original orders on $B$ and $A$.
For $b\in B, a\in A$ we have $b\leq a$ in $\chi(f)$ if and only if $b\leq f(a)$ in $B$.
The relation $a\leq b$ in $\chi(f)$ never holds. 

Consider now the embeddings
\be\label{eq:iota-n}
\iota_N: \Gamma_N \hra \Gamma_{N+1}, \quad (\eps_1,\cdots, \eps_{N-1})
\mapsto  (\eps_1,\cdots, \eps_{N-1},0). 
\ee

\begin {prop}
We have an isomorphism of posets $\Gamma_N \simeq \chi(\iota_{N-2}: \Gamma_{N-2}
\to\Gamma_{N-1})$. 
\end{prop}

Thus $\Gamma_N$ can be obtained from the ``seed embedding'' $\Gamma_1=\{0\}
\buildrel \iota_1\over\lra \Gamma_2 =\{0 < 1\}$ by repeatedly forming the 
Grothendieck construction of a monotone map, embedding the target of the map
there, forming the Grothendieck construction of this new embedding map
and so on. 

\vskip .2cm

\noindent {\sl Proof:} Inside $\Gamma_N$, we find a copy $\Gamma'_{N-1}$
of $\Gamma_{N-1}$ formed by $\eps=(\eps_1,\cdots, \eps_{N-1})$
with $\eps_{N-1}=0$, and a copy $\Gamma'_{N-2}$ of $\Gamma_{N-2}$
formed by $\eps$ with $\eps_{N-2}=0, \eps_{N-1}=1$. The induced orders on
these copies correspond to the standard orders on $\Gamma_{N-1}$
and $\Gamma_{N-2}$. As a set,
$\Gamma_N=\Gamma'_{N-1}\sqcup \Gamma'_{N-2}$. 
Further, let $\eps\in \Gamma'_{N-1}$ and $\zeta\in\Gamma'_{N-2}$.
Then the relation $\eps\leq\zeta$ in $\Gamma_N$ never holds because
the last digits do not allow it. The relation in the opposite order, namely
\[
\zeta=(\zeta_1, \cdots, \zeta_{N-2},0) \leq \eps = (\eps_1, \cdots\eps_{N-3}, 0,1)
\]
in $\Gamma_N$, is eqiuvent to the relation
\[
(\zeta_1, \cdots, \zeta_{N-2})\leq (\eps_1, \cdots, \eps_{N-3},0) =
\iota_{N-2}(\eps_1, \cdots, \eps_{N-3})
\]
in $\Gamma_{N-1}$. This matches the definition of the Grothendieck construction.
\qed

\paragraph{Totalization of Fibonacci cubes.} Let $\C$ be a stable $\oo$-category.
By a {\em covariant} (resp. {\em contravariant}) {\em $N$th Fibonacci cube in} 
$\C$ we will mean an $\oo$-functor $  \Gamma_N\to \C$,
resp. $\Gamma_N^\op\to\C$. The adjective ``covariant'' will be often
omitted. 

Let $\alpha: \Gamma_N\to\2^{N-1}$ be the standard embedding above. Given
a  Fibonacci cube $\gamma: \Gamma_N\to\C$ in $\C$, we can form the
right Kan extension  which is an $(N-1)$-cube $\alpha_*\gamma: \2^{N-1}\to\C$
in $\C$. Since
$\Gamma_N\subset \2^{N-1}$ is an left order ideal, $\alpha_*\gamma$ is
nothing but the extension of $\gamma$ by zero objects, i.e.,
\[
(\alpha_*\gamma)(\eps) = 
\begin{cases}
\gamma(\eps), & \text{if } \eps \text{ is separated}, 
\\
0, &\text{ otherwise}. 
\end{cases}
\]
Similarly, for a contravariant Fibonacci cube $\gamma: \Gamma_N^\op\to\C$
the left Kan extension $\alpha^\op_!\gamma: (\2^{N-1})^\op\to\C$ 
is a contravariant $(N-1)$-cube in $\C$ obtained by extending $\gamma$ by
zero objects. 

For a covariant (resp. contravariant) Fibonacci cube $\gamma$ in $\C$
we define the {\em totalization} of $\gamma$ as
the totalization of the corresponding cube obtained by extending $\gamma$ by zeroes,
i.e., 
 $\Tot(\gamma) = \Tot(\alpha_*\gamma)$, resp. 
  $\Tot(\gamma) = \Tot(\alpha^\op_!\gamma)$.

  Given an $N$th Fibonacci cube $\gamma: \Gamma_N\to\C$ in $\C$, we
  have the $(N-1)$st Fibonacci cube $\gamma': \Gamma_{N-1}\to\C$
  obtained by composing $\gamma$ with 
  the embedding $\iota_{N-1}:\Gamma_{N-1}\hra \Gamma_N$ given in
  \eqref{eq:iota-n} which adds $0$ at the end.
  We also have the $(N-2)$nd  Fibonacci cube $\gamma'': \Gamma_{N-2}\to\C$
  obtained by composing $\gamma$ with 
  the embedding    
  \[
  \wt\iota_{N-2}: \Gamma_{N-2}\to\Gamma_N, \quad \wt\iota_{N-2}
  (\eps_1, \cdots, \eps_{N-3}) = (\eps_1, \cdots, \eps_{N-3}, 0,1). 
  \]
  Similarly for contravariant Fibonacci cubes. 
  
  \begin {prop}[(The Fibonacci triangles)]\label{prop:fib-tri-1}
  For a covariant, resp. contravariant
   $n$th Fibonacci cube $\gamma$ in $\C$ we have a triangle
  \[
  \Tot(\gamma) \lra \Tot(\gamma') \lra \Tot(\gamma''), 
  \]
  resp.
  \[
  \Tot(\gamma) \lra \Tot(\gamma'') \lra \Tot(\gamma'). 
  \]
  \end{prop}
 
 \noindent{\sl Proof:} In the $(N-1)$ cube $Q$
 obtained by extending $\gamma$ by zeroes, the faces 
 $\del_{N-1}^0 Q$, $\del_{N-1}^1 Q$
 are precisely  the $(N-2)$-cubes obtained by extending $\gamma'$ and $\gamma''$
 by zeroes, so our statement follows from Proposition
 \ref{prop:Tot(Q)=iterfiber}. \qed
 
 
 \subsection {Continuant cubes associated to an adjoint string.} \label{subsec:cont-cubes}
 
 For $N\in \ZZ$ let $\ol N\in \{0,1\}$ be its residue modulo $2$. 
 
 Let $\A_0, \A_1$ be stable $\oo$-categories. By an {\em adjoint string of length $n$
 for} $(\A_0, \A_1)$ we will mean a sequence $(F_1, \cdots, F_N)$
 of  exact $\oo$-functors in alternate directions
 \[
 F_1: \A_1\lra \A_0, \quad F_2: \A_0\lra \A_1, \quad\cdots,\quad
 F_N: \A_{\ol N} \lra \A_{\ol{N+1}}
 \]
 with each $(F_i, F_{i+1})$ forming an adjoint pair. In particular, we have the
 unit and counit transformations
 \[
 e_i: \Id_{\A_{\ol i|}}\lra F_{i+1} F_i, \quad  c_i: F_i F_{i+1}\lra \Id_{\A_{\ol{i+1}}}, \,\,\,
 i=1,\cdots, N-1. 
 \]
 Further, let $I\subset [1,N]$ be a cotwinned subset. Writing $I=\{i_1 <\cdots < i_k\}$,
 we associate to an adjoint string $(F_1,\cdots, F_n)$ the compositions
 \[
 F_I = F_{i_1}\cdots F_{i_k}: \A_{\ol N} \lra \A_0, \quad
 F^I= F_{i_k} \cdots F_{i_1}: \A_1\lra \A_{\ol{N+1}}. 
 \]
 Here the fact that the compositions make sense and have the source and target
 as indicated, follows from the assumption that $I$ is cotwinned. 
 
 Let $I\leq I'$ be an elementary inequality of cotwinned subsets, i.e., $I'\subset I$
 is obtained by removing one more twin $\{i_\nu, i_{\nu+1}=i_\nu+1\}$. 
 We have the insertion and contraction transformations
 \[
 \begin{gathered}
 e_{I'I} \,= \, F_{i_k} \cdots F_{i_{\nu+2}}\circ_0 e_{i_\nu}
 \circ_0 F_{i_{\nu-1}}\cdots F_{i_1}: F^{I'} \lra F^I,
 \\
 c_{II'}\,=\, F_{i_1}\cdots F_{i_{\nu-1}}\circ_0 c_{i_\nu} \circ_0 
 F_{i_{nu+2}}\cdots F_{i_k}: F_I \lra F_{I'}. 
 \end{gathered}
 \]
  They define a contravariant and a covariant representation of the $1$-skeleton
  of the Fibonacci cube $\Cot[1,N]\=\Gamma_n$ in the corresponding functor
  categories. We next explain how these skeleton representations extend to
  full-fledged  contravariant and covariant Fibonacci cubes.
 Let us  start with some general remarks.
  
  \vskip .2cm
  
  Let $(S,\leq)$ be a poset. For $s\in S$ we denote by
  \[
  S^{\leq s} \,=\,\{t\in S|\,  t\leq s\}, \quad S^{\geq s}\,=\, \{t\in S|\,  t\geq s\}
  \]
  the lower and upper interval of $s$. Clearly, $S$ is the inductive limit (in the
  category of sets) of the $S^{\leq s}$, as well as of the $S^{\geq s}$.
  Moreover, at the level of nerves we have
  \[
  \N(S)\,=\,\varinjlim_{s\in S}\nolimits^{\Set_\Delta} \N(S^{\leq s}) \,=\,
  \varinjlim_{s\in S^\op}\nolimits^{\Set_\Delta} \N(S^{\leq s}). 
  \]
  This is because each simplex of $\N(S)$, i.e., each chain $(s_0\leq\cdots \leq s_p)$
  lies in the nerve of some lower interval, namely of $S^{\leq s_p}$,
  as well as in the nerve of some upper interval, namely of $S^{\geq s_0}$.
  Let $\C$ be any $\oo$-category. It follows that a datum of an $\oo$-functor
  $S\to \C$, i.e., of a morpism of simplicial sets $\N(S)\to\C$, is equivalent to
  the datum of a compatible system of $\oo$-functors $S^{\leq s} \to \C$
  as well as to the datum of a compatible system of $\oo$-functors
  $S^{\geq s}\to\C$. 
  
  \vskip .2cm
  
  We apply this observation to $S=(\Cot[0,N],\leq)$ which is identified with
  the Fibonacci cube $\Gamma_N$. 
  Thus any lower interval $\Cot[0,N]^{\leq I}$
  is naturally a $J_I$-cube, where $J_I \subset[1,N-1] $ is the set of
  $j$ such that the twins $\{j, j+1\}$ form a partition of $\ol I :=[1,N] \- I$. 
  To define an $\oo$-functor $\Cot[1,N]\to \C$ is therefore the same as to define
  a compatible system of $J_I$-cubes $\Cot[1,N]^{\leq I} \to\C$.
  Similarly for an $\oo$-functor $\Cot[1,N]^\op\to\C$. 
  
  \vskip .2cm
  
  Let us now return to the situation of an adjoint string $(F_1,\cdots, F_N)$. 
  Given $I\in\Cot[0,N]$, we write $[0,N]\- I$ as the union of twins
  \[
  [0,N]\- I \,=\, \bigl\{ j_1, j_1+1 < j_2, j_2+1 < \cdots < j_p, j_p+1\bigr\},
  \]
  so that $J_I = \{j_1, \cdots, j_p\}$ and $I$ becomes partitioned into
  the twins and (possibly empty) intervals between them: 
  \[
  [1,N] \,=\, [1, j_1-1] \cup \{j_1, j_1+1\} \cup [j_1+2, j_2-1] \cup \{j_2, j_2+1\}
  \cup\cdots\cup \{j_p, j_p+1\}\cup [j_{p}+2, N]. 
  \]
  For every subinterval $[a,b]\subset [1,N]$, $a\leq b$,  we write
  \[
  F_{[a,b]}\,=\, F_a F_{a+1}\cdots F_b, \quad F^{[a,b]} \,=\, F_b F_{b-1} \cdots F_a, 
    \]
 setting also  $F_{[a,b]}=F^{[a,b]}:=\Id$ for  $a>b$ (empty interval).
 
Let us consider the unit and counit arrows $e_j, c_j$ as $1$-cubes of functors ${\bf e}_j$, ${\bf c}_j$. 

  With these preparations, we define the $J_I$-cube in  $\Fun( \A_{\ol N}, \A_0)$
  \[
  \Ec_N^{\leq I}(F_1,\cdots, F_N) \,=\, 
  F_{[1, j_1-1]}\circ {\bf c}_{j_1} \circ F_{[j_1+2, j_2-1]} \circ {\bf c}_{j_2}
  \circ \cdots\circ  {\bf c}_{j_p}\circ  F_{[j_p+2, n]}
  \]
  by inserting the  $1$-cube ${\bf c}_{j_\nu}$
  between $F_{[j_{\nu-1}+2, j_\nu -1]}$
  and $F_{[j_\nu+2, j_{\nu+1}-1]}$.

  Similarly, we define a contravariant $J_I$-cube in $\Fun(\A_1, \A_{\ol{N+1}})$
  \[
  \Ec^N_{\leq I}(F_N,\cdots, F_1) \,=\, F^{[j_p+2,N]}\circ  {\bf e}_{j_p}\circ 
  F^{[j_{p-1}+2, j_p-1]} \circ \cdots \circ  {\bf e}_{j_1}\circ F^{[1, j_1-1]}
  \]
  by inserting the $1$-cubes ${\bf e}_{j_\nu}$ given by the unit arrows. 
  It is clear by construction that for different $I\in\Cot[1,N]$ these cubes
  are compatible to each other and so define the Fibonacci cubes
  \[
  \begin{gathered}
  \Ec_N(F_1,\cdots, F_N) \,=\,\varinjlim\nolimits_{I} \, \Ec_N^{\leq I}(F_1,\cdots, F_N):
  \Cot[1,N] \lra\Fun(\A_{\ol N}, \A_0), 
  \\
    \Ec^N(F_N,\cdots, F_1) \,=\,\varinjlim\nolimits_{I} \, \Ec^N_{\leq I}(F_N,\cdots, F_1):
  \Cot[1,N]^\op  \lra\Fun(\A_1, \A_{\ol{N+1}}),
  \end{gathered}
  \]
  with $\varinjlim_I$ meaning the Fibonacci cube assembled out of the ordinary
  cubes. 
  We call them the $N$th {\em homological and cohomological
  continuant cubes} associate to $(F_1,
  \cdots, F_N)$. They are categorical analogs of the Euler polynomials
  $E_N(x_1,\cdots x_N)$ and $E_N(x_N, \cdots,x_1)$ respectively. 
  
  \begin{exas}\label{exas:fibcubes}
  (a) For $N=2$
  \[
  \Ec_2(F_1, F_2) \,=\,\{ F_1F_2 \buildrel c_1\over \lra \Id\}, \quad \Ec^2(F_2, F_1)
  \,=\, \{\Id\buildrel e_1\over \lra F_2 F_1\}
  \]
  are the counit and unit arrows ($1$-cubes).
  
  \vskip .2cm
  
  (b) For $N=3$ the Fibonacci cube $\Ec_3(F_1, F_2, F_3)$
  resp.  $\Ec^3(F_3, F_2, F_1)$ is the amalgamation of two counit
  resp. unit arrows:
  \[
  \xymatrix{
  F_1F_2F_3
  \ar[d]_{c_1\circ_0 F_3}
  \ar[rr]^{F_1\circ_0 c_2}&& F_1, 
  \\
  F_3 &&
  }
  \quad\quad
  \xymatrix{
  F_3F_2F_1&& \ar[ll]_{e_2\circ_0 F_1} F_1.
  \\
  F_3\ar[u]^{F_3\circ_0 e_2} &&
  }
  \]
 (c) For $N=4$ the Fibonacci cube $\Ec_4(F_1, F_2, F_3, F_4)$
 is the amalgamation of the $1$-cube (arrow) on the left and 
 the $2$-cube on the right:
 \[
 \xymatrix{
 F_1F_4&& \ar[ll]_{F_1\circ_0 c_2\circ_0 F_4} F_1F_2F_3F_4
 \ar[d]_{F_1F_2\circ_0 c_3}
 \ar[rr]^{c_1\circ_0 F_3F_4}&& F_3F_4
 \ar[d]^{c_3}
 \\
 && F_1F_2 \ar[rr]_{c_1}&&\Id. 
 }
 \]
 (d) For $N=5$ the Fibonacci cube $\Ec_5(F_1,\cdots, F_5)$ iis the amalgamation
 of three $2$-cubes,  each given by a pair of independent counits:
 \[
 \xymatrix{
 & F_1 F_4 F_5 \ar[r]&F_1
 \\
 F_1 F_2 F_5 \ar[d] & \ar[l] \ar[d] F_1F_2F_3F_4F_5 \ar[u] \ar[r]& F_1F_2F_3 \ar[u] \ar[d]
 \\
 F_5&\ar[l] F_3F_4F_5 \ar[r]& F_3. 
 }
 \]
  \end{exas}


\subsection {Higher twists and cotwists of a functor.}\label{subsec:hi-twist}

\paragraph{Definitions. Fibonacci triangles.}

Let $F: \A_1\to \A_0$ be an exact $\oo$-functor between stable $\oo$-categories
and $N\geq 1$. 
Assume that $F$ has iterated adjoints $F^*, F^{**}, \cdots, F^{* (N-1)}$. 
We define the $N$th {\em twist} $\EE_N(F)$
and $N$th {\em  cotwist} $\EE^N(F)$ to be the functors
\be
\EE_N(F) = \Tot\,  \Ec_N(F, F^*,\cdots, F^{*(N-1)}), 
\quad
\EE^N(F) = \Tot\,  \Ec^N(F^{*(N-1)},\cdots,  F^*, F). 
\ee
We also put $\EE_0(F)=\EE^0(F)=0$.
Thus, for small values of $N$ we have:
\[
\begin{gathered}
\EE_1(F)=\EE^1(F)=F,
\\
\EE_2(F) = \Fib\{ FF^*\buildrel c_F\over\lra \Id\}, 
\quad \EE^2(F) = \Cof\{ \Id\buildrel e_F\over\lra F^* F\}, 
\\
\EE_3(F) = \Fib\bigl\{ 
\xymatrix{F F^* F^{**} \ar[rrr]^{ (F\circ c_{F^*}, \,\, \,
c_F\circ F^{**} )
 }
&&&
  F\oplus F^{**} 
}
\bigr\}, 
\\
\EE^3(F) = \Cof \bigl\{ 
\xymatrix{
  F\oplus F^{**} 
  \ar[rrr]^{ (F\circ e_{F^*})
 +
(e_F\circ F^{**} ) 
 }
 &&& F F^* F^{**}
}
\bigr\}. 
\end{gathered}
\]

Applying Proposition \ref{prop:fib-tri-1}, we get   exact triangles which
we also call the {\em Fibonacci triangles}
\be\label{eq:fib-tri-2}
\begin{gathered}
\EE_N(F)\buildrel\beta_{N,F}\over  \lra \EE_{N-1}(F) F^{*(N-1)} \buildrel\alpha_{N,F}\over\lra \EE_{N-2}(F), 
\\
\EE^{N-2}(F)\buildrel{\rho_{N,F}}\over \lra F^{*(N-1)}  \EE^{N-1}(F)\buildrel \pi_{N, F}\over \lra \EE^N (F). 
\end{gathered}
\ee

Let us also note the following. 

\begin{prop}\label{prop:E_N-dual-E^N}
$\EE^N(F^*)$ is identified with $(\EE_N(F))^*$, the right adjoint to $\EE_N(F)$. 
\end{prop}

\noindent{\sl Proof:} Recall that for any natural transformation $u: G\to H$ of $\oo$-functors
we have the {\em right transpose transformation} $u^t: H^*\to G^*$ of their right adjoints,
if the right adjoints exist. 
Further, we have:

\begin{lem}
(a) The cofiber of $u^t$ is the right transpose of the fiber of $u$: 
\[
\Cof\{u^t: H^*\lra G^*\} \,\= \, \Fib\{ u: G\lra H\}^*. 
\]
(b) If $G$ has two adjoints $G^*, G^{**}$, then the unit $e_{(G^*, G^{**})}:  \Id\to G^{**} G^*$
of the adjunction $(G^*, G^{**})$ is the right transpose of the counit $c_{(G, G^*)}: G G^*\to \Id$
of the adjunction $(G, G^*)$.   \qed

\end{lem}

Now applying (b), we see that the  entire Fibonacci cube $\Ec^N(F^{(N)}, \cdots, F^*)$ is the right adjoint
of $\Ec_N(F, \cdots, F^{(N-1)})$. Appling (a) repeatedly, we see that the totalizations of these
adjoint cubes are adjoint. 
Proposition \ref{prop:E_N-dual-E^N} is proved. \qed

\paragraph{Determinantal triangles.} For each $N\geq 1$ we can combine the morphisms from
the Fibonacci triangles \eqref{eq:fib-tri-2} to get a square of functors
\be\label{eq:E-E-square}
\xymatrix{
\EE_{N+1}(F) \EE^{N-1}(F^*)
\ar[d]_{\EE_{N+1}(F)\circ \rho_{N+1}, F^*} 
 \ar[rrr]^{\beta_{N+1, F}\circ \EE^{N-1}(F^*)} &&& \EE_N(F)F^{(N)} \EE^{N-1}(F^*)
 \ar[d]^{\EE_N(F)\circ \pi_{N, F^*}}
\\
\EE_{N+1}(F) F^{(N+1)}\EE^N(F^*) \ar[rrr]_{\alpha_{N+1, F}\circ \EE^N(F^*)} &&& \EE_N(F) \EE^N(F^*)
}
\ee

\begin{prop}\label{prop:E-E-square-com}
The square \eqref{eq:E-E-square} is coherently commutative.
\end{prop}

\noindent{\sl Proof:} Follows from coherent compatibility of the units and counits of various adjunctions.
\qed

\vskip .2cm

Denote $a_N: \EE_{N+1}(F) \EE^{N-1}(F^*) \to \EE_N(F) \EE^N(F^*)$ the common composition
of the two paths in  \eqref{eq:E-E-square}. 

\begin{prop}\label{prop:det-triang}
(a)  $\Cof(a_N)$ is identified with $\Id$, the identity functor, so we have a triangle of functors
\[
\EE_{N+1}(F ) \,  \EE^{N-1}(F^*) \buildrel a_N\over \lra
\EE_N(F )\,  \EE^N(F^*)   \lra \Id. 
\]
(b) Similarly,  we have a triangle of functors
\[
\Id   \lra \EE^N(F^*) \,\EE_N(F)\buildrel b_N\over \lra \EE^{N+1}(F) \,\EE_{N-1}(F^*). 
\]
\end{prop}

The triangles from the proposition will be  called {\em determinantal triangles}, since they categorify the
identities among Euler polynomials given in Proposition \ref{prop:euler-recurs}(d) and Remark 
\ref{rem:Euler-det-id}, these identities in turn expressing the fact that the determinants of
certain $2\times 2$ matrices are equal to $1$. 

\vskip .2cm

\noindent {\sl Proof of  Proposition \ref{prop:det-triang}:}  We prove (a), since (b) is then obtained
by invoking  Proposition \ref{prop:E_N-dual-E^N} and passing to the left adjoint of the instance of (a)
for $F^*$ instead of $F$.

\vskip .2cm

To see (b), consider the commutative diagram in which the  $3$-term rows and columns are triangles
and the central row and column are obtained from the Fibonacci triangles:
\[
\xymatrix{
 & \EE_{N+1}(F) \EE^{N-1}(F^*) \ar[d] \ar[r]^{a_N} & \EE_N(F) \EE^N(F^*)\ar[d]^{=}
 \\
 \EE_{N+2}(F) \EE^N(F^*) \ar[d]_{a_{N+1}}
 \ar[r]& \EE_{N+1}(F) F^{(N-1)} \EE^N(F^*)\ar[d] \ar[r]& \EE_N(F) \EE^N(F^*)\ar[d]
 \\
 \EE_{N+1}(F)\EE^{N+1}(F^*) \ar[r]^{=} & \EE_{N+1}(F) \EE^{N+1}(F^*) \ar[r]& 0.
}
\]
The commutativity of this diagram follows from Proposition \ref{prop:E-E-square-com}, since
$a_N$ is obtained as one type of composition while $a_{N+1}$ as the other type. 
Now, the octohedral axiom implies that $\Fib(a_{N+1}) \= \Fib(a_N)$, so $\Cof(a_{N+1}) \= \Cof(a_N)$.
But $\Cof(a_1)=\Id$, so we obtain our statement by induction. \qed


\section{$N$-spherical functors and periodic SODs}\label{sec:N-spher}

\subsection{$N$-spherical functors}\label{subsec:N-spher}

\paragraph{Definition via vanishing of higher (co)twists.} 

\begin{defi}
Let $F: \A_0\to\A_1$ be an exact functor of stable $\oo$-categories and $N\geq 2$. We say that $F$ is
{\em $N$-spherical}, if the adjoints $F^*, \cdots, F^{(N-1)}$ exist and the $(N-1)$th twist and cotwist
of $F$ vanish: $\EE_{N-1}(F) = \EE^{N-1}(F) = 0$. 
\end{defi}

\begin{exas}
A $2$-spherical functor is a zero functor.

A $3$-spherical functor is an equivalence of categories.

A $4$-spherical functor is the same as a spherical functor in the usual sense \cite{AL17, dkss-spher}.
This characterization of usual spherical functors is due to Kuznetsov \cite[Prop. 2.13]{kuznetsov}
who worked in the context of dg-enhancements and Fourier-Mukai functors. His argument adapts
directly to the stable $\oo$-categorical situation. 
\end{exas}

\begin{prop}\label{prop:spher-tw-eq}
If $F$ is $N$-spherical, then $\EE_{N-2}(F)$ and $\EE^{N-2}(F)$ are equivalences of categories.
\end{prop}

\noindent{\sl Proof:} This is a direct consequence of the determinantal triangles from Proposition 
 \ref{prop:det-triang}. \qed
 
 \vskip .2cm
 
 Note that for $N$ even $\EE_{N-2}(F): \A_1\to \A_1$ and $\EE^{N-2}(F): \A_0\to \A_0$ are self-equivalences,
 while for $N$ odd they are cross-equivalences $\EE_{N-2}(F): \A_0\to \A_1$, $\EE^{N-2}(F): \A_1\to \A_0$.

 \paragraph{ Various characterizations and features of $N$-spherical functors (list of results).}\label{par:N-spher-list}
 
 We now list various properties of $N$-spherical functors whose proofs, if not provided
 right away, will be given in \S\ref {subsec:N-spher-peri}  \ref{par:N-spher-pf} This postponement will be indicated by the sign $\Diamond$. 
 
 \begin{prop}\label{prop:N-sph-NM-sph}
 Let $F$ be an $N$-spherical functor. Then:
 
 \vskip .2cm
 
 (a) $F$ has all adjoints $F^{(i)}$, $i\in \ZZ$. Each $F^{(i)}$ is $N$-spherical. 
 
 \vskip .2cm
 
 (b) $F$ is $NM$-spherical for any $M\geq 1$. 
 $ \quad\quad \Diamond$
 \end{prop}
 
 We have the following analog of the fact that composition with the twist and cotwist of a usual spherical functor
 change left adjoints to right adjoints. 
 
 \begin{prop}
 Let $F$ be $N$-spherical. Then 
 \[
 \EE_{N-2}(F)\, F^{(N-2)} \= \EE_{N-3}(F) \text{ and }
 F^{(N-2)} E^{N-2}(F) \= \EE^{N-3}(F).
 \]
 \end{prop}
 
 \noindent{\sl Proof:} Follows at once for the Fibonacci triangles \eqref{eq:fib-tri-2}. \qed
 
 \vskip .2cm
 
For $N=4$ (a usual spherical functor) we have $\EE_1(F)=\EE^1(F) = F$, so the statements become
 \[
 \EE_2(F) \, F^{**} \,\=\, F, \quad F^{**} \, \EE^2(F) \,\=\, F. 
 \]
 
 A situation of natural interest is when the category $\A_0$ is ``simple'' (e.g., coincides with
  the $\oo$-enhancement
 of the derived category of vector spaces) while $\A_1$ is ``complicated'' and we want to
 construct  a self-equivalence of $\A_0$ as $\EE_{N-2}(F)$ for an $N$-spherical functor $F: \A_0 \to \A_1$.
 For this $N$ must be even. In this case we give an alternative criterion for $N$-sphericity which
 uses only functors $\A_0\to \A_1$ and so is easier to check.
 
 \begin{prop}\label{prop:spher=triang}
 Let $N$ be even. Then the following properties of $F: \A_0\to\A_1$ are equivalent:
 \begin{itemize}
 \item[(i)] $F$ is $N$-spherical.
 
 \item[(ii)] $\EE^{N-2}(F): \A_0 \to \A_0$ is an equivalence while $\EE^{N-1}(F): \A_0 \to \A_1$ is zero.   $ \quad\quad \Diamond$
 \end{itemize}
 
 \end{prop}
 
 
 \subsection{Periodic SODs}\label{subsec:N-spher-peri}
 
 \paragraph{The statement of result.} 
 
 Let $F: \A_0\to \A_1$ be an exact functor of stable $\oo$-categories. As explained in Proposition \ref{prop:F=*epssb.epsa},
 it gives to a new stable $\oo$-category $\A = \A_0\oplus_F \A_1$ with a semi-orthogonal decomposition $(\A_0, \A_1)$ so that,
 in particular, $\A_1 = {^\perp \A_0}$ and $\A_0=\A_1^\perp$.

\begin{thm}\label{thm:N-spher=per}
 Let $N\geq 2$ and  assume that $F$ has iterated adjoints $F^*, \cdots, F^{N-1}$. 
Then the following are equivalent:
\begin{itemize}
\item[(i)] $F$ is an $N$-spherical functor.

\item[(ii)] The chain of orthogonals $\A_0, \A_1= {^\perp \A}, \A_2:=  {^{\perp\perp}\A}, \cdots$ in $\A_0\oplus_F\A_1
$ is $N$-periodic, 
i.e., $\A_N:= {^{\perp N}\A_0}$ coincides with $\A_0$. 
\end{itemize}
\end{thm}

\noindent The proof will occupy the rest of this subsection until \S  \ref{par:N-spher-pf}

 \paragraph{The coordinate change lax matrix associated to a mutation.} We recall some finer features of
 the theory of SODs from \cite [\S 2.2.-2.5]{dkss-spher}. 
 
 Without assuming the existence of any adjoints,
 the standard SOD $(\A_0, \A_1)$ of $\A = \A_0\oplus_F\A_1$ is coCartesian, in terminology of
 \cite[Def. 2.2.8]{dkss-spher}. In particular, $(\A_1, \A_2 = {^\perp \A}_1)$ is again an SOD, now Cartesian,
 and we have the mutation equivalence $\rho_{\A_1}: \A_0\to \A_2$. 
 
 \vskip .2cm
 
 Assume now that $F$ has a right adjoint $F^*$. Then $(\A_1, \A_2)$ is also coCartesian
 \cite[Prop. 2.5.1]{dkss-spher} and so has a coCartesian gluing functor $\Phi: \A_1\to\A_2$ giving
 an identification of $\A$ with $\A_1\oplus_\Phi \A_2$. 
 Further, $\Phi\rho_{\A_1}$ is identified with $F^*$, so $\A_1\oplus_\Phi \A_2$, being obviously identified
 with $\A_1\oplus_{\Phi\rho_{\A_1}} \A_0$, is now identified with $\A_1\oplus_{F^*} \A_0$. 
 The composite equivalence
 \be
 \mu_{F^*}: \A_0\oplus_F \A_1 \lra \A \lra \A_1\oplus_{F^*}\A_0
 \ee
 will be called the {\em mutation coordinate change} (for $F^*$). 
 
 \vskip .2cm
 
 Recall that objects of $\A_0\oplus_F \A_1)$ have the form $(a_0, a_1, \eta: F(a_0)\to a_1)$,
 while objects of $\A_1\oplus_{F^*} \A_0$ have the form $(a'_1, a'_0, \xi: F^*(a'_1)\to a'_0)$, 
 with $a_i, a'_i\in\A_i$. In this notation,
 \be\label{eq:mu-F-formula}
 \mu_{F^*} \bigl(a_0, a_1,  f(a_0) \buildrel\eta\over\to a_1\bigr)\,=\,
 \biggl( a_1, \Cof\bigl(a_0\buildrel\eta'\over\to F^*(a_1)\bigl),  F^*(a_1)\buildrel {\on{can}}\over\to \Cof \bigl(a_0\buildrel
 \eta'\over\to F^*(a_1)\bigr)\biggr), 
 \ee
 where $\eta'$ corresponds to $\eta$ by adjunction and $\on{can}$ is the canonical morphism of the target of
 $\eta'$ to its cofiber.  Since $\mu_F$ is a functor between two catgories with semi-orthogonal decompositions,
 we can apply the formalism of lax matrices from \S \ref{subsec:laxmat}, and  \ref{eq:mu-F-formula}
 implies at once:
 
 \begin{prop}\label{prop:laxmat-mu}
 The equivalence $\mu_{F^*}$ 
 is represented by the lax matrix
	\[
		 \begin{pmatrix}
		 \xymatrix{
			0  \ar[r]\ar[d] & \Id_{\A_0}\ar[d]^{e}
			\\
			\Id_{\A_1} \ar[r]_{\Id}& F^*
			}
	  \end{pmatrix}.
	\]
	Here, using the conventions of \eqref{eq:naturalmatrix},  and the arrow labelled $e$ represents the unit transformation
	$e: \Id_{\A_0}\to F^* F$ and the arrow labelled $\Id$ represents the transformation $\Id_{F^*}: F^*\to F^*$. 
  \qed
	
	\end{prop}

	\paragraph{Product of mutation matrices.} Assume now that $F$ has adjoints $F^*, \cdots, F^{(N)}$.
	We then have equivalences
	\[
	\A_0\oplus_F\A_1\buildrel \mu_{F^*}\over\lra \A_1\oplus_{F^*}\A_0
	\buildrel \mu_{F^{**}}\over\lra \A_0\oplus_{F^{**}} \A_1 \buildrel \mu_{F^{***}}\over\lra\cdots 
	\buildrel \mu_{F^{(N)}}\over\lra \A_{\ol N} \oplus_{F^{(N)}}\A_{\ol{N+1}}.
	\]
	Let $\mu_N =\mu_{N, F^*} =  \mu_{F^{(N)}}\circ\cdots\circ \mu_{F^*}$ be the composite equivalence. 
	
	To formulate the next statement, let 
	\[
	\pi_N=\pi_{N,F}: F^{(N-1)} \EE^{N-1}(F) \lra \EE^N(F)
	\]
	be the second map in the second Fibonacci triangle in \eqref{eq:fib-tri-2}. Let also
	\[
	e_N= e_{N,F}: \EE^{N-1}(F^{**}) \lra \EE^N(F^*) F
	\]
	be  induced by the morphism of the corresponding Fibonacci cubes which takes
	the vertex $F^I = F^{(i_1)} \cdots F^{(i_p)}$ for a cotwinned $I=\{i_1 >\cdots > i_p \geq 2\}$
	to the vertex  $F^{(i_1)} \cdots F^{(i_p)}F^* F$ by composing with the unit $e: \Id_{\A_0} \to F^* F$.

	\begin{prop}
	The lax matrix of $\mu_N$ has the form
	\[
	\begin{pmatrix}
	\xymatrix{
	\phi_{00} = \EE^{N-2}(F^{**}) \ar[rr]^{\pi_{N-1, F^{**}}} \ar[d]_{e_{N-1, F}} && \phi_{01} = 
	\EE^{N-1}(F^{**})\ar[d] ^{e_{N,F}}
	\\
	\phi_{10} = \EE^{N-1}(F^*) \ar[rr] _{\pi_{N, F^*}}&& \phi_{11} = \EE^N(F^*)
	} 
	\end{pmatrix},
	\]
	 and the  arrows (natural transformations)  from \eqref{eq:naturalmatrix} are as follows:
	\[
\begin{gathered}
\phi_{00} = \EE^{N-2}(F^{**}) \Rightarrow   \phi_{10}\circ F = \EE^{N-1}(F^*)F \text{ is given by } e_{N-1, F}; 
\\  
\phi_{01} = \EE^{N-1}(F^{**}) \Rightarrow   \phi_{11}\circ F = \EE^N(F^*)F  \text{ is given by } e_{N, F}; 
\\
F^{(N)} \circ \phi_{00} = F^{(N)} \EE^{N-2}(F^{**}) \Rightarrow \phi_{01} = \EE^{N-1}(F^{**})  \text{ is given by }  
\pi_{N-1, F^{**}}; 
\\
F^{(N)} \phi_{10} = F^{(N)} \EE^{N-1}(F^*) \Rightarrow \phi_{11} = \EE^N(F^*)  \text{ is given by } \pi_{N, F^*}. 
\end{gathered} 
\]
	\end{prop}

\noindent{\sl Proof:} We proceed by induction, assuming the statement proved for given $N$. Then we
apply Proposition \ref{prop:comp-2-laxmat} to the three SODs given by $\A_0, \A_1, F$, then
 $\B_0=\A_{\ol N}$, $\B_1 = \A_{\ol {N+1}}$, $G= F^{(N)}$ and $\C_0=\A_{\ol{N+1}}$, $\C_1=\A_{\ol{N+2}} = \A_{\ol{N}}$, $H= F^{(N+1)}$
 and the functors $\phi= \mu_{N, F^*}$ and $\psi = \mu_{F^{(N+1)}}$. The  lax matrix for $\mu_{N+1, F^*}$ is then given as the product
 \[
  \begin{pmatrix}
		 \xymatrix{
			\psi_{00}=0  \ar[r]\ar[d] &\psi_{01}=  \Id_{\A_{\ol N}}\ar[d]^{e}
			\\
			\psi_{10}=\Id_{\A_{\ol{N+1}}} \ar[r]_{\Id}& \psi_{11}= F^{(N+1)} 
			}
	  \end{pmatrix}
	  \begin{pmatrix}
	\xymatrix{
	\phi_{00} = \EE^{N-2}(F^{**}) \ar[rr]^{\pi_{N-1, F^{**}}} \ar[d]_{e_{N-1, F}} && \phi_{01} = 
	\EE^{N-1}(F^{**})\ar[d] ^{e_{N,F}}
	\\
	\phi_{10} = \EE^{N-1}(F^*) \ar[rr] _{\pi_{N, F^*}}&& \phi_{11} = \EE^N(F^*). 
	} 
	\end{pmatrix}
	 \]
We find, first of all, that since $\psi_{00}=0$ and $\psi_{10}=\Id$, 
\[
(\psi\phi)_{00} = \phi_{01} = \EE^{N-1}(F^{**}), \quad (\psi\phi)_{10} = \phi_{11} = \EE^N(F^*).
\]	 
Next, since $\psi_{01}=\Id$ as well, we find
\[
(\psi\phi)_{01} = \Cof\{u_{01}: 	\psi_{01} \phi_{00} \to \psi_{11} \phi_{01}\} \,=\,\Cof \bigl\{ \EE^{N-2}(F^{**}) \to F^{(N+1)} \EE^{N-1}(F^{**})\bigr\} \,=\,
\EE^N(F^*)
\]	
by an instance (for $N-1$ instead of $N$ and $F^*$ instead of $F$)
of the second  Fibonacci triangle  \eqref{eq:fib-tri-2}, the arrow $u_{01}$ being the
left arrow of that triangle. The identification $(\psi\phi)_{11} = \EE^{N+1}(F^*)$ is another instance of the same triangle. 	
The arrows connecting the new matrix elements, are calculated in a similar way. We omit further details. \qed



\paragraph{Proof of Theorem  \ref{thm:N-spher=per}.}\label{par:proof-thm-per}
Suppose $F$ is $N$-spherical, i.e., $\EE_{N-1}(F)=\EE^{N-1}(F)=0$. Then $\EE^{N-1}(F^*) = \EE_{N-1}(F)^*=0$
and $\EE^{N-1}(F^{**}) = \EE^{N-1}(F)^{**}=0$, so the lax matrix of $\mu_N$ is block-diagonal. 
But up to multiplication with block-diagonal matrices of (mutation) equivalences, the matrix for $\mu_N$
is the coordinate change matrix from the SOD $(\A_0, \A_1)$ to the SOD $(\A_N, \A_{N+1})$, where,
we recall, $\A_i = {^\perp\A}_{i-1}$.  This means that $\A_N=\A_0$ and $\A_{N+1} = \A_1$, i.e., the
SOD $(\A_0, \A_1)$ is $N$-periodic.

Conversely, suppose $(\A_0, \A_1)$ is $N$-periodic. Then, by the above reasoning, the matrix for
$\mu_N$ is block-diagonal, so the vanishing of the off-diagonal elements implies, by passing to
appropriate adjoints, that $\EE_{N-1}(F)=\EE^{N-1}(F)=0$. \qed

\paragraph{Proofs of statements from \S \ref{subsec:N-spher}\ref{par:N-spher-list} }\label{par:N-spher-pf}

{\sl Proof of Proposition \ref{prop:N-sph-NM-sph}: } (a) If $F$ is $N$-spherical, then the SOD $(\A_0, \A_1)$
of $\A=\A_0\oplus_F\A_1$ is $N$-periodic. Therefore, $(\A_i = {^{\perp i} \A}_0, \A_{i+1} =  {^{\perp (i+1)} \A}_0)$
which we know being an SOD for $0\leq i\leq N-1$, is an SOD for any $i\in \ZZ$. This implies 
\cite[Prop. 2.5.1]{dkss-spher} that $F$ has adjoints $F^{(i)}$ for any $i\in\ZZ$ and $F^{(i)}$ is, up to
composing with mutation equivalences, a coCartesian gluing functor for $(\A_i, \A_{i+1})$. Since each
$(\A_i, \A_{i+1})$ is periodic, $F^{(i)}$ is spherical as well. 

(b) Follows by observing that an $N$-periodic SOD is $NM$-periodic for any $M\geq 1$.  \qed

\vskip .2cm

\noindent{\sl Proof of Proposition \ref{prop:spher=triang}: }  (i)$\Rightarrow$(ii) follows from Proposition 
\ref{prop:spher-tw-eq}. Conversely, suppose (ii) holds. Then the equivalence $\mu_N: \A_0\oplus_F\A_1\to\A_{\ol N}
\oplus_{F^{(N)}} \A_{\ol{N+1}}$ has the lax matrix block triangular. In other words, $\mu_N$ takes
$\A_0$ into  $\A_{\ol N}$ via an equivalence. But, since $\mu_N$ is an equivalence, it should then take
$\A_1= {^\perp A}_0$  to $\A_{\ol{N+1}} = {^\perp \A}_{\ol N}$ via an equivalence. This means that the
matrix of $\mu_N$ is block diagonal. Vanishing of both off-diagonal elements is equivalent to $F$ being $N$-spherical,
 as we saw in \S 
\ref{par:proof-thm-per}  \qed


\section{Examples}\label{sec:ex}

\subsection{Quivers of Dynkin type $\Att$}
\label{subsec:quivera}

Let $\D$ be a stable $\infty$-category, let $\Att_n$ be the quiver
\[
	\begin{tikzcd} 
		\bullet_1 \ar{r} & \bullet_2 \ar{r} & \dots \ar{r} & \bullet_n,
	\end{tikzcd}
\]
and define the stable $\infty$-category $\C = \Fun(\Att_n,\D)$, where we consider
 $\Att_n$ as the simplicial set 
\[
	\Delta^1 \amalg_{\Delta^{0}} \Delta^1 \dots \amalg_{\Delta^{0}} \Delta^1.
\]
In particular, the objects of $\C$ are given by diagrams of the form
\[
	\begin{tikzcd} 
		X_1 \ar{r} & X_2 \ar{r} & \dots \ar{r} & X_n
	\end{tikzcd}
\]
in $\D$. We consider full subcategories
\[
	\A = \{ X \to 0 \to \dots \to 0 \} \subset \C
\]
as well as the full subcategory 
\[
	\B = \{ 0 \to X_2 \to X_3 \to \dots \to X_n \} \subset \C.
\]

\begin{prop}\label{prop:typea} The pair $(\A,\B)$ is a $2(n+1)$-periodic semiorthogonal decomposition of $\C$. 
\end{prop}
\begin{proof}
	We provide an explicit description of the mutations of the SOD $(\A,\B)$ directly verifying
	the claimed periodicity. To simplify notation, we will refer to the various semiorthogonal
	summands by specifying their objects which define them as full subcategories. We fully spell
	out the case $n=3$ and then explain how it generalizes (for the benefit of readability).

	\[
	\begin{array}{lllll}
				(\A,\B)& = & (\{ X \to 0 \to 0 \}& , & \{ 0 \to X_2 \to X_3 \})\\
				(\B, \A')& = & (\{0 \to X_2 \to X_3\} & , & \{X_1 \overset{\simeq}{\to}
				X_2 \overset{\simeq}{\to} X_3 \})\\
				(\A', \B') &=&(\{X_1 \overset{\simeq}{\to} X_2 \overset{\simeq}{\to}
			X_3 \}&,& \{X_1 \to X_2 \to 0 \})\\
				(\B', \A'') &=&(\{X_1 \to X_2 \to 0 \}&,& \{0 \to 0 \to X\})\\
				(\A'', \B'') &=&(\{ 0 \to 0 \to X \}&,& \{X_1 \to X_2
					\overset{\simeq}{\to} X_3 \})\\
				(\B'', \A^{(3)}) &=&(\{X_1 \to X_2 \overset{\simeq}{\to} X_3 \}&,&
				\{0 \to X \to 0\})\\
				(\A^{(3)}, \B^{(3)}) &=& (\{0 \to X \to 0\}&,& \{X_1
					\overset{\simeq}{\to} X_{2} \to X_3 \})\\
				(\B^{(3)}, \A^{(4)}) &=& (\{X_1 \overset{\simeq}{\to} X_2 \to X_{3}
			\}&,& \{X \to 0 \to 0\})\\
				(\A^{(4)}, \B^{(4)}) &=&(\{X \to 0 \to 0\}&,&\{X_1 \overset{\simeq}{\to} X_2 \overset{\simeq}{\to} X_{3} \})
	\end{array}
	\]

	The general case is completely analogous but with semiorthogonal summands
	\[
		\A^{(i)} = \begin{cases}  \{X_1 \overset{\simeq}{\to} X_2 \overset{\simeq}{\to}
			\dots \overset{\simeq}{\to} X_n \} & \text{for $i=1$,}\\
			\{0 \to \dots \to 0 \to X_{n-i+2} \to 0 \to \dots \to 0\} & \text{for $i>1$,}
		\end{cases}
	\]
	and 
	\[
		\B^{(i)} = \begin{cases}  \{X_1 \to X_2 \to \dots \to X_{n-1} \to 0\} & \text{for $i=1$,}\\
			\{X_1 \to \dots \to X_{n-i+1} \overset{\simeq}{\to} X_{n-i+2} \to \dots \to X_{n} \} & \text{for $i>1$.}
		\end{cases}
	\]
\end{proof}

\

We conclude by describing the higher spherical twists. These autoequivalences of $\A$ and $\B$ are
obtained as the composites of the mutation equivalences
\[
	\A \overset{\simeq}{\to} \A^{(1)} \overset{\simeq}{\to} \A^{(2)} \dots  \overset{\simeq}{\to} \A^{(n+1)}
\]
and analogously for $\B$. Here the mutation equivalence $\A \to \A^{(1)}$ is given by the formula
\[
	a \mapsto \fib(a \to a_!)
\]
in $(\A,\B)$. We display the computations for $n=3$, for $\A$:
\[
	\begin{array}{lll}
		X \to 0 \to 0  & &  \in \A \\
		X \to X \to X  & = \fib((X \to 0 \to 0) \to (0 \to X[1] \to X[1])) & \in \A^{(1)}\\
		0 \to 0 \to X  & = \fib((X \to X \to X) \to (X \to X \to 0)) & \in \A^{(2)}\\
		0 \to X[-1] \to 0  & = \fib((0 \to 0 \to X) \to (0 \to X \to X)) & \in \A^{(3)}\\
		0 \to X[-1] \to 0  & = \fib((0 \to X[-1] \to 0) \to (X[-1] \to X[-1] \to 0)) & \in \A^{(3)}\\
		X[-2] \to 0 \to 0  & = \fib((0 \to X[-1] \to 0) \to (X[-1] \to X[-1] \to 0)) & \in \A^{(4)}
	\end{array}
\]
so that $T_{\A} = [-2]$. For $\B$, we have 
\[
	\begin{array}{lll}
		0 \to X_2 \to X_3  & &  \in \B \\
		X_3[-1] \to X_3/X_2[-1] \to 0  & = \fib((0 \to X_2 \to X_3) \to (X_3 \to X_3 \to X_3)) & \in \B^{(1)}\\
		X_3[-1] \to X_3/X_2[-1] \to X_3/X_2[-1] & = \fib((X_3[-1] \to X_3/X_2[-1] \to 0) \to (0 \to 0 \to X_3/X_2)) & \in \B^{(2)}\\
		X_3[-1] \to X_3[-1] \to X_3/X_2[-1]  & = \fib((X_3[-1] \to X_3/X_2[-1] \to X_3/X_2[-1]) \to (0 \to X_2
		\to 0)) & \in \B^{(3)}\\
		0 \to X_3[-1] \to X_3/X_2[-1]  & = \fib((X_3[-1] \to X_3[-1] \to X_3/X_2[-1]) \to (X_3[-1] \to 0 \to 0)) & \in \B^{(4)}\\
	\end{array}
\]
so that $T_{\B}$ is given by 
\[
	(X_2 \to X_3) \mapsto (X_3[-1] \to X_3/X_2[-1])
\]
which is precisely the (inverse) Coxeter functor $\tau$. This statement holds for arbitrary $n$. 


\subsection{Quivers of Dynkin type $\Dtt$}
\label{subsec:quiverd}

Let $\D$ be a stable $\infty$-category and let $\Dtt_n$ ($n \ge 4$) be the quiver
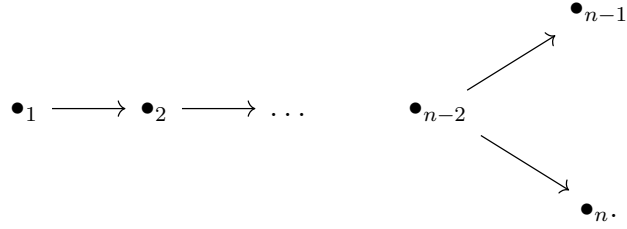
\[
	\begin{tikzcd} 
		& & & & \bullet_{n-1}\\
		\bullet_1 \ar{r} & \bullet_2 \ar{r} & \dots & \bullet_{n-2} \ar{ur} \ar{dr} &  \\
		& & & & \bullet_n.
	\end{tikzcd}
\]
We set $\C = \Fun(\Dtt_n,\D)$ where we consider $\Dtt_n$ as a simplicial set in analogy to \S
\ref{subsec:quivera}. Consider the full subcategories
\[
	\A = \left\{ \begin{tikzcd}[sep=small, row sep=.5ex] 
		& & & & 0\\
		X \ar{r} & 0 \ar{r} & \dots & 0 \ar{ur} \ar{dr} &  \\
		& & & & 0
	\end{tikzcd} \right\} \subset \C
\]
and 
\[
	\B = \left\{ \begin{tikzcd}[sep=small, row sep=.5ex]  
		& & & & X_{n-1}\\
		0 \ar{r} & X_2 \ar{r} & \dots & X_{n-2} \ar{ur} \ar{dr} &  \\
		& & & & X_n.
	\end{tikzcd} \right\} \subset \C
\]

\begin{prop}\label{prop:typed} The pair $(\A,\B)$ defines an $2(n-1)$-periodic SOD of $\C$.
\end{prop}

We describe the mutations of $(\A,\B)$ for $n = 4$, the general case is a straightforward
modification.
\[
	\begin{array}{ccccc}
				(\A,\B)& = & \biggl(
				 \begin{tikzcd}[sep=small, row sep=.5ex] 
						& &  0\\
						X \ar{r} & 0  \ar{ur} \ar{dr} &  \\
						& &  0
					\end{tikzcd}  & , & 
				 \begin{tikzcd}[sep=small, row sep=.5ex] 
						& &  X_3\\
						0 \ar{r} & X_2  \ar{ur} \ar{dr} &  \\
						& & X_4 
					\end{tikzcd} 
				\biggr)
				\end{array}
				\]
				
				\[
					\begin{array}{ccccc}
				(\B, \A^{(1)})& = & \biggl(
				 \begin{tikzcd}[sep=small, row sep=.5ex] 
						& &  X_{3}\\
						0 \ar{r} & X_2  \ar{ur} \ar{dr} &  \\
						& &  X_{4}
					\end{tikzcd} & , & 
				 \begin{tikzcd}[sep=small, row sep=.5ex] 
						& &  X_{3}\\
						X_1 \ar{r}{\simeq} & X_2 \ar{ur}{\simeq} \ar{dr}{\simeq} &  \\
						& &  X_{4}
					\end{tikzcd} 
				\biggr)
				\end{array}
				\]
				
				\[
					\begin{array}{ccccc}
				(\A^{(1)},\B^{(1)})& = & \biggl(
				 \begin{tikzcd}[sep=small, row sep=.5ex] 
						& &  X_{3}\\
						X_1 \ar{r}{\simeq} & X_2 \ar{ur}{\simeq} \ar{dr}{\simeq} &  \\
						& &  X_{4}
					\end{tikzcd} & , & 
				 \begin{tikzcd}[sep=small, row sep=.5ex] 
						& &  X_{3}\\
						X_1 \ar{r} & X_3 \oplus X_4 \ar{ur} \ar{dr} &  \\
						& &  X_{4}
					\end{tikzcd} 
				\biggr)
				\end{array}
				\]
				
				\[
					\begin{array}{ccccc}
				(\B^{(1)},\A^{(2)})& = & \biggl(
				 \begin{tikzcd}[sep=small, row sep=.5ex] 
						& &  X_{3}\\
						X_1 \ar{r} & X_3 \oplus X_4 \ar{ur} \ar{dr} &  \\
						& &  X_{4}
					\end{tikzcd} & , & 
				 \begin{tikzcd}[sep=small, row sep=.5ex] 
						& &  0\\
						0 \ar{r} & X_2 \ar{ur} \ar{dr} &  \\
						& &  0
					\end{tikzcd} 
				\biggr)
				\end{array}
				\]
				
				\[
					\begin{array}{ccccc}
				(\A^{(2)},\B^{(2)})& = & \biggl(
				 \begin{tikzcd}[sep=small, row sep=.5ex] 
						& &  0\\
						0 \ar{r} & X_2 \ar{ur} \ar{dr} &  \\
						& &  0
					\end{tikzcd} & , & 
				 \begin{tikzcd}[sep=small, row sep=.5ex] 
						& &  X_{3}\\
						X_1 \ar{r}{\simeq} & X_2 \ar{ur} \ar{dr} &  \\
						& &  X_{4}
					\end{tikzcd} 
				\biggr)
				\end{array}
				\]
				
				\[
					\begin{array}{ccccc}
				(\B^{(2)},\A^{(3)})& = &\biggl(
				 \begin{tikzcd}[sep=small, row sep=.5ex] 
						& &  X_{3}\\
						X_1 \ar{r}{\simeq} & X_2 \ar{ur} \ar{dr} &  \\
						& &  X_{4}
					\end{tikzcd} & , & 
				 \begin{tikzcd}[sep=small, row sep=.5ex] 
						& &  0\\
						X \ar{r} & 0 \ar{ur} \ar{dr} &  \\
						& & 0 
					\end{tikzcd} 
				\biggr)
				\end{array}
				\]
				
				\[
					\begin{array}{ccccc}
				(\A^{(3)},\B^{(3)})& = & \biggl(
				 \begin{tikzcd}[sep=small, row sep=.5ex] 
						& &  0\\
						X \ar{r} & 0 \ar{ur} \ar{dr} &  \\
						& & 0 
					\end{tikzcd} & , & 
				 \begin{tikzcd}[sep=small, row sep=.5ex] 
						& &  X_{3}\\
						0 \ar{r} & X_2 \ar{ur} \ar{dr} &  \\
						& &  X_{4}
					\end{tikzcd} 
				\biggr).
	\end{array}
\]


\subsection{Fractional Calabi-Yau categories}\label{subsec:fra-CY}

\paragraph {$\k$-linear stable $\oo$-categories.} 
Recall that $\k$ is a field.  By $\Vect_\k$ we denote the category of finite-dimensional $\k$-vector
spaces.

In this subsection we consider {\em $\k$-linear stable $\oo$-categories},  i.e.,   stable $\oo$-categories 
 enriched over the  monoidal category of $\k$-module spectra \cite{lurie:ha}. 
 As shown in \cite{cohn}, see also references in \cite[1.1.0.2]{lurie:ha}, such $\oo$-categories
 can be identified with dg-nerves \cite[1.3.1.16]{lurie:ha}, \cite{faonte} of pre-triangulated dg-categories
 over $\k$
 in the sense of \cite{BK:enh}. This means, in particular,
  that we can assume that we are given cochain complexes
 $\Hom^\bullet_\A(x,y)$ over $\k$ for $x,y\in\Ob\A$ such that
\[
\Map_\A(x,y) = \DK\bigl(\tau_{\leq 0} \Hom^\bullet_\A(x,y)\bigr), 
\]
where $\DK$ is the Dold-Kan functor from $\ZZ_{\leq 0}$-graded cochain complexes to simplicial
vector spaces, and $\tau_{\leq 0}$ is the $\ZZ_{\leq 0}$-truncation. For the triangulated category $\h\A$ this
means that
\[
\Hom_{\h\A}(x,y) = H^0 \Hom^\bullet_\A(x,y). 
\]
In examples of interest for us  such complexes are given at the outset. 
In particular, if $X$ is a smooth algebraic variety over $\k$, we denote by $D^b(\Coh_X)$
  the standard dg-enhancement  \cite{BK:enh}
   of the bounded derived category of coherent sheaves on $X$, as well as the corresponding
   stable $\oo$-category obtained as the dg-nerve. For $X=\Spec(\k)$ we use the notation
   $D^b(\Vect_\k)$.

\paragraph{Serre functors for $\k$-linear stable $\oo$-categories.}

Let $\A$ be a $\k$-linear 
stable $\oo$-category. We say that $\A$ is {\em  proper}, if the complexes $\Hom^\bullet_\A(x,y)$
are quasi-isomorphic to bounded complexes of finite-dimensional $\k$-vector spaces. 
 Suppose that $\A$ is proper. Then for each object $x\in\A$ we have a contravariant $\oo$-functor
 \[
 \sigma_x: \A \to\Sp: y\mapsto \DK\biggl(\tau_{\leq 0} \bigl(\Hom^\bullet_\A(x,y)^*\bigr)\biggr).
 \]
 Here $*$ means dualization of the complex over $\k$. 
 Suppose further that for each object $x\in \A$ the functor $\sigma_x$ is represented by
 an object $S(x)\in \A$, i.e., 
 \be\label{eq:serre}
 \sigma_x(y) \,\=\, \Map_\A(y, S(x)). 
 \ee
 Then the standard properties of representable $\oo$-functors imply that the correspondence 
 $x\mapsto S(x)$ extends, in a homotopy unique way, to an $\oo$-functor $S=S_\A: \A\to\A$
 which we call the {\em Serre functor} of $\A$. This agrees with the standard terminology
 \cite{BK:SOD}  for triangulated categories, since \eqref{eq:serre} implies that
 \[
 \Hom_{\h\A}(x,y)^* \,\=\, \Hom_{\h\A}(y, S(x)). 
 \]
As in \cite{BK:SOD} we have:

\begin{prop}\label{prop:serre-properties}
(a) Let $\B, \C$ be proper $\k$-linear stable $\oo$-categories with Serre functors $S_\B$ and $S_\C$.
Suppose that $F: \B\to \C$ is an exact $\oo$-functor admitting a right adjoint $F^*$. Then $F$ admits
a second adjoint $F^{**}$ given by $F^{**}=S_C F S_B^{-1}$. This implies that $F$ has adjoints of all orders. 

\vskip .2cm

(b) Let $(\A_0, \A_1 = {^\perp \A}_0)$ be a coCartesian semi-orthogonal decomposition of $\A$,
  If $\A$ has a Serre functor $S_\A$, then $\A_{-1} = \A_0^\perp = \A_1^{\perp\perp}$ is obtained as 
  $S_\A (\A_1)$. In particular, $(\A_0, \A_1)$ is $\oo$-admissible and in the chain of orthogonals
  $\A_{2i} = {^{\perp{2i}}}\A_0$, $i\in \ZZ$,  is obtained as $S_\A^{-i}(\A_0)$ while
  $\A_{2i+1} =  {^{\perp{(2i+1)}}}\A_0$ is obtained as  $S_\A^{-i}(\A_1)$.
  
  \vskip .2cm
  
  (c) If $X$ is a smooth $n$-dimensional projective variety over $X$, then the Serre functor in $D^b(\Coh_X)$
  is given by tensoring with $\omega_X:= \Omega^n_X[n]$. 
  \qed

\end{prop}

\paragraph {Calabi-Yau and fractional Calabi-Yau stable $\oo$-categories.} 

Let $\C$ be a proper $\k$-linear stable $\oo$-category.
 We say that $\C$ is Calabi-Yau of dimension $d\in \ZZ$, if $S_\C \=  [d]$
is the functor of shift by $d$. In this case any semi-orthogonal decomposition $(\A,\B)$ of $\C$ is orthogonal,
or $2$-periodic, i.e., not only $\B= {^\perp \A}$ but also $\A= {^\perp \B}$,  
as
\[
\Hom_{\h\C}(a,b) = \Hom_{\h\C}(b, a[d])^* =0, \quad a\in \A, \, b\in\B. 
\]
In particular, the gluing functor $F: \A\to \B$ of each such SOD is $0$. 

\vskip .2cm

Further, we say that $\C$ is  {\em fractional Calabi-Yau} of dimension $p/q$, if $S_\C^q \= [p]$, see \cite{kuznetsov}. 

\begin{prop}
If $\C$ is fractional Calabi-Yau of dimension $p/q$, then any coCartesian
SOD $(\A, \B)$ of $\C$ is $2q$-periodic. 
\end{prop}

\noindent{\sl Proof:} Follows from Proposition \ref{prop:serre-properties}. \qed

\paragraph{Example: quivers of type $\Att$ over a field.} We specialize the considerations of \S \ref{subsec:quivera}
to the case $\D = D^b(\Vect_\k)$. In this case the Serre functor in $\Fun(\Att_n, D)$ is given by the Coxeter functor 
$\tau[1]$. We further have the
 relation 
	\[
		\tau^{n+1} \simeq [-2],
	\]
	categorfying the equation $ z^{n+1} = w^2$ for the $\Att_n$-singularity. 
	This means that $\Fun(\Att_n,\D)$ is fractional Calabi-Yau of dimension $\frac{n-1}{n+1}$.
	 In particular,  every coCartesian semiorthogonal decomposition is $2(n+1)$-periodic providing a more abstract
	explanation for Proposition \ref{prop:typea}.

\paragraph{Quivers of type $\Dtt$ over a field.} We now specialize the considerations of \S \ref{subsec:quiverd}
to the case $\D = D^b(\Vect_\k)$. In this case (as for any finite type quiver) we also have the Coxeter functor
$\tau$ and the Serre functor is identified as
 $S=\tau[1]$
but the  behavior of $S$ depends on the parity of $n$: 

 \begin{itemize}

\item[((0)] For $n$ even, there is the relation 
		\[
			\tau^{n-1} \simeq [-1]
		\]
		so that the category is fractionally Calabi-Yau of dimension $\frac{n-2}{n-1}$.
		Again, this provides an abstract justification of Proposition \ref{prop:typed} and
		shows more generally that, in this case, every  coCartesian SOD is $2(n-1)$-periodic. 
		
		\vskip .2cm

\item[(1)] 	  For $n$ odd, there is the relation 
		\[
			\tau^{n-1} \simeq \theta [-1]
		\]
		where $\theta$ is the involution induced by the reflection of the quiver. Therefore,
		only the $2(n-1)$th power of the Serre functor is a shift.
		Still, this provides an abstract justification of Proposition \ref{prop:typed},
		since both summands of the given SOD are stable under the reflection autoequivalence
		$\theta$. Note that SODs for which this is not the case, however, may only be $4(n-1)$
		periodic , such as (for the case $n=4$)
		\end{itemize}
		
		 \[	
			\begin{array}{ccccc}
			(\A,\B)& = & \biggl(
				 \begin{tikzcd}[sep=small, row sep=.5ex] 
						& &  X_3\\
						X_1 \ar{r} & X_2 \ar{ur} \ar{dr} &  \\
						& & 0 
					\end{tikzcd} & , & 
				 \begin{tikzcd}[sep=small, row sep=.5ex] 
						& &  0\\
						0 \ar{r} & 0 \ar{ur} \ar{dr} &  \\
						& &  X
					\end{tikzcd} 
				\biggr)
				\end{array} 
				\]
				
				\[
				\begin{array}{ccccc}
				(\B,\A^{(1)})& = & \biggl(
				 \begin{tikzcd}[sep=small, row sep=.5ex] 
						& &  0\\
						0 \ar{r} & 0 \ar{ur} \ar{dr} &  \\
						& &  X
					\end{tikzcd} & , & 
				 \begin{tikzcd}[sep=small, row sep=.5ex] 
						& &  X_3\\
						X_1 \ar{r} & X_2 \ar{ur} \ar{dr}{\simeq} &  \\
						& &  X_4
					\end{tikzcd} 
				\biggr)
				\end{array}
				\]
				
				\[
				\begin{array}{ccccc}
				(\A^{(1)},\B^{(1)})& = & \biggl(
				 \begin{tikzcd}[sep=small, row sep=.5ex] 
						& &  X_3\\
						X_1 \ar{r} & X_2 \ar{ur} \ar{dr}{\simeq} &  \\
						& &  X_4
					\end{tikzcd} & , & 
				 \begin{tikzcd}[sep=small, row sep=.5ex] 
						& &  X_3\\
						0 \ar{r} & X_2 \ar{ur}{\simeq} \ar{dr} &  \\
						& &  0
					\end{tikzcd} 
				\biggr)
				\end{array}
				\]
				
				\[
				\begin{array}{ccccc}
				(\B^{(1)},\A^{(2)})& = & \biggl(
				 \begin{tikzcd}[sep=small, row sep=.5ex] 
						& &  X_3\\
						0 \ar{r} & X_2 \ar{ur}{\simeq} \ar{dr} &  \\
						& &  0
					\end{tikzcd} & , & 
				 \begin{tikzcd}[sep=small, row sep=.5ex] 
						& &  X_3\\
						X_1 \ar{r} \ar{urr}{\simeq} & X_2 \ar{ur} \ar{dr} &  \\
						& &  X_4
					\end{tikzcd} 
				\biggr)
				\end{array}
				\]
				
				\[
				\begin{array}{ccccc}
				(\A^{(2)},\B^{(2)})& = & \biggl(
				 \begin{tikzcd}[sep=small, row sep=.5ex] 
						& &  X_3\\
						X_1 \ar{r} \ar{urr}{\simeq} & X_2 \ar{ur} \ar{dr} &  \\
						& &  X_4
					\end{tikzcd} & , & 
				 \begin{tikzcd}[sep=small, row sep=.5ex] 
						& &  X_3\\
						X_1 \ar{r}{\simeq} & X_2 \ar{ur}{\simeq} \ar{dr} &  \\
						& &  0
					\end{tikzcd} 
				\biggr)
				\end{array}
				\]
				
				\[
				\begin{array}{ccccc}
				(\B^{(2)},\A^{(3)})& = & \biggl(
				 \begin{tikzcd}[sep=small, row sep=.5ex] 
						& &  X_3\\
						X_1 \ar{r}{\simeq} & X_2 \ar{ur}{\simeq} \ar{dr} &  \\
						& &  0
					\end{tikzcd} & , & 
				 \begin{tikzcd}[sep=small, row sep=.5ex] 
						& &  0\\
						X_1 \ar{r} & X_2 \ar{ur} \ar{dr} &  \\
						& & X_3
					\end{tikzcd} 
				\biggr)
				\end{array}
				\]

				\[
				\begin{array}{ccccc}
				(\A^{(3)},\B^{(3)})& = & \biggl(
				 \begin{tikzcd}[sep=small, row sep=.5ex] 
						& &  0\\
						X_1 \ar{r} & X_2 \ar{ur} \ar{dr} &  \\
						& & X_4
					\end{tikzcd} & , & 
				 \begin{tikzcd}[sep=small, row sep=.5ex] 
						& &  X_3\\
						0 \ar{r} & 0 \ar{ur} \ar{dr} &  \\
						& &  0
					\end{tikzcd} 
				\biggr)
				\end{array}
				\]

			 And so on until
				
				\[
				\begin{array}{ccccc}
				(\A^{(6)}, \B^{(6)}) & = & \biggl(
				 \begin{tikzcd}[sep=small, row sep=.5ex] 
						& &  X_3\\
						X_1 \ar{r} & X_2 \ar{ur} \ar{dr} &  \\
						& & 0 
					\end{tikzcd} & , & 
				 \begin{tikzcd}[sep=small, row sep=.5ex] 
						& &  0\\
						0 \ar{r} & 0 \ar{ur} \ar{dr} &  \\
						& &  X
					\end{tikzcd} 
				\biggr).
				\end{array}
				\]


\subsection {The Waldhausen S-construction}
Let $F: \A \to\B$ be  an exact functor of stable $\oo$-categories.   
We refer to  \cite[\S 3]{dkss-spher} and references therein for background
on 
 $S_\bullet(F)$,  the (relative) {\em Waldhausen S-construction}
of $F$. Thus $S_\bullet(F)$ is a simplicial object in the (usual) category of
stable $\oo$-categories, so for each $n\geq 0$ we have a stable $\oo$-category
  $S_n(F)$ and these categories are connected by  simplicial face and
  degeneration maps. 
  
  \vskip .2cm
  
  Objects of $S_n(F)$ are data 
  \be\label{eq:S_n-obj}
   \bigl(A_1\to\cdots\to A_n, B, F(A_n)\to B\bigr)
  \ee
  consisting of objects $A_1,\cdots, A_n\in \A$,
  $B\in\B$ and 
   morphisms
 $ A_1\to\cdots\to A_n$ in $\A$ and $F(A_n)\to B$ in $\B$. Thus $S_0(F) = \B$ and $S_1(F) = \A\oplus_F \B$.  More generally, let $\Fun(\Att_n, \A)$ be the
 category of representations of the $\Att_n$-quiver in $\A$, as in 
 \S \ref{subsec:quivera}, whose objects are chains of morphisms
 $ A_1\to\cdots\to A_n$ in $\A$. The evaluation at the last term
 sending a chain as above to $A_n$ defines a functor $\ev_n: \Fun(\Att_n, \A)\to\A$.
 In the notation of this paper we can write
 \[
 S_n(F) = \Fun(\Att_n, \A)\oplus_{F(\ev_n)} \B. 
 \]
 In particular, it has a coCartesian SOD $ (\Fun(\Att_n, \A), \B)$.
 
 \begin{thm}
 Suppose $F$ is spherical in the usual sense, i.e., $4$-spherical. Then
 $F(\ev_n)$ is $2(n+1)$-spherical. 
 \end{thm} 
  
  \begin{proof} Equivalently, we need to show that the SOD above is $2(n+1)$-periodic. 
  As shown in \cite[Th. 3.2.1]{dkss-spher},  $F$ being spherical
  implies that that $S_\bullet(F)$ has not just simplicial but a {\em paracyclic
  structure}. In particular, we have the paracyclic rotation $\tau_n$ 
  which is a self-equivalence  of $S_n(F)$ with $\tau_n^{n+1}$
  (the  ``monodromy of the schober'')
  preserving the  SOD $(\Fun(\Att_n, \A), \B)$. 
  
  The category $\B$ is the left orthogonal of $\Fun(\Att_n, \A)$ and we claim that $\tau_n(\B)$ is
  the right orthogonal of $\Fun(\Att_n, \A)$.
  For this we recall the action of $\tau_{n+1}$ on objects \cite{dkss-spher}.
  In a datum  \eqref{eq:S_n-obj}
   representing an object of $S_n(F)$,
  the morphism $F(A_n)\to B$ is associated to
  a homotopy unique morphism $A_n\to F^*(B)$. Then
  \[
  \begin{gathered}
  \tau_n \bigl(A_1\to\cdots\to A_n, \, B, F(A_n) \to B\bigr)
  =
  \\
  =
  \bigl( A_2/A_1 \to \cdots \to A_n/A_1 \to F^*(B)/A_1,\, B/F(A_1)), \, FF^*(B)/F(A_1) \to B/F(A_1))\bigr). 
  \end{gathered}
  \]
 Applying this to the case when all $A_i=0$ and $B$ is arbitrary,
 we get an object of the form
 \[
 \bigl(0\to \cdots \to 0 \to F^*(B), \, F^*(B), FF^*(B) \overset{\eta}{\to} B\bigr)
 \]
 where $\eta$ is the counit of the adjunction $F \dashv F^*$ or, equivalently, the adjoint morphism
 $F^*(B) \to F^*(B)$ of $\eta$ is an equivalence.   
 It follows from this last characterization by direct computation that the full subcategory on such objects constitutes the right
 orthogonal to the subcatgory $\Fun(\Att_n, \A)$ so that we obtain an SOD $(\tau_n(\B),\Fun(\Att_n,
 \A))$. Piecing together the SODs obtained from the above two SODs by applying powers of $\tau_n$ we thus obtain the chain
 of orthogonals
 \[
	\tau_n^{n+1}(\Fun(\Att_n, \A)) =\Fun(\Att_n, \A),\; \tau_n^{n+1}(\B) = \B,\;
\tau_n^{n}(\Fun(\Att_n, \A)),\quad  ... \quad ,\; \tau_n(\B),\; \Fun(\Att_n, \A)
  \]
  of period $2(n+1)$. 
\end{proof}


\subsection{$N$-spherical objects}\label{subsec:N-spher-obj}

\paragraph{The general concept.} 
 In this subsection we assume all stable $\oo$-categories to be $\k$-linear. 
Each such category $\B$ is tensored over $D^b(\Vect_\k)$. In particular, each object $E\in \B$ gives
rise to an exact $\oo$-functor
\be
F=F_E: \A=D^b(\Vect_\k) \lra \B, \quad V\mapsto V\otimes E. 
\ee

An object $E\in \B$ will be called an {\em $N$-spherical object} of $\B$, if $F_E$ is an $N$-spherical functor.
This extends the standard terminology relating spherical objects and spherical functors \cite{AL17}
which corresponds to $N=4$. As the sizes of $\A$ and $\B$ are typically quite different,
the concept of  $N$-spherical objects makes real sense only for $N$ even. 

\paragraph{$N$-spherical objects in algebraic geometry.}\label{par:N-spher-AG}

Let $X$ be a smooth projective
 $n$-dimensional algebraic variety over $\k$. As the Serre functor of $\A=D^b(\Vect_\k)$
is trivial and that of $\B=D^b(\Coh_X)$ is given by tensoring with $\omega = \omega_X = \Omega^n_X[n]$,
 we conclude, see Proposition \ref{prop:serre-properties}(a), that 
 for any object $E\in \B$ the functor $F_E$ has adjoints of all integer orders of which the iterated 
right adjoints
$F^{(i)}$, $i\geq 0$,  have the form:
\[
   \xymatrix@R=0.5pc{
   \A = D^b(\Vect_\k)\ar[rr]^{F=-\otimes E} && \B= D^b(\Coh_X)
   \\
   &&\ar[ll]_{F^* = \Hom^\bullet(E,-)}
   \\
    \ar[rr]^{F^{**}= -\otimes E\otimes\omega}&&
    \\
    && \ar[ll]_{F^{(3)} = \Hom^\bullet(E\otimes\omega, -)} 
    \\
    \ar[rr]^{F^{(4)} = -\otimes E\otimes\omega^{\otimes 2}} &&
    \\
    \cdots\cdots\cdots & \cdots\cdots\cdots &  \cdots\cdots\cdots
  }
   \]

   \paragraph{$6$-spherical objects.} After   usual (or $4$-)spherical
   objects, the first new case is that of $6$-spherical objects. Let $F=F_E$. 
   By Proposition \ref{prop:spher=triang} to check that $E$ is $6$-spherical,
   it suffices to verify that $\EE^4(F): \A\to \A$ is an equivalence and $\EE^5(F): \A\to\B$ is
   zero.  Now,  $\EE^4(F)$ is the totalization of
   \be\label{eq:E^3(F_E)}
   \xymatrix{
    & F^{(3)} F^{(2)} \ar[dr]&
    \\
    \Id_\A \ar[ur] \ar[dr]& F^{(3)} F \ar[r] & F^{(3)} F^{(2)} F^* F.
    \\
    & F^* F \ar[ur] &
   }
   \ee
   We can view this as a $3$-term complex with the middle term being the sum of
   three functors in the central column. Since all these functors are endofunctors of $\A=\D^b(\Vect_\k)$,
   they reduce to tensor multiplication with some (dg-)vector spaces which can be found by applying them
   to the object $\k\in\A$. 
    Let  $A=\Hom^\bullet_\B(E,E)$ be the endomorphism
   dg-algebra of $E$. We assume the algebro-geometric situation of \S \ref{par:N-spher-AG}
   Then:
   \be\label{eq:F3F2(k)}
   \begin{gathered}
   F^{(3)} F^{(2)}: \k \mapsto F^{(3)}(E\otimes\omega) = \Hom^\bullet(E\otimes\omega, E\otimes\omega) = A;
   \\ F^{(3)} F: \k \mapsto F^{(3)} (E) = \Hom^\bullet (E\otimes\omega, E); 
   \\
   F^*F: \k \mapsto F^*(E)=\Hom^\bullet(E. E) = A;
   \\
    F^{(3)} F^{(2)} F^* F: \k \mapsto F^{(3)} F^{(2)} (A) = A\otimes A. 
   \end{gathered}
   \ee
   So \eqref{eq:E^3(F_E)} reduces to the complex of (dg-)vector spaces
   \be\label{eq:E^3(F_E)-dg}
   \xymatrix{
    & A \ar[dr]^{a\mapsto a\otimes 1}&
    \\
    \k \ar[ur]^1 \ar[dr]_1& \Hom^\bullet (E\otimes\omega, E) \ar[r]^{\quad\quad s_*} &  A\otimes A,
    \\
    & A \ar[ur]_{a\mapsto 1\otimes a} &
   }
   \ee
where $s_*: \Hom^\bullet(E\otimes \omega, E)\to A\otimes A$ is induced by the map
\[
s: E\lra E\otimes \omega\otimes \Hom^\bullet(E, E) = E\otimes\omega\otimes A
\]
which corresponds by Serre duality to   
\[
1\in A^*\otimes A = \Hom^\bullet(E, E\otimes\omega) \otimes \Hom^\bullet(E, E). 
\]
The condition that $\EE^4(F)$ is an equivalence means that the cohomology of the complex
\eqref{eq:E^3(F_E)-dg} reduces to a $1$-dimensional vector space in some degree. 

\vskip .2cm

Next, $\EE^5(F): \A\to\B$ is the totalization of 
\be\label{eq:E5F-general}
\xymatrix@R=0.1pt@C=70pt{
& F^{(4)}F^{(3)} F^{(2)} &
\\
F & \oplus &
\\
\oplus & F^{(4)} F^{(3)} F & 
\\
F^{(2)}\quad  \ar[r]  & \quad \quad \oplus  \quad\quad  \ar[r] & \quad F^{(4)} F^{(3)} F^{(2)} F^* F.
\\
\oplus & F^{(4)} F^* F &
\\
F^{(4)}  & \oplus & 
\\
& F^{(2)} F^* F & 
}
\ee
 and is determined by  the object $\EE^5(F)(\k)\in\B$. Continuing the computations in 
 \eqref{eq:F3F2(k)}, we find:
 \[
 \begin{gathered}
 F^{(4)} F^{(3)} F^{(2)}: \k \mapsto F^{(4)} (A) = A\otimes E\otimes \omega^{\otimes 2}; 
 \\
 F^{(4)} F^{(3)} F: \k \mapsto F^{(4)} (\Hom^\bullet(E\otimes\omega, E)) = 
 \Hom^\bullet(E\otimes\omega, E)\otimes\omega^{\otimes 2}; 
 \\
 F^{(4)} F^* F: \k \mapsto F^{(4)}(A) = A\otimes E\otimes \omega^{\otimes 2};
 \\
 F^{(2)} F^* F: \k \mapsto F^{(2)}(A) = A\otimes E\otimes\omega; 
 \\
 F^{(4)} F^{(3)} F^{(2)} F^* F: \k \mapsto F^{(4)}(A\otimes A) = A\otimes A\otimes\omega^{\otimes 2}. 
 \end{gathered} 
 \]
So $\EE^5(F)(\k)$ is the totalization of 
\be\label{eq:e5(F_E)}
\xymatrix@R=0.1pt@C=50pt{
& A\otimes E\otimes\omega^{\otimes 2}&
\\
E & \oplus &
\\
\oplus & \Hom^\bullet(E\otimes\omega, E)\otimes E\otimes \omega^{\otimes 2} & 
\\
E\otimes\omega \quad  \ar[r]  & \quad \quad \oplus  \quad\quad  \ar[r] & 
\quad A\otimes A\otimes E\otimes \omega^{\otimes 2}.
\\
\oplus & A\otimes E\otimes \omega^{\otimes 2} &
\\
E\otimes\omega^{\otimes 2} & \oplus & 
\\
& A\otimes E\otimes\omega & 
}
\ee
The second condition for $E$ to be $6$-spherical is that \eqref{eq:e5(F_E)} should be exact.

\paragraph{Enriques manifolds.} We say that $X$ is an {\em Enriques manifold},
if it can be obtained as a quotient $X=\wt X/\ZZ_2$ where $\wt X$ is a Calabi-Yau manifold
(of dimension $n$) and $\ZZ_2=\{1,\sigma\}$ is generated by a fixed point free involution $\sigma$ which acts
on $H^0(X, \Omega^n)$ by $(-1)$. For $n=2$ we get the classical concept of an Enriques
surface. 

We have, in particular, that $(\Omega_X^n)^{\otimes 2} = \Oc_X$, and thus $\omega^{\otimes 2} =
\Oc_X[2n]$ and $\omega^{\otimes (-1)} = \omega[-2n]$. 

\begin{prop}\label{prop:enriques}
For an Enriques manifold of dimension $n$ the sheaf $\Oc= \Oc_X$ is a $6$-spherical object.
\end{prop}

\noindent{\sl Proof:} If $E=\Oc$, then $A=\k$, while $\Hom^\bullet(E\otimes \omega, E) = 
\k[-2n]$. 
So the complex \eqref{eq:E^3(F_E)-dg}
has the form
 \[
   \xymatrix{
    & \k \ar[dr]^{  1}&
    \\
    \k \ar[ur]^1 \ar[dr]_1&   \k[-2n] \ar[r]^{\quad s_*} &  \k,
    \\
    & \k \ar[ur]_{  1 } &
   }
   \]
   The outside rim  here is an exact subcomplex, and the quotient by it is $\k[-2n]$, so the total
   complex has $1$-dimensional cohomology and the functor of tensoring with it is an equivalence. 
   
   Next, the complex \eqref{eq:e5(F_E)}, denote it $C$,  has the form
 \[
\xymatrix@R=0.1pt@C=70pt{
&    \Oc [2n] &
\\
\Oc & \oplus &
\\
\oplus & \Oc & 
\\
\quad \omega  \quad \ar[r]  & \quad \quad \oplus  \quad\quad  \ar[r] & 
\quad \Oc [2n].
\\
\oplus & \Oc [2n]&
\\
\Oc [2n] & \oplus & 
\\
&  \omega & 
}
\]
We see that the summands in $C$ split into matching groups: two copies of $\Oc$, two copies 
of $\omega$ and four copies of
of $\Oc[2n]$. Looking at the differentials more closely, we  see that
  the sum of $4$ copies
of $\Oc[2n]$ is a subcomplex in $C$, denote it $C_1$. These copies correpond to the summands in \eqref{eq:E5F-general}
beginning with $F^{(4)}$, and the complex itself consists of applying two commuting units:
for $(F^{(2)}, F^{(3)})$ and $(F, F^*)$. The subcomplex $C_1$ is therefore exact.
Next, let $C_2$ be the sum of $C_1$ and the two summands $\omega$. This is also a subcomplex,
as the new summands correspond to summands in \eqref{eq:E5F-general}
 beginning with $F^{(2)}$. The quotient $C_2/C_1$ is the $2$-term complex 
 $\{\omega\buildrel \Id\over \to\omega\}$ and so is exact. Finally, the quotient of  $C$ 
 by $C_2$ is the $2$-term complex  $\{\Oc \buildrel \Id\over \to\Oc \}$ and is exact as well.
 So $C$ is exact and the proposition is proved. \qed
 
 \paragraph{Further remarks.}
 
 (a) For $n$ even, an $n$-dimensional Enriques manifold can be seen as a holomorphic analog
 of the real projective space $\RR P^n$ which is a non-orientable $C^\oo$-manifold, quotient of
 the sphere $S^n$ by an orientation reversing involution. In fact, we have a topological analog of Proposition 
 \ref{prop:enriques} where $\B$  is the  derived category of complexes of sheaves of $\k$-vector spaces
  on $\RR P^n$ with locally constant cohomology. Then
   the constant sheaf $\ul \k_{\RR P^n}$ is a $6$-spherical object in $\B$. The proof is analogous to that
   of Proposition
 \ref{prop:enriques}.

 \vskip .2cm
 
 (b)
 Let us call an {\em $m$-fold Enriques manifold}  a quotient $\wt X/\ZZ_m$,
 where $\wt X$ is an $n$-dimensional Calabi-Yau manifold and $\ZZ_m$ is a cyclic group of
 order $m$ acting on $\wt X$ freely and such that the induced action on $H^0(\wt X, \Omega_{\wr X}^n)$
 is given by a primitive $m$-th root of $1$. One can  generalize Proposition
 \ref{prop:enriques} by showing that for an $m$-fold Enriques manifold $X$ the sheaf
 $\Oc_X$ is a $2(m+1)$-spherical object. The proof is similar but more combinatorially involved
 and we do not give it here.

 
 \section{Discussion and further directions}\label{sec:fur-dir}
 
 \subsection{Interpretation as categorified Stokes data}\label{subsec:inter-cat-stokes}
 
 An analogy between $N$-periodic SODs and Stokes data for irregular local systems was
 pointed out by Kuwagaki \cite{kuwagaki}. We recall \cite{deligne} that an  irregular local system
  on $\CC$ near
 $\oo$ has, first of all, the {\em exponential datum} which gives a local system of  finite partially ordered
 sets on $S^1_\oo$,  the circle of directions. Following Kontsevich,
 such  a datum is represented geometrically by a ``Lissajous figure"
 as in Fig. \ref {fig:lissajous} so that the intersection of the figure with the ray in any direction $\zeta$ gives
 the poset corresponding to $\zeta$. 
 
 \begin{figure}[h]
 \centering
 \begin{tikzpicture}[scale=5]
 \draw[->, dotted, line width = 1] (0,0) -- (0.5,0.5); 
 \draw[domain=0:6.28,samples=200,smooth] plot (canvas polar cs:angle=\x r,radius={10+sin(5*\x r)}); 
\draw[domain=0:6.28,samples=200,smooth] plot (canvas polar cs:angle=\x r,radius={10-sin(5*\x r)});       
\node at (0.5, 0.4){$\zeta$}; 
\node at (0,0) {\tiny$\bullet$}; 
\end{tikzpicture} 
\caption{A local system of $2$-element posets on $S^1$.}\label{fig:lissajous}
\end{figure}
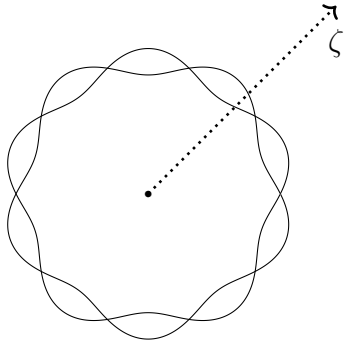
\noindent Second, the local system of solutions of the system near $\oo$ is filtered by the local system of
 posets in the exponential datum, according to the rate of exponential growth of solutions. 
 This is the {\em Stokes filtration}. 
 The case of $2$-element posets depicted in Fig. \ref{fig:lissajous}
 corresponds to the special situation when there is only one type
 of sub-dominant growth, and the subspace  formed by solutions of such growth switches  (Stokes
 phenomenon)
 at finitely many {\em Stokes directions} corresponding to the intersections of the branches of the
 Lissajous figure. Such special situation is always present if the system is of rank $2$, when the subdominant
 subspace is $1$-dimensional. 
 
 Thus $N$-periodic SODs  
 can be seen as categorical analogs of irregular local systems 
 (``irregular schobers'') with  only one type of sub-dominant growth and with $N$ Stokes directions.  
 An example of such system is given by the complex Schr\"odinger equation
 \be\label{eq:schro}
 \psi''(z) = U(z)\,  \psi(z), \quad  U(z) \in \CC[z], \quad \deg(U) = N-2. 
 \ee
 In particular, $3$-periodic SODs can be seen as categorical analogs of the Airy equation
 \[
 \psi''(z) = z\, \psi(z)
 \]
 which has $3$ Stokes directions,
 while $4$-periodic SODs (coming from usual spherical functors)  correspond in this analogy
 to the complex harmonic oscillator (the Weber equation)
 \[
 \psi''(z) = (z^2+a)\,  \psi(z) 
 \]
  which has $4$ Stokes directions. 
  
  \vskip .2cm
  
  Further, the moduli space of Stokes data for Schr\"odinger equations \eqref{eq:schro} with even $N$
  described by Sibuya \cite{sibuya},  was interpreted by Boalch \cite{boalch} as a multiplicative
  symplectic quotient in which the group valued moment map is given by an appropriate
  continuant. This is in complete agreement with the above analogy and  with our use
  of categorified continuants to describe $N$-periodic SODs. 
  
  More recently, the noncommutative Poisson formalism of Fairon and Fernandez \cite{fairon}
  involves, in a universal way, continuants whose arguments are operators $x_1: V_0\to V_1$,
  $x_2: V_1\to V_0$ etc. with alternating sources and targets, just like our chains of adjoints. 
  It corresponds to $2$-level Stokes filtrations with graded components of arbitrary dimension. 
  In fact, the approach of \cite{fairon} gives a construction of the multiplicative quiver variety associated
  to any graph $Q$. Our case at hand corresponds to $Q = Q_N$ being the intersection
  graph of the Lissajous Fig. \ref{fig:lissajous}, e.g., for $N$ even $Q_N$ has two vertices and
  $N$ edges joining them. Thus our approach can be seen as a categorification of that of 
  \cite{fairon} for $Q=Q_N$. Generalizations to other $Q$ seem possible and very interesting. 
  
  \vskip .2cm
  
  In a seemingly different direction, continuants and their level varieties appear in  recent work
  of Etingof, Frenkel and Kazdhan \cite[\S 4.8]{etingof} on the analytic Langlands correspondence for
  $PGL_2$ over  archimedean fields. In that approach, continuants describe
  moduli space of certain {\em regular}  differential equations (balanced opers) on $\CC$ with singularities
  on the real line. Our categorification of continuants suggest that the analytic
  Langlands correspondence itself may allow one more level of categorical lift
  involving perverse schobers to replace differential systems on the eigenvalues of the Hecke
  operators considered in \cite{etingof}. 
  
    
    \subsection{More general lax additive $(\oo,2)$-categories}\label{subsec:more-gen-lax}
    
    The natural context for categorified continuants considered in this paper is that of arbitrary
    lax additive $(\oo, 2)$-categories in the sense of \cite{CDW}.   The situation considered here corresponds to
 the particular $(\oo, 2)$-category $\SSt$ of stable $\oo$-categories and their exact functors. 
 Indeed, in any   lax additive $(\oo, 2)$-category $\DD$ we can speak about (iterated) adjoints of $1$-morphisms and  form totalizations of cubes of $1$-morphisms.  We can also form analogs of semi-orthogonal decompositions which are (op)lax (co)limits $a\oplus_f b$ for any two objects $a,b\in \DD$ and any
 $1$-morphism $f: a\to b$. All our statements should extend to this context more or less straightforwardly. 
 
 \vskip .2cm

 In fact, one can consider more general, not necessarily lax additive $(\oo, 2)$-categories 
 $\DD$ in which each $\Hom(a,b)$
 is stable. In this case we can totalize  Fibonacci cubes of iterated
 adjoints (if they exist) but not necessarily form $a\oplus_f b$. But the concept of $N$-spherical
 $1$-morphisms makes sense in this more general situation.

\vskip .2cm

In particular, one can consider monoidal stable $\oo$-categories $(\M, \otimes)$. 
which can be seen as $(\oo,2)$-categories with one object. In this case the $(\oo,2)$-categorical left and right
adjoints of $1$-morphisms are the same as the left and right duals $A^*, {^* A}$ of objects of $M$, and
one can still form the Fibonacci cubes and their totalizations.  One can then speak about
objects of $\M$ which are $N$-spherical as $1$-morphisms of the corresponding $(\oo, 2)$-category
(they are different from $N$-spherical objects in the sense of \S \ref{subsec:N-spher-obj}).

Even more specifically, one
can start with an abelian, not necessarily symmetric, monoidal category $\Mc$ and let $\M = D^b(\Mc)$
be its derived category. This is the context studied by Coulembier and Etingof \cite{coulembier}
 who found several
interesting examples of objects of $\M$ which are $N$-spherical as $1$-morphisms of the corresponding
$(\oo, 2)$-category.  

\vskip .2cm

Note that any monoidal stable $\oo$-category $(\M, \otimes)$ can be embedded into
$\St$ via the regular representation. That is, we associate to an object $A\in \M$ the functor
$L_A: \M\to \M$ sending $B\mapsto A\otimes B$. Then the right adjoint of $L_A$ is $L_{A^*}$,
where $A^*$ is the right dual. So one can use the techniques of the present paper (such as forming
semi-orthogonal decompositions) in this
more general context. 

\begin{rem}
It would be interesting to make sense of 
  the ``universal $N$-spherical $1$-morphism", 
i.e.,  to construct and study the  lax additive $(\oo, 2)$-category $\SSph_N$
with two objects $a_0, a_1$,  whose higher morphisms are generated by a formal 
$1$-morphism $f: a_0\to a_1$, 
its  adjoints $f^*, \cdots, f^{N-1}$
(and the corresponding adjunction data) such that $\EE_{N-1}(f)$ and $\EE^{N-1}(f)$ vanish
and no further data are included or conditions are imposed. 
Then  $N$-spherical functors would be the same as $(\oo, 2)$-functors $\SSph_N\to\SSt$. 
The $(\oo,2)$-category $\SSt_N$ would then be a categorification of the fission algebra
$\Fc^q(Q_N)$ associated by Boalch \cite{boalch2} \cite[Rem. 17]{boalch} to the graph $Q_N$
from \S \ref{subsec:inter-cat-stokes}. 
\end{rem}


\subsection{Possible relations to the  cluster formalism}\label{subsec:cluster}

Periodicity of categorical orthogonals brings to mind another periodocity phenomenon
whose natural explanation is provided by the theory of cluster algebras: that of Zamolodchikov Y-systems  \cite{fomin-zelevinsky}. 
So it is tempting to look for closer relationships of our approach with  cluster theory. 

In fact, continuants do make frequent appearances in that theory. Thus, they provide cluster
coordinates in some basic examples \cite[Ex. 5.3.10]{FWZ}. In particular, they enjoy various
positivity properties \cite{morier-genoud}. Further, in the correspondence between solutions of
Y-systems and T-systems  \cite{kuniba} the latter play the role of continuants and the former that
of continued fractions, i.e., ratios of continuants. This analogy becomes essentially an identity
for systems associated to the (affine) $sl_2$, when the general solution of the $T$-system is found
in terms of continuants \cite{etingof}. 
 
 The result of \cite{etingof} can be illustrated at a slightly more conceptual level.
 Namely, it is classical that in the Grothendieck ring of representations of $sl_2(\k)$ for a field $\k$
 of characteristic $0$ we have the identity $[V_N] = U_N([V_1])$ where where $[V_N]$ is the class of the representation
 $V_N= S^N(V_1)$, $V_1=\k^2$ and $U_N$ is the Chebyshev polynomial of the second kind, see Proposition 
 \ref{prop:chebyshev}. Passing  to quantum affine algebras, which provide the proper context for T-systems,
 this identity upgrades to one involving a full-fledged continuant \cite[Ex. 3.4]{hernandez}.
 
 This suggests that our categorification of continuants can play a role in some categorification of the cluster formalism.
 There are two levels of such categorification known at present: the additive one \cite {keller, reiten} and
 the monoidal one \cite{hernandez}. Our approach which does not a priori use quantum groups or quivers
 may be hinting at a still further, more abstract level.


\vskip 1cm

\small{
 T.D.:  Universit\"at Hamburg,
Fachbereich Mathematik,
Bundesstrasse 55,
20146 Hamburg, Germany. Email: 
{\tt tobias.dyckerhoff@uni-hamburg.de}

\smallskip

M.K.: Kavli IPMU, 5-1-5 Kashiwanoha, Kashiwa, Chiba, 277-8583 Japan. Email: 
\hfil\break
{\tt mikhail.kapranov@protonmail.com}

\smallskip

 V.S.: Institut de Math\'ematiques de Toulouse, Universit\'e Paul Sabatier, 118 route de Narbonne, 
31062 Toulouse, France and Kavli IPMU, 5-1-5 Kashiwanoha, Kashiwa, Chiba, 277-8583 Japan.
Email: 
 {\tt schechtman@math.ups-tlse.fr }

   }          


\begin{thebibliography}{IKM16}

{\small

 

\bibitem{AL17}
  R. Anno,  T.  Logvinenko.
 Spherical DG-functors. {\em J. Eur. Math. Soc.}  {\bf 19} (2017) 2577-2656. 
 
 \bibitem{benjamin} A. Benjamin, D. Walton. Counting on Chebyshev polynomials.
 {\em Mathematics Magazine} {\bf 82}  (2009) 117-126.
 
 \bibitem{bezr-kapr} R. Bezrukavnikov, M. Kapranov. Microlocal sheaves and quiver varieties. 
 {\em Ann. Fac. Sci. de Toulouse. Math.} {\bf 25} (2016) 473-516. 
 
  \bibitem{boalch2} P. Boalch. Global Weyl groups and a new theory of multiplicative quiver varieties.
 {\em Geom. and Topology} {\bf 19} (2015)3467-3536. 
 arXiv 1307.1033.
 
 \bibitem{boalch} P. Boalch. Wild character varieties, points on the Riemann sphere
 and Calabi's examples. in: ``Representation Theory, Special Functions and Painlev\'e Equations''
 (H. Konno et al. Eds.) 
 ({\em Adv. Stud. Pure Math} {\bf 76}) p. 67-94, Math. Soc. Japan 2018.  arXiv:1501.00930. 
 

 
 \bibitem{bondal-symplectic} A. I. Bondal. 
A symplectic groupoid of triangular bilinear forms and the braid group.   
 {\em Izv. Math. } 
 {\bf 68} (2004) 659-708. 

\bibitem{BK:SOD} A. I. Bondal, M. M. Kapranov. Representable functors, Serre functors and
mutations.
 {\em Math. USSR-Izv.}
{\bf 35} (1990) 519-541. 

 \bibitem{BK:enh} A. I. Bondal, M. M. Kapranov. Enhanced triangulated
categories. {\em Math. of the USSR-Sb.} {\bf 70} (1991)  93-107. 
  
 
 
 \bibitem{bondal-orlov} A. Bondal, D. Orlov. Semi-orthogonal decompositions
 for algebraic varieties. arXiv:alg-geom/9506012. 
 
 \bibitem{borini} B. Borini. I Continuanti. G. Medri Publ., Forli, 1900. 
 
 \bibitem{CDW} M. Christ, T. Dyckerhoff, T. Walde. Complexes of stable $\oo$-categories.
 arXiv:2301.02606. 
 
 \bibitem{cohn} L. Cohn. Differential graded categories are $k$-linear stable infinity-categories.
 arXiv:1308.2587.  
 
 \bibitem {coulembier} K. Coulembier, P. Etingof. Spherical objects in tensor categories,  in preparation. 
 
\bibitem{deligne} P. Deligne, B. Malgrange, J.-P. Ramis (Eds.) Singularit\'es Irr\'egulières: Correspondance
et Documents. Soc. Math. France, 2007.
 
 \bibitem{dkss-spher} T. Dyckerhoff, M. Kapranov, V. Schechtman, Y. Soibelman.
 Spherical adjunctions of stable $\oo$-categories and the relative S-construction.  arXiv:2106.02873 . 
 
 \bibitem{etingof} P. Etingof, E. Frenkel, D. Kazhdan. 
 Analytic Langlands correspondence for $PGL_2$ on $\PP^1$
 with parabolic structures over local fields. aXiv:2106.05243. 
 
 \bibitem{euler} L. Euler. Specimen algorithmi singularis. {\em Novi Commentarii academiae scientiarum
 Petropolitanae} {\bf 9} (1764) 53-69 (item E281 at {\tt eulerarchive.maa.org.}). 
 
 \bibitem{faonte} G. Faonte. Simplicial nerve of an $A_\oo$-category. 
  	{\em Theory and Appl. of Categories} {\bf  32} (2017) No. 2, pp.  31-52. 
 
 \bibitem{fairon} M. Fairon, D. Fernandez. Euler continuants in noncommutative
 quasi-Poisson geometry. arXiv:2105.04858. 
 
 \bibitem{fomin-zelevinsky} S. Fomin, A. Zelevinsky. Y-systems and generalized associahedra.
 {\em Ann. Math.} {\bf 158} (2003) 977-1018. 
 
 \bibitem{FWZ}  S. Fomin, L. Williams, A. Zelevinsky. Introduction to Cluster Algebras, Chapters 4-5.
 arXiv:1707.07190. 
 
 
 \bibitem{GH} D. Gepner, R. Haugseng.  Enriched $\oo$-categories via non-symmetric $\oo$-operads.
 {\em Adv. Math.} {\bf 279} (2015) 575-716. 
 
 \bibitem{goodwillieII} T.  G. Goodwillie. 
 Calculus II, Analytic functors. {\em  K-Theory } {\bf 5}  (1991/92) 295-332.
 
  
 \bibitem
{halpern-shipman}
D.~Halpern-Leistner and I.~Shipman. Autoequivalences of derived
  categories via geometric invariant theory. {\em  Adv.  Math.}
  \textbf{303} (2016)  1264--1299.
  
  \bibitem{hernandez} D. Hernandez, B. Leclerc. Cluster algebras and quantum affine 
  algebras. {\em Duke Math. J. } {\bf 154} (2010) 265-341. 
  
  \bibitem{hoggatt} V. E. Hoggatt, M. Bicknell. Roots of Fibonacci polynomials.
  {\em Fibonacci Quarterly} {\bf 11} (1973) 271-274. 

   
 \bibitem{hsu} W.-J. Hsu. Fibonacci cubes -- a new interconnection topology.
 {\em IEEE Transactions on Parallel and Distributed Systems} {\bf 4} (1993)
 3-13. 
 
  \bibitem {kapranov-schechtman:schobers} M. Kapranov, V. Schechtman. 
Perverse Schobers. arXiv:1411.2772.  

\bibitem{keller} B. Keller.  Categorification of acyclic cluster algebras: an introduction.
arXiv:0801.3103.
 

\bibitem{KS} M. Kelly,  R. Street, Review of the elements of 2-categories, 
Sydney Category Seminar 1972/1973, in G. M. Kelly (ed.) Lecture Notes in Mathematics {\bf 420}, Springer 1974.  


\bibitem {khrushchev} S. Khrushchev. Orthogonal Polynomials and Continued Fractions
(Encyclopedia of Mathematics and its Applications {\bf 122}). Cambridge Univ. Press,
2008

\bibitem{knuth} D. E. Knuth. The Art of Computer Programming.
Addison-Wesley, 1998. 

\bibitem{kuniba} A. Kuniba, T. Nakanishi, J. Suzuki. T-systems and Y-systems in integrable systems.
{\em J. Phys. A:  Math. Theor.}  {\bf 44} (2011) 103001. 

\bibitem{kuwagaki} T. Kuwagaki. Categorification of Legendrian knots. arXiv:1902.04269. 

\bibitem{kuznetsov} A. Kuznetsov. Calabi-Yau and fractional Calabi-Yau categories. 
 {\em J. Reine Angew. Math.}  {\bf 753}  (2019) 239-267. 
 
 \bibitem{kuznetsov-lunts} A. Kuznetsov, V. Lunts. Categorical resolutions of irrational singularities.
  {\em Int. Math. Res.
Notices}   {\bf 2015}  (2015) 4536-4625.
 
 
    
   \bibitem
{lurie:htt}
J.~Lurie, Higher Topos Theory, Annals of Mathematics Studies, vol. 170,
  Princeton University Press, Princeton, NJ, 2009.
  
  \bibitem
  {lurie:ha} J.~Lurie, Higher Algebra. Available at: 
  $<$http://people.math.harvard.edu/~lurie/papers/HA.pdf$>$
  
  
  \bibitem{maclane} S. MacLane. Categories for the Working Mathematician.
  Springer-Verlag, 1978. 
  
  \bibitem{morier-genoud} S. Morier-Genoud, V. Ovsienko. $q$-deformed rationals and $q$-continued
  fractions. {\em Forum Math (SIGMA)} {\bf 8} (2020), paper e13. 
  
  \bibitem{perron} O. Perron. Die Lehre von Kettenbr\"uchen, Band I. 
  Teubner, Stuttgart, 1954. 
  
  \bibitem{reiten} I. Reiten. Cluster categories. Proc. ICM Hyderabad 2010,
  arXiv:1012.4949.
  
  
\bibitem{segal} E. Segal. All autoequivalences are spherical twists.
{\em Int.  Math. Res. Notices}  {\bf 2018} 3137-3154. 
 
\bibitem{sibuya} Y. Sibuya. Global Theory of a Second Order Linear Ordinary
Differential Equation with a Polynomial Coefficient. North-Holland Publ.,  Amsterdam, 1975. 
 
  
  }          

\end{thebibliography}
\end{document}